\documentclass[letter,10pt]{amsart}
\usepackage{amssymb,amsmath,amsfonts}
\usepackage{mathrsfs}
\usepackage[left=1 in, right=1 in,top=1 in, bottom=1 in]{geometry}
\usepackage{mathenv}

\newtheorem{theorem}{Theorem}[section]
\newtheorem{lemma}[theorem]{Lemma}

\newtheorem{corollary}[theorem]{Corollary}

\newtheorem{remark}[theorem]{Remark}
\newtheorem{definition}[theorem]{Definition}

\makeatletter
\renewcommand \theequation {%
\ifnum \c@section>\z@ \@arabic\c@section.%
\fi\@arabic\c@equation} \@addtoreset{equation}{section}
\makeatother

\providecommand{\ud}[1]{\mathrm{d}{#1}}
\providecommand{\abs}[1]{\left\vert#1\right\vert}
\providecommand{\nm}[1]{\left\Vert#1\right\Vert}

\providecommand{\br}[1]{\langle #1 \rangle}
\providecommand{\tm}[2]{\left\Vert#1\right\Vert_{L^2(#2)}}
\providecommand{\im}[2]{\left\Vert#1\right\Vert_{L^{\infty}(#2)}}

\providecommand{\lnnm}[1]{{\left\Vert#1\right\Vert}_{L^{\infty}L^{\infty}}}
\providecommand{\lnm}[1]{\left\Vert#1\right\Vert_{L^{\infty}}}
\providecommand{\tnm}[1]{\left\Vert#1\right\Vert_{L^{2}}}
\providecommand{\tnnm}[1]{{\left\Vert#1\right\Vert}_{L^{2}L^2}}

\def\p{\partial}

\def\half{\frac{1}{2}}
\def\rt{\rightarrow}

\def\no{\nonumber}
\def\ue{\mathrm{e}}

\def\u{U^{\e}}
\def\ub{\mathscr{U}^{\e}}
\def\bu{\bar U^{\e}}
\def\bub{\bar{\mathscr{U}}^{\e}}

\def\e{\epsilon}
\def\s{\mathcal{S}}
\def\vx{\vec x}
\def\vw{\vec w}

\def\nx{\nabla_{x}}
\def\vn{\vec n}

\def\px{\p_{\eta}}

\def\l{\lambda}

\def\ll{\mathcal{L}}

\def\k{\kappa}
\def\q{Q}
\def\qb{\mathscr{Q}}

\def\v{\mathscr{V}}
\def\w{\mathscr{W}}
\def\d{\delta}

\def\pp{\mathcal{P}}
\def\vn{\vec\nu}
\def\xc{X_{cl}}
\def\wc{W_{cl}}

\def\t{\mathcal{T}}
\def\k{\mathcal{K}}

\def\a{\mathscr{A}}
\def\b{\mathscr{B}}
\def\rk{R_{\kappa}}

\def\z{\mathscr{Z}}
\def\ea{\eta_{\ast}}
\def\pa{\phi_{\ast}}

\def\id{{\bf{1}}}

\begin{document}
\title{Geometric Correction in Diffusive Limit of Neutron Transport Equation in 2D Convex Domains}

\author[Y. Guo]{Yan Guo}
\address[Y. Guo]{\newline\indent
Division of Applied Mathematics, Brown University,
\newline\indent Providence, RI 02912, USA }
\email{Yan\_Guo@brown.edu}

\author[L. Wu]{Lei Wu}
\address[L. Wu]{
   \newline\indent Department of Mathematical Sciences, Carnegie Mellon University
\newline\indent Pittsburgh, PA 15213, USA}
\email{lwu2@andrew.cmu.edu}


\begin{abstract}
Consider the steady neutron transport equation with diffusive boundary condition. In \cite{AA003} and \cite{AA006}, it was discovered that geometric correction is necessary for the Milne problem of Knudsen-layer construction in a disk or annulus. In this paper, we establish diffusive limit for a 2D convex domain. Our contribution relies on novel $W^{1,\infty}$ estimates for the Milne problem with geometric correction in the presence of a convex domain, as well as an $L^{2m}-L^{\infty}$ framework which yields stronger remainder estimates.
\end{abstract}

\maketitle

\tableofcontents

\newpage


\pagestyle{myheadings} \thispagestyle{plain} \markboth{LEI WU}{DIFFUSIVE LIMIT IN GENERAL DOMAIN}

\section{Introduction}

\subsection{Problem Formulation}

We consider the steady neutron transport equation in a
two-dimensional convex domain with diffusive boundary. In the space
domain $\vx=(x_1,x_2)\in\Omega$ where $\p\Omega\in C^2$ and the velocity domain
$\vw=(w_1,w_2)\in\s^1$, the neutron density $u^{\e}(\vx,\vw)$
satisfies
\begin{eqnarray}\label{transport}
\left\{
\begin{array}{rcl}\displaystyle
\e \vw\cdot\nabla_x u^{\e}+u^{\e}-\bar u^{\e}&=&0\ \ \text{in}\ \ \Omega,\\
u^{\e}(\vx_0,\vw)&=&\pp[u^{\e}](\vx_0)+\e g(\vx_0,\vw)\ \ \text{for}\
\ \vw\cdot\vn<0\ \ \text{and}\ \ \vx_0\in\p\Omega,
\end{array}
\right.
\end{eqnarray}
where
\begin{eqnarray}\label{average}
\bar u^{\e}(\vx)=\frac{1}{2\pi}\int_{\s^1}u^{\e}(\vx,\vw)\ud{\vw},
\end{eqnarray}
\begin{eqnarray}\label{diffusive}
\pp[u^{\e}](\vx_0)=\frac{1}{2}\int_{\vw\cdot\vec
n>0}u^{\e}(\vx_0,\vw)(\vw\cdot\vn)\ud{\vw},
\end{eqnarray}
$\vn$ is the outward unit normal vector, with the Knudsen number $0<\e<<1$. Also, $u^{\e}$ satisfies the normalization condition
\begin{eqnarray}\label{normalization}
\int_{\Omega\times\s^1}u^{\e}(\vx,\vw)\ud{\vw}\ud{\vx}=0,
\end{eqnarray}
and $g$ satisfies the compatibility condition
\begin{eqnarray}\label{compatibility}
\int_{\p\Omega}\int_{\vw\cdot\vn<0}g(\vx_0,\vw)(\vw\cdot\vn)\ud{\vw}\ud{\vx_0}=0.
\end{eqnarray}
We intend to study the behavior of $u^{\e}$ as $\e\rt0$.

Based on the flow direction, we can divide the boundary $\Gamma=\{(\vx,\vw): \vx\in\p\Omega\}$ into
the in-flow boundary $\Gamma^-$, the out-flow boundary $\Gamma^+$
and the grazing set $\Gamma^0$ as
\begin{eqnarray}
\Gamma^{-}&=&\{(\vx,\vw): \vx\in\p\Omega, \vw\cdot\vn<0\}\\
\Gamma^{+}&=&\{(\vx,\vw): \vx\in\p\Omega, \vw\cdot\vn>0\}\\
\Gamma^{0}&=&\{(\vx,\vw): \vx\in\p\Omega, \vw\cdot\vn=0\}
\end{eqnarray}
It is easy to see $\Gamma=\Gamma^+\cup\Gamma^-\cup\Gamma^0$.
Hence, the boundary condition is only given for $\Gamma^{-}$.

\subsection{Main Result}

\begin{theorem}\label{main}
Assume $g(\vx_0,\vw)\in C^2(\Gamma^-)$ satisfying (\ref{compatibility}). Then for the steady neutron
transport equation (\ref{transport}), there exists a unique solution
$u^{\e}(\vx,\vw)\in L^{\infty}(\Omega\times\s^1)$ satisfying (\ref{normalization}). Moreover, for any $0<\d<<1$, the solution obeys the estimate
\begin{eqnarray}
\im{u^{\e}(\vx,\vw)-\u_0(\vx)}{\Omega\times\s^1}\leq C(\d,\Omega)\e^{1-\d},
\end{eqnarray}
where $\u_0(\vx)$ satisfies
\begin{eqnarray}
\left\{
\begin{array}{rcl}
\Delta_x\u_0&=&0\ \ \text{in}\
\ \Omega,\\\rule{0ex}{2em}\dfrac{\p\u_0}{\p\vec
\nu}&=&\dfrac{1}{\pi}\displaystyle
\int_{\vw\cdot\vn<0}g(\vx,\vw)(\vw\cdot\vn)\ud{\vw}\ \ \text{on}\ \
\p\Omega,\\\rule{0ex}{1em}
\displaystyle\int_{\Omega}\u_0(\vx)\ud{\vx}&=&0,
\end{array}
\right.
\end{eqnarray}
in which $C(\d,\Omega)>0$ denotes a constant that depends on $\d$ and $\Omega$.
\end{theorem}

\subsection{Background and Methods}

Diffusive limit, or more general hydrodynamic limit, plays a key role in connecting kinetic theory and fluid mechanics. Since 1960s, this type of problems have been extensively studied in many different settings: steady or unsteady, linear or nonlinear, strong solution or weak solution, etc. We refer to the references
\cite{Larsen1974=}, \cite{Larsen1974}, \cite{Larsen1975}, \cite{Larsen1977}, \cite{Larsen.D'Arruda1976}, \cite{Larsen.Habetler1973}, \cite{Larsen.Keller1974}, \cite{Larsen.Zweifel1974}, \cite{Larsen.Zweifel1976}, \cite{Li.Lu.Sun2015=}, \cite{Li.Lu.Sun2015}, \cite{Li.Lu.Sun2015==} for more details. Among all these variations, one of the simplest but most important models - steady neutron transport equation with one-speed velocity in bounded domains, where the boundary layer effect shows up, has long been believed to be satisfactorily solved since Bensoussan, Lions and Papanicolaou published their remarkable paper \cite{Bensoussan.Lions.Papanicolaou1979} in 1979.

Unfortunately, their results are shown to be false due to lack of regularity for the classical Milne problem in \cite{AA003} and \cite{AA006}. A new approach with geometric correction to the Milne problem has been developed to ensure regularity in the cases of disk and annulus in \cite{AA003} and \cite{AA006}. However, this new method fails to treat more general domains.

Consider the boundary layer expansion with geometric correction
\begin{eqnarray}
\ub(\eta,\tau,\phi)=\ub_0(\eta,\tau,\phi)+\e\ub_1(\eta,\tau,\phi),
\end{eqnarray}
where $\eta$ denotes the rescaled normal variable, $\tau$ the tangential variable and $\phi$ the velocity variable defined in (\ref{substitution 1}), (\ref{substitution 2}), (\ref{substitution 3}), and (\ref{substitution 4}). Thanks to the diffusive boundary condition, $\ub_0=0$. As \cite{AA003} stated, the boundary layer must formally satisfy
\begin{eqnarray}
\sin\phi\frac{\p \ub_1}{\p\eta}+\frac{\e}{R_{\kappa}-\e\eta}\cos\phi\frac{\p
\ub_1}{\p\phi}+\ub_1-\bub_1&=&0,
\end{eqnarray}
where $R_{\kappa}$ is the radius of curvature at boundary.

In the absence of the geometric correction $\dfrac{\e}{R_{\kappa}-\e\eta}\cos\phi\dfrac{\p\ub_1}{\p\phi}$ as in \cite{Bensoussan.Lions.Papanicolaou1979}, the key tangential derivative $\dfrac{\p\ub_1}{\p\tau}$ is not bounded for such a classical Milne problem. Therefore, the expansion breaks down. In the case when $R_{\kappa}$ is constant, as in \cite{AA003} and \cite{AA006}, $\dfrac{\p\ub_1}{\p\tau}$ is smooth, since the tangential derivative commutes with the equation. On the other hand, when $R_{\kappa}$ is a function of $\tau$, then $\dfrac{\p\ub_1}{\p\tau}$ relates to the normal derivative $\dfrac{\p\ub_1}{\p\eta}$, whose boundedness had remained open.

Our main contribution is to show $\dfrac{\p\ub_1}{\p\tau}$ is bounded when $R_{\kappa}$ is not a constant for a general convex domain. Our proof is intricate and lies on the weighted $L^{\infty}$ estimates for the normal derivative. We use careful analysis along the characteristic curves in the presence of non-local averaging $\bub_1$ over $\phi$. The convexity and invariant kinetic distance $\zeta(\eta,\tau,\phi)$ defined in (\ref{weight function}), plays the crucial role. Our paper marks an important first step towards the study of diffusive expansions of neutron transport equations and other kinetic equations with boundary layer correction.

Moreover, we have to improve the remainder estimate to avoid higher-order expansion. This is done by a new $L^{2m}$-$L^{\infty}$ framework. The main idea is to introduce a special test function in weak formulation to treat kernel and non-kernel parts separately, and further improve the $L^{\infty}$ estimate by a modified double Duhamel's principle. The proof relies on a delicate analysis using interpolation and Young's inequality.

Applying these two new techniques, we successfully obtain the diffusive limit of neutron transport equation in a convex domain with diffusive boundary.

\subsection{Notation and Structure}

Throughout this paper, unless specified, $C>0$ denotes a universal constant which does not depend on the data and
can change from one inequality to another.
When we write $C(z)$, it means a certain positive constant depending
on the quantity $z$.

Our paper is organized as follows: in Section 2, we present the asymptotic analysis of the equation (\ref{transport});  in Section 3, we prove the weighted $L^{\infty}$ estimates of derivatives in $\e$-Milne problem with geometric correction; in Section 4, we prove the improved $L^{\infty}$ estimate of remainder equation; finally, in
Section 5, we prove the diffusive limit, i.e. Theorem \ref{main}.

\section{Asymptotic Analysis}

\subsection{Interior Expansion}

We define the interior expansion as follows:
\begin{eqnarray}\label{interior expansion}
\u(\vx,\vw)\sim\u_0(\vx,\vw)+\e\u_1(\vx,\vw)+\e^2\u_2(\vx,\vw),
\end{eqnarray}
where $\u_k$ can be determined by comparing the order of $\e$ by
plugging (\ref{interior expansion}) into the equation
(\ref{transport}). Thus we have
\begin{eqnarray}
\u_0-\bu_0&=&0,\label{expansion temp 1}\\
\u_1-\bu_1&=&-\vw\cdot\nx\u_0,\label{expansion temp 2}\\
\u_2-\bu_2&=&-\vw\cdot\nx\u_1.\label{expansion temp 3}
\end{eqnarray}
Plugging (\ref{expansion temp 1}) into (\ref{expansion temp 2}),
we obtain
\begin{eqnarray}
\u_1=\bu_1-\vw\cdot\nx\bu_0.\label{expansion temp 4}
\end{eqnarray}
Plugging (\ref{expansion temp 4}) into (\ref{expansion temp 3}),
we get
\begin{eqnarray}\label{expansion temp 5}
\u_2-\bu_2=-\vw\cdot\nx(\bu_1-\vw\cdot\nx\bu_0)=-\vw\cdot\nx\bu_1+\vw^2\Delta_x\bu_0+2w_1w_2\p_{x_1x_2}\bu_0.
\end{eqnarray}
Integrating (\ref{expansion temp 5}) over $\vw\in\s^1$, we achieve
the final form
\begin{eqnarray}
\Delta_x\bu_0=0.
\end{eqnarray}
which further implies $\u_0(\vx,\vw)$ satisfies the equation
\begin{eqnarray}\label{interior 1}
\left\{ \begin{array}{rcl} \u_0&=&\bu_0,\\
\Delta_x\bu_0&=&0.
\end{array}
\right.
\end{eqnarray}
In a similar fashion, for $k=1,2$, $\u_k$ satisfies
\begin{eqnarray}\label{interior 2}
\left\{ \begin{array}{rcl} \u_k&=&\bu_k-\vw\cdot\nx\u_{k-1},\\
\Delta_x\bu_k&=&\displaystyle-\int_{\s^1}\vw\cdot\nx\u_{k-1}\ud{\vw}.\end{array}
\right.
\end{eqnarray}
It is easy to see $\bu_k$ satisfies an elliptic equation. However, the boundary condition of $\bu_k$ is unknown at this stage, since generally $\u_k$ does not necessarily satisfy the diffusive boundary condition of (\ref{transport}). Therefore, we have to resort to boundary layer.

\subsection{Local Coordinate System}

Basically, we use two types of coordinate systems: Cartesian coordinate
system for interior solution, which is stated above, and local coordinate system in a neighborhood of the boundary for boundary layer.

Assume the Cartesian coordinate system is $\vx=(x_1,x_2)$. Using polar coordinates system $(r,\theta)\in[0,\infty)\times[-\pi,\pi)$ and choosing pole in $\Omega$, we assume $\p\Omega$ is
\begin{eqnarray}
\left\{
\begin{array}{rcl}
x_1&=&r(\theta)\cos\theta,\\
x_2&=&r(\theta)\sin\theta,
\end{array}
\right.
\end{eqnarray}
where $r(\theta)>0$ is a given function. Our local coordinate system is similar to polar coordinate
system, but varies to satisfy the specific requirement.

In the domain near the boundary, for each $\theta$, we have the
outward unit normal vector
\begin{eqnarray}
\vn=\left(\frac{r(\theta)\cos\theta+r'(\theta)\sin\theta}{\sqrt{r(\theta)^2+r'(\theta)^2}},\frac{r(\theta)\sin\theta-r'(\theta)\cos\theta}{\sqrt{r(\theta)^2+r'(\theta)^2}}\right).
\end{eqnarray}
We can determine each
point on this normal line by $\theta$ and its distance $\mu$ to
the boundary point $\bigg(r(\theta)\cos\theta,r(\theta)\sin\theta\bigg)$ as follows:
\begin{eqnarray}\label{local}
\left\{
\begin{array}{rcl}
x_1&=&r(\theta)\cos\theta+\mu\dfrac{-r(\theta)\cos\theta-r'(\theta)\sin\theta}{\sqrt{r(\theta)^2+r'(\theta)^2}},\\\rule{0ex}{2.0em}
x_2&=&r(\theta)\sin\theta+\mu\dfrac{-r(\theta)\sin\theta+r'(\theta)\cos\theta}{\sqrt{r(\theta)^2+r'(\theta)^2}},
\end{array}
\right.
\end{eqnarray}
where $r'(\theta)=\dfrac{\ud{r}}{\ud{\theta}}$. It is easy to see that $\mu=0$ denotes the boundary $\p\Omega$ and $\mu>0$ denotes the interior of $\Omega$.

By chain rule, for any $u=u(x_1,x_2)$ we have
\begin{eqnarray}
\frac{\p u}{\p x_1}=\frac{\p u}{\p\theta}\frac{\p\theta}{\p
x_1}+\frac{\p u}{\p\mu}\frac{\p\mu}{\p x_1},\label{local 1}\\
\frac{\p u}{\p x_2}=\frac{\p u}{\p\theta}\frac{\p\theta}{\p
x_2}+\frac{\p u}{\p\mu}\frac{\p\mu}{\p x_2}.\label{local 2}
\end{eqnarray}
Hence, the key part is to calculate $\dfrac{\p\theta}{\p
x_1}$, $\dfrac{\p\mu}{\p x_1}$,
$\dfrac{\p\theta}{\p
x_2}$ and $\dfrac{\p\mu}{\p x_2}$ in terms of $\mu$ and $\theta$. For
simplicity, we may denote the transform (\ref{local}) as follows.
\begin{eqnarray}\label{local.}
\left\{
\begin{array}{rcl}
x_1&=&a(\theta)+\phi A(\theta),\\
x_2&=&b(\theta)+\phi B(\theta),
\end{array}
\right.
\end{eqnarray}
where
\begin{eqnarray}
a=r\cos\theta,&\quad& A=\frac{-r\cos\theta-r'\sin\theta}{(r^2+(r')^2)^{1/2}},\\
b=r\sin\theta,&\quad& B=\frac{-r\sin\theta+r'\cos\theta}{(r^2+(r')^2)^{1/2}}.
\end{eqnarray}
Taking $x_1$ and $x_2$ derivative in (\ref{local.}) reveals that
\begin{eqnarray}
(a'+\phi A')\frac{\p\theta}{\p x_1}+A\frac{\p\phi}{\p x_1}&=&1,\label{system 1}\\
(b'+\phi B')\frac{\p\theta}{\p x_1}+B\frac{\p\phi}{\p x_1}&=&0,\label{system 2}\\
(a'+\phi A')\frac{\p\theta}{\p x_2}+A\frac{\p\phi}{\p x_2}&=&0,\label{system 3}\\
(b'+\phi B')\frac{\p\theta}{\p x_2}+B\frac{\p\phi}{\p x_2}&=&1,\label{system 4}
\end{eqnarray}
where the superscript $'$ denotes the derivative with respect to
$\theta$. The detailed expression is as follows:
\begin{eqnarray}
a'&=&r'\cos\theta-r\sin\theta,\\
b'&=&r'\sin\theta+r\cos\theta,\\
A'&=&\frac{r^3\sin\theta-2(r')^3\cos\theta-r''r^2\sin\theta+2r(r')^2\sin\theta-r'r^2\cos\theta+rr'r''\cos\theta}{(r^2+(r')^2))^{3/2}},\\
B'&=&\frac{-r^3\cos\theta-2(r')^3\sin\theta-2r(r')^2\cos\theta-r^2r'\sin\theta+r^2r''\cos\theta+rr'r''\sin\theta}{(r^2+(r')^2))^{3/2}}.
\end{eqnarray}
Then we can solve the linear system (\ref{system 1}) to
(\ref{system 4}) by Cramer's rule as
\begin{eqnarray}
\frac{\p\theta}{\p x_1}&=&\frac{\bigg\vert\begin{array}{cc}1&A\\
0&B\end{array}\bigg\vert}{\bigg\vert\begin{array}{cc}a'+\mu
A'&A\\
b'+\mu B'&B\end{array}\bigg\vert}=\frac{B}{C},\qquad
\frac{\p\mu}{\p x_1}=\frac{\bigg\vert\begin{array}{cc}a'+\mu
A'&1\\
b'+\mu
B'&0\end{array}\bigg\vert}{\bigg\vert\begin{array}{cc}a'+\mu
A'&A\\
b'+\mu B'&B\end{array}\bigg\vert}=-\frac{b'+\mu B'}{C},\\
\frac{\p\theta}{\p x_2}&=&\frac{\bigg\vert\begin{array}{cc}0&A\\
1&B\end{array}\bigg\vert}{\bigg\vert\begin{array}{cc}b'+\mu B'&A\\
b'+\mu B'&B\end{array}\bigg\vert}=-\frac{A}{C},\qquad
\frac{\p\mu}{\p x_2}=\frac{\bigg\vert\begin{array}{cc}a'+\mu
A'&0\\
b'+\mu
B'&1\end{array}\bigg\vert}{\bigg\vert\begin{array}{cc}a'+\mu
A'&A\\
b'+\mu B'&B\end{array}\bigg\vert}=\frac{a'+\mu A'}{C},
\end{eqnarray}
where $C$ denotes the determinant of the system, which is also the
Jacobian of the transform $(x_1,x_2)\rightarrow(\mu,\theta)$ as
\begin{equation}
C=\bigg\vert\begin{array}{cc}a'+\mu
A'&A\\
b'+\mu B'&B\end{array}\bigg\vert.
\end{equation}
Then a direct calculation reveals that
\begin{eqnarray}
C&=&(r^2+(r')^2)^{1/2}+\mu\frac{rr''-r^2-2r'^2}{(r^2+r'^2)},\nonumber
\end{eqnarray}
and
\begin{eqnarray}
\frac{\p\theta}{\p x_1}&=&\frac{(-r\sin\theta+r'\cos\theta)(r^2+r'^2)^{1/2}}{(r^2+r'^2)^{3/2}+\mu(rr''-r^2-2r'^2)},\\
\frac{\p\theta}{\p x_2}&=&\frac{(r\cos\theta+r'\sin\theta)(r^2+r'^2)^{1/2}}{(r^2+r'^2)^{3/2}+\mu(rr''-r^2-2r'^2)},\\
\frac{\p\mu}{\p x_1}&=&
\frac{-(r\cos\theta+r'\sin\theta)(r^2+r'^2)}{(r^2+r'^2)^{3/2}+\mu(rr''-r^2-2r'^2)}\\
&&-\mu\frac{-r^3\cos\theta-2r'^3\sin\theta-2rr'^2\cos\theta
-r^2r'\sin\theta+r^2r''\cos\theta+rr'r''\sin\theta}{(r^2+r'^2)^{2}+\mu(rr''-r^2-2r'^2)(r^2+r'^2)^{1/2}},\nonumber\\
\frac{\p\mu}{\p x_2}&=&\frac{(-r\sin\theta+r'\cos\theta)(r^2+r'^2)}{(r^2+r'^2)^{3/2}+\mu(rr''-r^2-2r'^2)}\\
&&+\mu\frac{r^3\sin\theta-2r'^3\cos\theta-r''r^2\sin\theta+2rr'^2\sin\theta-r'r^2\cos\theta+rr'r''\cos\theta}
{(r^2+r'^2)^{2}+\mu(rr''-r^2-2r'^2)(r^2+r'^2)^{1/2}}.\nonumber
\end{eqnarray}
A further simplification shows that, we may denote above relation as
follows:
\begin{eqnarray}
\frac{\p\theta}{\p x_1}=\frac{MP}{P^3+Q\mu},&\quad&
\frac{\p\mu}{\p x_1}=-\frac{N}{P},\\
\frac{\p\theta}{\p x_2}=\frac{NP}{P^3+Q\mu},&\quad&
\frac{\p\mu}{\p x_2}=\frac{M}{P},
\end{eqnarray}
where
\begin{eqnarray}
P&=&(r^2+r'^2)^{1/2},\\
Q&=&rr''-r^2-2r'^2,\\
M&=&-r\sin\theta+r'\cos\theta,\\
N&=&r\cos\theta+r'\sin\theta.
\end{eqnarray}
Therefore, by (\ref{local 1}) and (\ref{local 2}), noting the fact that for smooth convex domain, the curvature
\begin{eqnarray}
\kappa(\theta)=\frac{r^2+2r'^2-rr''}{(r^2+r'^2)^{3/2}},
\end{eqnarray}
and radius of curvature
\begin{eqnarray}
R_{\kappa}(\theta)=\frac{1}{\kappa(\theta)}=\frac{(r^2+r'^2)^{3/2}}{r^2+2r'^2-rr''},
\end{eqnarray}
we define substitutions as follows:\\
\ \\
Substitution 1: \\
Let $u^{\e}(x_1,x_2,w_1,w_2)\rt u^{\e}(\mu,\theta,w_1,w_2)$ with
$(\mu,\theta,w_1,w_2)\in [0,R_{\min})\times[-\pi,\pi)\times\s^1$ for $R_{\min}=\min_{\theta}R_{\kappa}$ as
\begin{eqnarray}\label{substitution 1}
\left\{
\begin{array}{rcl}
x_1&=&r(\theta)\cos\theta+\mu\dfrac{-r(\theta)\cos\theta-r'(\theta)\sin\theta}{\sqrt{r(\theta)^2+r'(\theta)^2}},\\\rule{0ex}{2.0em}
x_2&=&r(\theta)\sin\theta+\mu\dfrac{-r(\theta)\sin\theta+r'(\theta)\cos\theta}{\sqrt{r(\theta)^2+r'(\theta)^2}},
\end{array}
\right.
\end{eqnarray}
and then the equation (\ref{transport}) is transformed into
\begin{eqnarray}\label{transport 1}
\left\{
\begin{array}{rcl}
&&\displaystyle\e\Bigg(w_1\frac{-r\cos\theta-r'\sin\theta}{(r^2+r'^2)^{1/2}}+w_2\frac{-r\sin\theta+r'\cos\theta}{(r^2+r'^2)^{1/2}}\Bigg)\frac{\p u^{\e}}{\p\mu}\\
&&\displaystyle+\e\Bigg(w_1\frac{-r\sin\theta+r'\cos\theta}{(r^2+r'^2)(1-\kappa\mu)}+w_2\frac{r\cos\theta+r'\sin\theta}{(r^2+r'^2)(1-\kappa\mu)}\Bigg)\frac{\p u^{\e}}{\p\theta}+u^{\e}-\bar u^{\e}=0,\\\rule{0ex}{2.0em}
&&u^{\e}(0,\theta,\vw)=\pp[u^{\e}](\theta)+\e g(\theta,\vw)\ \ \text{for}\
\ \vw\cdot\vn<0,
\end{array}
\right.
\end{eqnarray}
where
\begin{eqnarray}
\vw\cdot\vn=w_1\frac{-r\cos\theta-r'\sin\theta}{(r^2+r'^2)^{1/2}}+w_2\frac{-r\sin\theta+r'\cos\theta}{(r^2+r'^2)^{1/2}},
\end{eqnarray}
and
\begin{eqnarray}
\pp[u^{\e}](\vx_0)=\frac{1}{2}\int_{\vw\cdot\vec
n>0}u^{\e}(\vx_0,\vw)(\vw\cdot\vn)\ud{\vw},
\end{eqnarray}
in a neighborhood of the boundary.

In order for the transform being bijective, we require the Jacobian
$C>0$. Then it implies that $0\leq\mu<R_{\kappa}(\theta)$, which is the maximum
extension of the valid domain for local coordinate system. Since
we will only use this coordinate system in a neighborhood of the
boundary, above analysis reveals that as long as the largest
curvature of the boundary is strictly positive and finite, which is naturally satisfied in a smooth convex domain, we can take the transform as
valid for area of $0\leq\mu<\min_{\theta}R_{\kappa}$. For the unit
plate, we have $R_{\kappa}=1$ and the transform is valid for all the
points in the plate except the center.

Noting the fact that
\begin{eqnarray}
\left(\frac{M}{P}\right)^2+\left(\frac{N}{P}\right)^2=
\left(\frac{-r\cos\theta-r'\sin\theta}{(r^2+r'^2)^{1/2}}\right)^2+\left(\frac{-r\sin\theta+r'\cos\theta}{(r^2+r'^2)^{1/2}}\right)^2=1,
\end{eqnarray}
we can further simplify (\ref{transport 1}).\\
\ \\
Substitution 2: \\
Let $u^{\e}(\mu,\theta,w_1,w_2)\rt u^{\e}(\mu,\tau,w_1,w_2)$ with
$(\mu,\tau,w_1,w_2)\in [0,R_{\min})\times[-\pi,\pi)\times\s^1$ as
\begin{eqnarray}\label{substitution 2}
\left\{
\begin{array}{rcl}
\sin\tau&=&\dfrac{r\sin\theta-r'\cos\theta}{(r^2+r'^2)^{1/2}},\\\rule{0ex}{2.0em}
\cos\tau&=&\dfrac{r\cos\theta+r'\sin\theta}{(r^2+r'^2)^{1/2}},
\end{array}
\right.
\end{eqnarray}
which implies
\begin{eqnarray}
\frac{\ud{\tau}}{\ud{\theta}}=\kappa(r^2+r'^2)^{1/2}>0,
\end{eqnarray}
and then the equation (\ref{transport}) is transformed into
\begin{eqnarray}\label{transport 2}
\left\{
\begin{array}{l}\displaystyle
-\e\left(w_1\cos\tau+w_2\sin\tau\right)\frac{\p
u^{\e}}{\p\mu}-\frac{\e}{R_{\kappa}-\mu}\left(w_1\sin\tau-w_2\cos\tau\right)\frac{\p
u^{\e}}{\p\tau}+u^{\e}-\bar u^{\e}=0,\\\rule{0ex}{2.0em}
u^{\e}(0,\tau,\vw)=\pp[u^{\e}](0,\tau)+\e g(\tau,\vw)\ \
\text{for}\ \ w_1\cos\tau+w_2\sin\tau<0,
\end{array}
\right.
\end{eqnarray}
where
\begin{eqnarray}
\pp[u^{\e}](0,\tau)=\half\int_{w_1\cos\tau+w_2\sin\tau>0}u^{\e}(0,\tau,\vw)(\vw\cdot\vn)\ud{\vw},
\end{eqnarray}
in a neighborhood of the boundary. Note that here since $\tau$ denotes the angle of normal vector, the domain of $\tau$ is the same as $\theta$, i.e. $[-\pi,\pi)$.

\subsection{Boundary Layer Expansion with Geometric Correction}

In order to define boundary layer, we need several more substitutions:\\
\ \\
Substitution 3:\\
We further make the scaling transform for $u^{\e}(\mu,\tau,w_1,w_2)\rt
u^{\e}(\eta,\tau,w_1,w_2)$ with $(\eta,\tau,w_1,w_2)\in
[0,R_{\min}/\e)\times[-\pi,\pi)\times\s^1$ as
\begin{eqnarray}\label{substitution 3}
\left\{
\begin{array}{rcl}
\eta&=&\mu/\e,\\
\tau&=&\tau,\\
w_1&=&w_1,\\
w_2&=&w_2,
\end{array}
\right.
\end{eqnarray}
which implies
\begin{eqnarray}
\frac{\p u^{\e}}{\p\mu}=\frac{1}{\e}\frac{\p u^{\e}}{\p\eta}.
\end{eqnarray}
Then equation (\ref{transport}) is transformed into
\begin{eqnarray}\label{transport 3}
\left\{\begin{array}{l}\displaystyle
-\bigg(w_1\cos\tau+w_2\sin\tau\bigg)\frac{\p
u^{\e}}{\p\eta}-\frac{\e}{R_{\kappa}-\e\eta}\bigg(w_1\sin\tau-w_2\cos\tau\bigg)\frac{\p
u^{\e}}{\p\tau}+u^{\e}-\bar u^{\e}=0,\\\rule{0ex}{2.0em}
u^{\e}(0,\tau,w_1,w_2)=\pp[u^{\e}](0,\tau)+\e g(\tau,w_1,w_2)\ \
\text{for}\ \ w_1\cos\tau+w_2\sin\tau<0,
\end{array}
\right.
\end{eqnarray}
where
\begin{eqnarray}
\pp[u^{\e}](0,\tau)=\half\int_{w_1\cos\tau+w_2\sin\tau>0}u^{\e}(0,\tau,\vw)(\vw\cdot\vn)\ud{\vw}.
\end{eqnarray}
\ \\
Substitution 4:\\
Define the velocity substitution for $u^{\e}(\eta,\tau,w_1,w_2)\rt
u^{\e}(\eta,\tau,\xi)$ with $(\eta,\tau,\xi)\in
[0,R_{\min}/\e)\times[-\pi,\pi)\times[-\pi,\pi)$ as
\begin{eqnarray}\label{substitution 4}
\left\{
\begin{array}{rcl}
\eta&=&\eta\\
\tau&=&\tau\\
w_1&=&-\sin\xi\\
w_2&=&-\cos\xi
\end{array}
\right.
\end{eqnarray}
We have the succinct form
\begin{eqnarray}\label{transport 4}
\left\{\begin{array}{l}\displaystyle \sin(\tau+\xi)\frac{\p
u^{\e}}{\p\eta}-\frac{\e}{R_{\kappa}-\e\xi}\cos(\tau+\xi)\frac{\p
u^{\e}}{\p\tau}+u^{\e}-\bar u^{\e}=0,\\\rule{0ex}{2.0em}
u^{\e}(0,\tau,\xi)=\pp[u^{\e}](0,\tau)+\e g(\tau,\xi),\ \ \text{for}\ \
\sin(\tau+\xi)>0
\end{array}
\right.
\end{eqnarray}
where
\begin{eqnarray}
\pp[u^{\e}](0,\tau)=-\half\int_{\sin(\tau+\xi)<0}u^{\e}(0,\tau,\xi)\sin(\tau+\xi)\ud{\xi}
\end{eqnarray}
\ \\
Substitution 5:\\
Finally, we make the substitution for $u^{\e}(\eta,\tau,\xi)\rt
u^{\e}(\eta,\tau,\phi)$ with $(\eta,\tau,\phi)\in
[0,R_{\min}/\e)\times[-\pi,\pi)\times[-\pi,\pi)$ as
\begin{eqnarray}\label{substitution 5}
\left\{
\begin{array}{rcl}
\eta&=&\eta\\
\tau&=&\tau\\
\phi&=&\tau+\xi
\end{array}
\right.
\end{eqnarray}
and achieve the form
\begin{eqnarray}\label{transport temp}
\left\{\begin{array}{l}\displaystyle \sin\phi\frac{\p
u^{\e}}{\p\eta}-\frac{\e}{R_{\kappa}-\e\eta}\cos\phi\bigg(\frac{\p
u^{\e}}{\p\phi}+\frac{\p
u^{\e}}{\p\tau}\bigg)+u^{\e}-\bar u^{\e}=0\\\rule{0ex}{2.0em}
u^{\e}(0,\tau,\phi)=\pp[u^{\e}](0,\tau)+\e g(\tau,\phi)\ \ \text{for}\ \
\sin\phi>0
\end{array}
\right.
\end{eqnarray}
where
\begin{eqnarray}
\pp[u^{\e}](0,\tau)=-\half\int_{\sin\phi<0}u^{\e}(0,\tau,\tau)\sin\phi\ud{\phi}
\end{eqnarray}
We define the boundary layer expansion as follows:
\begin{eqnarray}\label{boundary layer expansion}
\ub(\eta,\tau,\phi)\sim\ub_0(\eta,\tau,\phi)+\e\ub_1(\eta,\tau,\phi),
\end{eqnarray}
where $\ub_k$ can be determined by comparing the order of $\e$ via
plugging (\ref{boundary layer expansion}) into the equation
(\ref{transport temp}). Thus, in a neighborhood of the boundary, we have
\begin{eqnarray}
\sin\phi\frac{\p \ub_0}{\p\eta}+\frac{\e}{R_{\kappa}-\e\eta}\cos\phi\frac{\p
\ub_0}{\p\phi}+\ub_0-\bub_0&=&0,\label{expansion temp 6}\\
\sin\phi\frac{\p \ub_1}{\p\eta}+\frac{\e}{R_{\kappa}-\e\eta}\cos\phi\frac{\p
\ub_1}{\p\phi}+\ub_1-\bub_1&=&\frac{1}{R_{\kappa}-\e\eta}\cos\phi\frac{\p
\ub_0}{\p\tau},\label{expansion temp 7}
\end{eqnarray}
where
\begin{eqnarray}
\bub_k(\eta,\tau)=\frac{1}{2\pi}\int_{-\pi}^{\pi}\ub_k(\eta,\tau,\phi)\ud{\phi}.
\end{eqnarray}

\subsection{Matching of Interior Solution and Boundary Layer}

The bridge between interior solution and boundary layer
is the boundary condition of (\ref{transport}), so we
first consider the boundary expansion:
\begin{eqnarray}
(\u_0+\ub_0)&=&\pp[\u_0+\ub_0],\\
(\u_1+\ub_1)&=&\pp[\u_1+\ub_1]+g.
\end{eqnarray}
Noting the fact that $\bu_k=\pp[\bu_k]$, we can simplify above
conditions as follows:
\begin{eqnarray}
\ub_0&=&\pp[\ub_0],\\
\ub_1&=&\pp[\ub_1]+(\vw\cdot\u_0-\pp(\vw\cdot\u_0))+g.
\end{eqnarray}
The construction of $\u_k$ and $\ub_k$ is as follows:\\
\ \\
Step 0: Preliminaries.\\
Assume the cut-off functions $\psi_0\in C^{\infty}[0,\infty)$ is defined as
\begin{eqnarray}\label{cut-off 2}
\psi_0(y)=\left\{
\begin{array}{ll}
1&0\leq y\leq\dfrac{1}{2},\\
0&\dfrac{3}{4}\leq y<\infty.
\end{array}
\right.
\end{eqnarray}
Also, define the force as
\begin{eqnarray}\label{force}
F(\e;\eta,\tau)=-\frac{\e}{R_{\kappa}(\tau)-\e\eta},
\end{eqnarray}
and the length for $\e$-Milne problem as $L=\e^{-1/2}$. For $\phi\in[-\pi,\pi]$, denote $R\phi=-\phi$.\\
\ \\
Step 1: Construction of $\ub_0$.\\
Define the zeroth order boundary layer as
\begin{eqnarray}\label{expansion temp 9}
\left\{
\begin{array}{rcl}
\ub_0(\eta,\tau,\phi)&=&\psi_0(\e^{1/2}\eta)\bigg(f_0^{\e}(\eta,\tau,\phi)-f^{\e}_{0,L}(\tau)\bigg),\\
\sin\phi\dfrac{\p f_0^{\e}}{\p\eta}+F(\e;\eta,\tau)\cos\phi\dfrac{\p
f_0^{\e}}{\p\phi}+f_0^{\e}-\bar f_0^{\e}&=&0,\\\rule{0ex}{1em}
f_0^{\e}(0,\tau,\phi)&=&\pp[f_0^{\e}](0,\tau)\ \ \text{for}\ \
\sin\phi>0,\\\rule{0ex}{1em}
f_0^{\e}(L,\tau,\phi)&=&f_0^{\e}(L,\tau,R\phi),
\end{array}
\right.
\end{eqnarray}
with
\begin{eqnarray}
\pp[f_0^{\e}](0,\tau)=0,
\end{eqnarray}
and $f^{\e}_{0,L}$ is defined as in Theorem \ref{Milne theorem 1}. Thus, we have $\ub_0$ is well-defined. It is
obvious to see $f_0^{\e}=f_{0,L}^{\e}=0$ is the only solution.\\
\ \\
Step 2: Construction of $\ub_1$ and $\u_0$.\\
Define the first order boundary layer as
\begin{eqnarray}\label{expansion temp 10}
\left\{
\begin{array}{rcl}
\ub_1(\eta,\tau,\phi)&=&\psi_0(\e^{1/2}\eta)\bigg(f_1^{\e}(\eta,\tau,\phi)-f^{\e}_{1,L}(\tau)\bigg),\\
\sin\phi\dfrac{\p f_1^{\e}}{\p\eta}+F(\e;\eta,\tau)\cos\phi\dfrac{\p
f_1^{\e}}{\p\phi}+f_1^{\e}-\bar
f_1^{\e}&=&\dfrac{1}{R_{\kappa}-\e\eta}\cos\phi\dfrac{\p
\ub_0}{\p\tau},\\\rule{0ex}{1.5em} f_1^{\e}(0,\tau,\phi)&=&\pp
[f_1^{\e}](0,\tau)+g_1(\tau,\phi)\ \ \text{for}\ \
\sin\phi>0,\\\rule{0ex}{1.5em}
f_1^{\e}(L,\tau,\phi)&=&f_1^{\e}(L,\tau,R\phi),
\end{array}
\right.
\end{eqnarray}
with
\begin{eqnarray}
\pp[f_1^{\e}](0,\tau)=0,
\end{eqnarray}
and $f^{\e}_{1,L}$ is defined as in Theorem \ref{Milne theorem 1}, where
\begin{eqnarray}
g_1(\vx_0,\vw)&=&\vw\cdot\nx\u_0(\vx_0)-\pp[\vw\cdot\nx\u_0(\vx_0)]+g(\vx_0,\vw),
\end{eqnarray}
with $\vx_0$ and $(0,\tau)$ denoting the same boundary point, and
\begin{eqnarray}
\vw&=&(-\sin(\phi-\tau),-\cos(\phi-\tau)),\\
\vn&=&(\cos\tau,\sin\tau).
\end{eqnarray}
To solve (\ref{expansion temp 10}), the data must satisfy the compatibility
condition (\ref{Milne compatibility condition}) as
\begin{eqnarray}
\\
\int_{\sin\phi>0}\bigg(g+\vw\cdot\nx\u_0(\vx_0)-\pp[\vw\cdot\nx\u_0(\vx_0)]\bigg)\sin\phi\ud{\phi}
+\int_0^{L}\int_{-\pi}^{\pi}\ue^{-V(s)}\frac{1}{1-\e
s}\cos\phi\frac{\p
\ub_0}{\p\tau}(s,\tau,\phi)\ud{\phi}\ud{s}\nonumber\\
=0,&&\nonumber
\end{eqnarray}
where $\dfrac{\p V}{\p\eta}=-F$ and $V(0)=0$.
Note the fact
\begin{eqnarray}
&&\int_{\sin\phi>0}\bigg(\vw\cdot\nx\u_0(\vx_0)-\pp[\vw\cdot\nx\u_0(\vx_0)]\bigg)\sin\phi\ud{\phi}\\
&=&
\int_{\sin\phi>0}(\vw\cdot\nx\u_0(\vx_0))\sin\phi\ud{\phi}-2\pp[\vw\cdot\nx\u_0(\vx_0)]\nonumber\\
&=&\int_{\sin\phi>0}(\vw\cdot\nx\u_0(\vx_0))\sin\phi\ud{\phi}+\int_{\sin\phi<0}(\vw\cdot\nx\u_0(\vx_0))\sin\phi\ud{\phi}\nonumber\\
&=&\int_{-\pi}^{\pi}(\vw\cdot\nx\u_0(\vx_0))\sin\phi\ud{\phi}\nonumber\\
&=&-\pi\nx\bu_0(\vx_0)\cdot\vn=-\pi\frac{\p\bu_0(\vx_0)}{\p\vec
\nu}.\nonumber
\end{eqnarray}
We can simplify the compatibility condition as follows:
\begin{eqnarray}
\int_{\sin\phi>0}g(\phi)\sin\phi\ud{\phi}-\pi\frac{\p\bu_0(\vx_0)}{\p\vec
\nu} +\int_0^{L}\int_{-\pi}^{\pi}\ue^{-V(s)}\frac{1}{1-\e
s}\cos\phi\frac{\p \ub_0}{\p\tau}(s,\tau,\phi)\ud{\phi}\ud{s}=0.
\end{eqnarray}
Then we have
\begin{eqnarray}
\frac{\p\bu_0(\vx_0)}{\p\vec
\nu}&=&\frac{1}{\pi}\int_{\sin\phi>0}g(\tau,\phi)\sin\phi\ud{\phi}
+\frac{1}{\pi}\int_0^{L}\int_{-\pi}^{\pi}\ue^{-V(s)}\frac{1}{1-\e
s}\cos\phi\frac{\p
\ub_0}{\p\tau}(s,\tau,\phi)\ud{\phi}\ud{s}\\
&=&\frac{1}{\pi}\int_{\sin\phi>0}g(\tau,\phi)\sin\phi\ud{\phi}.\nonumber
\end{eqnarray}
Hence, we define the zeroth order interior solution $\u_0(\vx,\vw)$ as
\begin{eqnarray}\label{expansion temp 11}
\left\{
\begin{array}{rcl}
\u_0&=&\bu_0 ,\\\rule{0ex}{1em} \Delta_x\bu_0&=&0\ \ \text{in}\
\ \Omega,\\\rule{0ex}{1em}\dfrac{\p\bu_0}{\p\vec
\nu}&=&\dfrac{1}{\pi}\displaystyle
\int_{\sin\phi>0}g(\tau,\phi)\sin\phi\ud{\phi}\ \ \text{on}\ \
\p\Omega,\\\rule{0ex}{1em}
\displaystyle\int_{\Omega}\u_0(\vx)\ud{\vx}&=&0.
\end{array}
\right.
\end{eqnarray}
\ \\
Step 3: Construction of $\u_1$.\\
We do not expand the boundary layer to $\ub_2$ and just terminate at $\ub_1$. Then we define the first order interior solution $\u_1(\vx)$ as
\begin{eqnarray}\label{expansion temp 12.}
\left\{
\begin{array}{rcl}
\u_1&=&\bu_1-\vw\cdot\nx\u_0,\\\rule{0ex}{1em}
\Delta_x\bu_1&=&-\displaystyle\int_{\s^1}(\vw\cdot\nx\u_{0})\ud{\vw}\
\ \text{in}\ \ \Omega,\\\rule{0ex}{1.0em} \dfrac{\p\bu_1}{\p\vec
\nu}&=&0\ \ \text{on}\ \
\p\Omega,\\\rule{0ex}{1em}
\displaystyle\int_{\Omega}\bu_1(\vx)\ud{\vx}&=&0.
\end{array}
\right.
\end{eqnarray}
Note that here we only require the trivial boundary condition since we cannot resort to the compatibility condition in $\e$-Milne problem with geometric correction. Based on \cite{AA003}, this might lead to $O(\e^2)$ error to the boundary approximation. Thanks to the improved remainder estimate, this error is acceptable.\\
\ \\
Step 4: Construction of $\u_2$.\\
By a similar fashion, we define the second order interior solution as
\begin{eqnarray}
\left\{
\begin{array}{rcl}
\u_{2}&=&\bu_{2}-\vw\cdot\nx\u_{1},\\\rule{0ex}{1em}
\Delta_x\bu_{2}&=&-\displaystyle\int_{\s^1}(\vw\cdot\nx\u_{1})\ud{\vw}\
\ \text{in}\ \ \Omega,\\\rule{0ex}{1.0em} \dfrac{\p\bu_{2}}{\p\vec
\nu}&=&0\ \ \text{on}\ \
\p\Omega,\\\rule{0ex}{1em}
\displaystyle\int_{\Omega}\bu_2(\vx)\ud{\vx}&=&0.
\end{array}
\right.
\end{eqnarray}
As the case of $\u_1$, we might have $O(\e^3)$ error in this step due to the trivial boundary data. However, it will not affect the diffusive limit.

\section{Regularity of $\e$-Milne Problem with Geometric Correction}

We consider the $\e$-Milne problem with geometric correction for $f^{\e}(\eta,\tau,\phi)$ in
the domain $(\eta,\tau,\phi)\in[0,L]\times[-\pi,\pi)\times[-\pi,\pi)$ where $L=\e^{-1/2}$ as
\begin{eqnarray}\label{Milne problem.}
\left\{ \begin{array}{rcl}\displaystyle \sin\phi\frac{\p
f^{\e}}{\p\eta}+F(\e;\eta,\tau)\cos\phi\frac{\p
f^{\e}}{\p\phi}+f^{\e}-\bar f^{\e}&=&S^{\e}(\eta,\tau,\phi),\\
f^{\e}(0,\tau,\phi)&=& h^{\e}(\tau,\phi)+\pp[f^{\e}](0,\tau)\ \ \text{for}\
\ \sin\phi>0,\\
f^{\e}(L,\tau,\phi)&=&f^{\e}(L,\tau,R\phi),
\end{array}
\right.
\end{eqnarray}
where $R\phi=-\phi$,
\begin{eqnarray}
\pp
[f^{\e}](0,\tau)=-\half\int_{\sin\phi<0}f^{\e}(0,\tau,\phi)\sin\phi\ud{\phi},
\end{eqnarray}
\begin{eqnarray}
F(\e;\eta,\tau)=-\frac{\e}{R_{\kappa}(\tau)-\e\eta},
\end{eqnarray}
In this section, for convenience, we temporarily ignore the superscript on $\e$. We define the norms in the
space $(\eta,\phi)\in[0,L]\times[-\pi,\pi)$ as follows:
\begin{eqnarray}
\tnnm{f(\tau)}&=&\bigg(\int_0^{L}\int_{-\pi}^{\pi}\abs{f(\eta,\tau,\phi)}^2\ud{\phi}\ud{\eta}\bigg)^{1/2},\\
\lnnm{f(\tau)}&=&\sup_{(\eta,\phi)\in[0,L]\times[-\pi,\pi)}\abs{f(\eta,\tau,\phi)},
\end{eqnarray}
Similarly, we can define the norm at in-flow boundary as
\begin{eqnarray}
\tnm{f(0,\tau)}&=&\bigg(\int_{\sin\phi>0}\abs{f(0,\tau,\phi)}^2\ud{\phi}\bigg)^{1/2},\\
\lnm{f(0,\tau)}&=&\sup_{\sin\phi>0}\abs{f(0,\tau,\phi)},
\end{eqnarray}
Also define
\begin{eqnarray}
\br{f,g}_{\phi}(\eta,\tau)=\int_{-\pi}^{\pi}f(\eta,\tau,\phi)g(\eta,\tau,\phi)\ud{\phi}
\end{eqnarray}
as the $L^2$ inner product in $\phi$. We further assume
\begin{eqnarray}\label{Milne bounded}
\lnm{h(\tau)}+\lnm{\frac{\p h}{\p\phi}(\tau)}+\lnm{\frac{\p h}{\p\tau}(\tau)}\leq C,
\end{eqnarray}
and
\begin{eqnarray}\label{Milne decay}
\lnnm{S(\tau)}+\lnnm{\frac{\p S}{\p\eta}(\tau)}+\lnnm{\frac{\p S}{\p\phi}(\tau)}+\lnnm{\frac{\p S}{\p\tau}(\tau)}\leq C\ue^{-K\eta},
\end{eqnarray}
for $C>0$ and $K>0$ uniform in $\e$ and $\tau$.

As in \cite[Section 6]{AA003}, in order to study problem with diffusive boundary, we first need to study the $\e$-Milne problem with in-flow boundary for
$f(\eta,\tau,\phi)$ in the domain
$(\eta,\tau,\phi)\in[0,L]\times[-\pi,\pi)\times[-\pi,\pi)$ as
\begin{eqnarray}\label{Milne problem}
\left\{
\begin{array}{rcl}\displaystyle
\sin\phi\frac{\p f}{\p\eta}+F(\eta,\tau)\cos\phi\frac{\p
f}{\p\phi}+f-\bar f&=&S(\eta,\tau,\phi),\\
f(0,\tau,\phi)&=&h(\tau,\phi)\ \ \text{for}\ \ \sin\phi>0,\\\rule{0ex}{1.0em}
f(L,\tau,\phi)&=&f(L,\tau, R\phi).
\end{array}
\right.
\end{eqnarray}
Define a potential function $V(\e;\eta,\tau)$ satisfying $V(\e;0,\tau)=0$ and $\dfrac{\p V}{\p\eta}=-F(\e;\eta,\tau)$.
\begin{lemma}\label{rt lemma 1}
We have
$\ue^{-V(\e;0,\tau)}=1$ and
\begin{eqnarray}
\ue^{-V(\e;L,\tau)}=1-\frac{\e^{1/2}}{\rk}.
\end{eqnarray}
\end{lemma}
\begin{proof}
We directly compute
\begin{eqnarray}
V(\e;\eta,\tau)=\ln\bigg(\frac{\rk(\tau)}{\rk(\tau)-\e\eta}\bigg),
\end{eqnarray}
and
\begin{eqnarray}
\ue^{-V(\e;\eta,\tau)}=\frac{\rk(\tau)-\e\eta}{\rk(\tau)}.
\end{eqnarray}
Hence, our result naturally follows.
\end{proof}
In the following, we will temporarily ignore $\e$ dependence. Note that all the estimates are uniform in $\e$, which further means uniform in $L$. From now on, let $C$ denote a finite universal constant that is independent of $\e$ and $\tau$.

\subsection{Well-Posedness and Decay}

Since most of the results can be obtained via obvious modifications of \cite[Section 4]{AA003}, we will only state the main theorems without proofs.

\subsubsection{$L^2$ Estimates}

We may decompose the solution
\begin{eqnarray}
f(\eta,\tau,\phi)=q(\eta,\tau)+r(\eta,\tau,\phi),
\end{eqnarray}
where the hydrodynamical part $q$ is in the null space of the
operator $f-\bar f$, and the microscopic part $r$ is
the orthogonal complement, i.e.
\begin{eqnarray}\label{hydro}
q(\eta,\tau)=\frac{1}{2\pi}\int_{-\pi}^{\pi}f(\eta,\tau,\phi)\ud{\phi}\quad
r(\eta,\tau,\phi)=f(\eta,\tau,\phi)-q(\eta,\tau).
\end{eqnarray}
\begin{lemma}\label{Milne lemma 1}
Assume (\ref{Milne bounded}) and (\ref{Milne decay}) hold. Then there exists a solution $f(\eta,\tau,\phi)$ of the equation
(\ref{Milne problem}), satisfying
\begin{eqnarray}
\tnnm{r(\eta,\tau,\phi)}&\leq&C,\\
\br{\sin\phi,r}_{\phi}(\eta,\tau)&=&-\int_{\eta}^{L}\ue^{V(\eta,\tau)-V(y,\tau)}\bar
S(y,\tau)\ud{y}.
\end{eqnarray}
Also for
\begin{eqnarray}
f_L(\tau)=q_L(\tau)=\frac{\br{\sin^2\phi,f}_{\phi}(L,\tau)}{\tnm{\sin\phi}^2}.
\end{eqnarray}
we have
\begin{eqnarray}
\abs{q_L(\tau)}&\leq&C,\\
\tnm{q(\eta,\tau)-q_L(\tau)}&\leq&C\bigg(\tnm{r(\eta,\tau)}+\int_{\eta}^{L}\abs{F(y,\tau)}\tnm{r(y,\tau)}\ud{y}+\int_{\eta}^{L}\lnm{S(y,\tau)}\ud{y}
\bigg),\\
\tnnm{q(\eta,\tau)-q_L(\tau)}&\leq&C.
\end{eqnarray}
The
solution is unique among functions satisfying
$\tnnm{f(\eta,\tau,\phi)-f_L(\tau)}\leq C$.
\end{lemma}
\begin{theorem}\label{Milne theorem 1}
Assume (\ref{Milne bounded}) and (\ref{Milne decay}) hold. For the $\e$-Milne problem (\ref{Milne problem}), there exists a unique
solution $f(\eta,\tau,\phi)$ satisfying the estimates
\begin{eqnarray}
\tnnm{f(\eta,\tau,\phi)-f_L(\tau)}\leq C
\end{eqnarray}
for some real number $f_L(\tau)$ satisfying
\begin{eqnarray}
\abs{f_L(\tau)}\leq C.
\end{eqnarray}
\end{theorem}

\subsubsection{$L^{\infty}$ Estimates}

\begin{lemma}\label{Milne lemma 2}
Assume (\ref{Milne bounded}) and (\ref{Milne decay}) hold. The solution $f(\eta,\tau,\phi)$ to the Milne problem (\ref{Milne
problem}) satisfies
\begin{eqnarray}
\lnnm{f(\eta,\tau,\phi)-f_L(\tau)}\leq C\bigg(1+\tnnm{f(\eta,\tau,\phi)-f_L(\tau)}\bigg).
\end{eqnarray}
\end{lemma}
\begin{theorem}\label{Milne theorem 2}
Assume (\ref{Milne bounded}) and (\ref{Milne decay}) hold. The unique solution $f(\eta,\tau,\phi)$ to the $\e$-Milne problem
(\ref{Milne problem}) satisfies
\begin{eqnarray}
\lnnm{f(\eta,\tau,\phi)-f_L(\tau)}\leq C.
\end{eqnarray}
\end{theorem}

\subsubsection{Exponential Decay}

\begin{theorem}\label{Milne theorem 3}
Assume (\ref{Milne bounded}) and (\ref{Milne decay}) hold. There exists $K_0>0$ such that the unique solution $f(\eta,\tau,\phi)$ to the
$\e$-Milne problem (\ref{Milne problem}) satisfies
\begin{eqnarray}\label{Milne temp 85}
\lnnm{\ue^{K_0\eta}\bigg(f(\eta,\tau,\phi)-f_L(\tau)\bigg)}\leq C.
\end{eqnarray}
\end{theorem}

\subsection{Preliminaries of Regularity Estimates}

It is easy to see $\v(\eta,\tau,\phi)=f(\eta,\tau,\phi)-f_L(\tau)$ satisfies the equation
\begin{eqnarray}\label{Milne difference problem}
\left\{
\begin{array}{rcl}\displaystyle
\sin\phi\frac{\p \v}{\p\eta}+F(\eta,\tau)\cos\phi\frac{\p
\v}{\p\phi}+\v-\bar\v&=&S(\eta,\tau,\phi),\\
\v(0,\tau,\phi)&=&p(\tau,\phi)\ \ \text{for}\ \ \sin\phi>0,\\\rule{0ex}{1.0em}
\v(L,\tau,\phi)&=&\v(L,\tau, R\phi).
\end{array}
\right.
\end{eqnarray}
where
\begin{eqnarray}
p(\tau,\phi)=h(\tau,\phi)-f_L(\tau).
\end{eqnarray}
We intend to estimate the normal, tangential and velocity derivative. This idea is motivated by \cite{Guo.Kim.Tonon.Trescases2013}.
From now on, without specification, we temporarily ignore the dependence on $\tau$ and all the estimates are uniform in $\tau$.
Define a distance function $\zeta(\eta,\phi)$ as
\begin{eqnarray}\label{weight function}
\zeta(\eta,\phi)=\Bigg(1-\bigg(\ue^{-V(\eta)}\cos\phi\bigg)^2\Bigg)^{1/2}.
\end{eqnarray}
Note that the closer $(\eta,\phi)$ is to the grazing set, the smaller $\zeta$ is. In particular, at grazing set, $\zeta=0$. Also, we have $0\leq\zeta\leq 1$.
\begin{lemma}\label{rt lemma 2}
We have
\begin{eqnarray}
\sin\phi\frac{\p \zeta}{\p\eta}+F(\eta)\cos\phi\frac{\p
\zeta}{\p\phi}=0.
\end{eqnarray}
\end{lemma}
\begin{proof}
We may directly compute
\begin{eqnarray}
\frac{\p \zeta}{\p\eta}&=&\frac{1}{2}\Bigg(1-\bigg(\ue^{-V(\eta)}\cos\phi\bigg)^2\Bigg)^{-1/2}\bigg(-2\ue^{-2V(\eta)}\cos^2\phi\bigg)F(\eta)
=-\frac{\ue^{-2V(\eta)}F(\eta)\cos^2\phi}{\zeta},\\
\frac{\p\zeta}{\p\phi}&=&\frac{1}{2}\Bigg(1-\bigg(\ue^{-V(\eta)}\cos\phi\bigg)^2\Bigg)^{-1/2}\bigg(-2\ue^{-2V(\eta)}\cos\phi\bigg)(-\sin\phi)
=\frac{\ue^{-2V(\eta)}\cos\phi\sin\phi}{\zeta}.
\end{eqnarray}
Hence, we know
\begin{eqnarray}
\sin\phi\frac{\p \zeta}{\p\eta}+F(\eta)\cos\phi\frac{\p\zeta}{\p\phi}&=&
\frac{-\sin\phi\bigg(\ue^{-2V(\eta)}F(\eta)\cos^2\phi\bigg)+F(\eta)\cos\phi\bigg(\ue^{-2V(\eta)}\cos\phi\sin\phi\bigg)}{\zeta}=0.
\end{eqnarray}
\end{proof}

\subsection{Direct Estimates along Characteristics}

In this section, we will prove some preliminary estimates that are based on the characteristics of $\v$ itself instead of the derivative.
Here, we have two formulations of the equation (\ref{Milne difference problem}) along the characteristics:
\begin{itemize}
\item
Formulation I: $\eta$ is the principle variable, $\phi=\phi(\eta)$, and the equation can be rewritten as
\begin{eqnarray}
\sin\phi\frac{\ud{\v}}{\ud{\eta}}+\v=S+\bar\v.
\end{eqnarray}
\item
Formulation II: $\phi$ is the principle variable, $\eta=\eta(\phi)$ and the equation can be rewritten as
\begin{eqnarray}
F(\eta)\cos\phi\frac{\ud{\v}}{\ud{\phi}}+\v=S+\bar\v.
\end{eqnarray}
\end{itemize}
These two formulations are equivalent and can be applied to different regions of the domain.
Define the energy as follows:
\begin{eqnarray}
E(\eta,\phi)=\ue^{-V(\eta)}\cos\phi .
\end{eqnarray}
Along the characteristics, this energy is conserved. In the following, let $0<\d_0<<1$ be a small quantity.

\begin{lemma}\label{pt lemma 1}
Assume $\lnnm{S}+\lnnm{\dfrac{\p S}{\p\phi}}\leq C$. For $\sin\phi>\d_0$, we have
\begin{eqnarray}
\abs{\sin\phi\frac{\p\v}{\p\eta}(\eta,\phi)}\leq C\bigg(1+\frac{1}{\d_0^3}\bigg).
\end{eqnarray}
\end{lemma}
\begin{proof}
Using the $\e$-Milne problem (\ref{Milne difference problem}), we only need to show
\begin{eqnarray}
\abs{F(\eta)\cos\phi\frac{\p\v}{\p\phi}(\eta,\phi)}\leq C\bigg(1+\frac{1}{\d_0^3}\bigg).
\end{eqnarray}
We use Formulation I to rewrite the equation along the characteristics as
\begin{eqnarray}\label{mt 01}
\v(\eta,\phi)&=&\exp\left(-G_{\eta,0}\right)\Bigg(p(\phi'(0))
+\int_0^{\eta}\frac{(S+\bar\v)(\eta',\phi'(\eta'))}{\sin\phi'(\eta')}
\exp\left(G_{\eta',0}\right)\ud{\eta'}\Bigg).
\end{eqnarray}
where $\phi'(\eta')=\phi'(\eta';\eta,\phi)$ satisfying $(\eta',\phi')$ and $(\eta,\phi)$ are on the same characteristic with $\sin\phi'\geq0$, and
\begin{eqnarray}
G_{t,s}&=&\int_{s}^{t}\frac{1}{\sin\phi'(\xi)}\ud{\xi}.
\end{eqnarray}
for any $s,t\geq0$. Note that
\begin{eqnarray}
\frac{\p{G_{t,s}}}{\p{\phi}}&=&\int_{s}^{t}\frac{\p{}}{\p{\phi}}\bigg(\frac{1}{\sin\phi'(\xi)}\bigg)\ud{\xi}
=-\int_{s}^{t}\frac{\cos\phi'(\xi)}{\sin\phi'^2(\xi)}\frac{\p\phi'(\xi)}{\p\phi}\ud{\xi}.
\end{eqnarray}
Taking $\phi$ derivative on both sides of (\ref{mt 01}), we have
\begin{eqnarray}\label{mt 04}
\frac{\p\v}{\p\phi}&=&J=J_1+J_2+J_3+J_4,
\end{eqnarray}
where
\begin{eqnarray}
J_1&=&\exp\left(-G_{\eta,0}\right)
\bigg(\int_{0}^{\eta}\frac{\cos\phi'(\xi)}{\sin\phi'^2(\xi)}\frac{\p\phi'(\xi)}{\p\phi}\ud{\xi}\bigg)\Bigg(p(\phi'(0))\\
&&+\int_0^{\eta}\frac{(S+\bar\v)(\eta',\phi'(\eta'))}{\sin\phi'(\eta')}
\exp\left(G_{\eta',0}\right)\ud{\eta'}\Bigg),\no\\
J_2&=&\exp\left(-G_{\eta,0}\right)\frac{\p p(\phi'(0))}{\p\phi},\\
J_3&=&\exp\left(-G_{\eta,0}\right)\Bigg(\int_0^{\eta}(S+\bar\v)(\eta',\phi'(\eta'))
\exp\left(G_{\eta',0}\right)\\
&&\bigg(-\frac{1}{\sin\phi'(\eta')}\int_{0}^{\eta'}\frac{\cos\phi'(\xi)}{\sin\phi'^2(\xi)}\frac{\p\phi'(\xi)}{\p\phi}\ud{\xi}
-\frac{\cos\phi'(\eta')}{\sin\phi^2_{\ast}(\eta')}\frac{\p\phi'(\eta')}{\p\phi}\bigg)\ud{\eta'}\Bigg),\no\\
J_4&=&\exp\left(-G_{\eta,0}\right)
\int_0^{\eta}\frac{1}{\sin\phi'(\eta')}\frac{\p S(\eta',\phi'(\eta'))}{\p\phi}
\exp\left(G_{\eta',0}\right)\ud{\eta'}.
\end{eqnarray}
Then we divide the proof into several steps:\\
\ \\
Step 1: Estimate of $J_1$.\\
We can directly compute
\begin{eqnarray}
J_1&=&\v\bigg(\int_{0}^{\eta}\frac{\cos\phi'(\xi)}{\sin\phi'^2(\xi)}\frac{\p\phi'(\xi)}{\p\phi}\ud{\xi}\bigg).
\end{eqnarray}
Since
\begin{eqnarray}\label{mt 05}
E(\eta,\phi)=\ue^{-V(\eta)}\cos\phi=\ue^{-V(\xi)}\cos\phi'(\xi),
\end{eqnarray}
when taking $\phi$ derivative on both sides of (\ref{mt 05}), we obtain
\begin{eqnarray}
\frac{\p\phi'(\xi)}{\p\phi}=\frac{\sin\phi}{\sin\phi'(\xi)}\ue^{V(\xi)-V(\eta)}.
\end{eqnarray}
Hence, we have
\begin{eqnarray}\label{mt 06}
\int_{0}^{\eta}\frac{\cos\phi'(\xi)}{\sin\phi'^2(\xi)}\frac{\p\phi'(\xi)}{\p\phi}\ud{\xi}
&=&\int_{0}^{\eta}\frac{\cos\phi'(\xi)\sin\phi}{\sin\phi'^3(\xi)}\ue^{V(\xi)-V(\eta)}\ud{\xi}.
\end{eqnarray}
Since $\sin\phi'\geq\sin\phi\geq\d_0$, we naturally have
\begin{eqnarray}
\abs{\int_{0}^{\eta}\frac{\cos\phi'(\xi)\sin\phi}{\sin\phi'^3(\xi)}\ue^{V(\xi)-V(\eta)}\ud{\xi}}\leq \frac{C\eta}{\d_0^3}.
\end{eqnarray}
Since $\v$ decays exponentially, we obtain
\begin{eqnarray}
\abs{J_1}\leq \ue^{-K_0\eta}\frac{C\eta}{\d_0^3}\leq \frac{C}{\d_0^3}.
\end{eqnarray}
\ \\
Step 2: Estimate of $J_2$.\\
For $J_2$, we can estimate
\begin{eqnarray}
\abs{J_2}=\abs{\exp\left(-G_{\eta,0}\right)\frac{\p p(\phi'(0))}{\p\phi}}\leq \abs{\frac{\p p(\phi'(0))}{\p\phi}},
\end{eqnarray}
since for any $\xi\in[0,\eta]$,
\begin{eqnarray}
\frac{1}{\sin\phi'(\xi)}\geq1.
\end{eqnarray}
We may directly compute
\begin{eqnarray}
\ue^{-V(\eta)}\cos\phi=\ue^{-V(0)}\cos\phi'(0)=\cos\phi'(0).
\end{eqnarray}
Taking $\phi$ derivative on both sides, we get
\begin{eqnarray}
\frac{\p \phi'(0)}{\p\phi}=\frac{\sin\phi}{\sin\phi'(0)}\ue^{-V(\eta)},
\end{eqnarray}
which implies
\begin{eqnarray}
\abs{\frac{\p p(\phi'(0))}{\p\phi}}\leq \nm{p}_{W^{1,\infty}}\abs{\frac{\p \phi'(0)}{\p\phi}}\leq (C+\nm{h}_{W^{1,\infty}})\abs{\frac{\p \phi'(0)}{\p\phi}}\leq C\ue^{-V(\eta)}\leq C.
\end{eqnarray}
Hence, we have shown
\begin{eqnarray}
\abs{J_2}\leq C.
\end{eqnarray}
\ \\
Step 3: Estimate of $J_3$.\\
Similar to the estimate of $J_1$, we have
\begin{eqnarray}
\abs{J_3}&\leq&C\abs{\int_0^{\eta}
\exp\left(-G_{\eta',\eta}\right)\bigg(\frac{C(\eta'+1)}{\d_0^3}\bigg)\ud{\eta'}}\leq\frac{C(\eta^2+1)}{\d_0^3}.
\end{eqnarray}
Considering $F(\eta)\leq \e$ and $\eta\leq L=\e^{-1/2}$, we have
\begin{eqnarray}
\abs{J_3}\leq \abs{\frac{1}{F(\eta)}}\frac{C}{\d_0^3}.
\end{eqnarray}
\ \\
Step 4: Estimate of $J_4$.\\
Similar to the estimate of $J_1$, we have
\begin{eqnarray}
\abs{J_4}&\leq&\frac{C}{\d_0}\abs{
\int_0^{\eta}
\exp\left(-G_{\eta',\eta}\right)\ud{\eta'}}\leq C.
\end{eqnarray}
\ \\
Step 5: Synthesis.\\
In summary, we have
\begin{eqnarray}
\abs{J}\leq \frac{C}{\d_0^3}\bigg(1+\abs{\frac{1}{F(\eta)}}\bigg).
\end{eqnarray}
which implies
\begin{eqnarray}
\abs{F(\eta)\cos\phi\frac{\p\v}{\p\phi}}\leq C\bigg(1+\frac{1}{\d_0^3}\bigg),
\end{eqnarray}
and further
\begin{eqnarray}
\abs{\sin\phi\frac{\p\v}{\p\eta}}\leq C\bigg(1+\frac{1}{\d_0^3}\bigg).
\end{eqnarray}
\end{proof}

\begin{lemma}\label{pt lemma 2}
Assume $\lnnm{S}+\lnnm{\dfrac{\p S}{\p\phi}}\leq C$. For $\sin\phi<0$ with $\abs{E(\eta,\phi)}\leq \ue^{-V(L)}$, if it satisfies $\min_{\phi'}\sin\phi'\geq\d_0$ where $(\eta',\phi')$ are on the same characteristics as $(\eta,\phi)$ with $\sin\phi'\geq0$, then we have
\begin{eqnarray}
\abs{\sin\phi\frac{\p\v}{\p\eta}(\eta,\phi)}\leq C\bigg(1+\frac{1}{\d_0^3}\bigg).
\end{eqnarray}
\end{lemma}
\begin{proof}
We use Formulation I to rewrite the equation along the characteristics as
\begin{eqnarray}\label{mt 03}
\v(\eta,\phi)&=&p(\phi'(0))\exp(-G_{L,0}-G_{L,\eta})\\
&&+\int_0^{L}\frac{(S+\bar\v)(\eta',\phi'(\eta'))}{\sin(\phi'(\eta'))}
\exp(-G_{L,\eta'}-G_{L,\eta})\ud{\eta'}\no\\
&&+\int_{\eta}^{L}\frac{(S+\bar\v)(\eta',R\phi'(\eta'))}{\sin(\phi'(\eta'))}\exp(-G_{\eta',\eta})\ud{\eta'}\nonumber.
\end{eqnarray}
Then taking $\eta$ derivative on both sides of (\ref{mt 03}) yields
\begin{eqnarray}
\frac{\p\v}{\p\eta}&=&JJ=JJ_1+JJ_2+JJ_3+JJ_4+JJ_5+JJ_6+JJ_7,
\end{eqnarray}
where
\begin{eqnarray}
JJ_1&=&-\frac{(S+\bar\v)(\eta,\phi)}{\sin\phi},\\
JJ_2&=&\int_{\eta}^{L}\frac{(S+\bar\v)(\eta',\phi'(\eta'))}{\sin\phi'(\eta')}
\exp(-G_{\eta',\eta})\\
&&\bigg(\frac{\cos\phi'(\eta')}{\sin\phi'(\eta')}\frac{\p\phi'(\eta')}{\p\eta}
-\int_{\eta}^{\eta'}\frac{\cos\phi'(\xi)}{\sin\phi'^2(\xi)}\frac{\p\phi'(\xi)}{\p\eta}\ud{\xi}\bigg)\ud{\eta'},\no\\
JJ_3&=&\int_{\eta}^{L}
\frac{\p S(\eta',\phi'(\eta'))}{\p\eta}\frac{1}{\sin\phi'(\eta')}\exp(-G_{\eta',\eta})\ud{\eta'},\\
JJ_4&=&\frac{\p p(\phi'(0))}{\p\eta}\exp(-G_{L,0}-G_{L,\eta}),\\
JJ_5&=&p(\phi'(0))\exp(-G_{L,0}-G_{L,\eta})\bigg(-\int^{L}_{0}\frac{\cos\phi'(\xi)}{\sin\phi'^2(\xi)}\frac{\p\phi'(\xi)}{\p\eta}\ud{\xi}
-\int^{L}_{\eta}\frac{\cos\phi'(\xi)}{\sin\phi'^2(\xi)}\frac{\p\phi'(\xi)}{\p\eta}\ud{\xi}\bigg),\\
JJ_6&=&\int_{0}^{L}\frac{(S+\bar\v)(\eta',\phi'(\eta'))}{\sin\phi'(\eta')}
\exp(-G_{L,\eta'}-G_{L,\eta})\\
&&\bigg(\frac{\cos\phi'(\eta')}{\sin\phi'(\eta')}\frac{\p\phi'(\eta')}{\p\eta}
-\int^{L}_{\eta'}\frac{\cos\phi'(\xi)}{\sin\phi'^2(\xi)}\frac{\p\phi'(\xi)}{\p\eta}\ud{\xi}
-\int^{L}_{\eta}\frac{\cos\phi'(\xi)}{\sin\phi'^2(\xi)}\frac{\p\phi'(\xi)}{\p\eta}\ud{\xi}\bigg)\ud{\eta'},\no\\
JJ_7&=&\int_{0}^{L}
\frac{\p S(\eta',\phi'(\eta'))}{\p\eta}\frac{1}{\sin\phi'(\eta')}\exp(-G_{L,\eta'}-G_{L,\eta})\ud{\eta'}.
\end{eqnarray}
We divide the proof into several steps:\\
\ \\
Step 1: Estimate of $JJ_1$.\\
This is obvious, since $S$ and $\bar\v$ are bounded.\\
\ \\
Step 2: Estimate of $JJ_2$.\\
We may directly compute
\begin{eqnarray}
\frac{\p\phi'(\eta')}{\p\eta}=-\frac{F(\eta)\cos\phi'(\eta')}{\sin\phi'(\eta')}
\end{eqnarray}
Hence, we have
\begin{eqnarray}
\frac{\cos\phi'(\eta')}{\sin^2\phi'(\eta')}\frac{\p\phi'(\eta')}{\p\eta}=
-\frac{F(\eta)\cos^2\phi'(\eta')}{\sin^3\phi'(\eta')}
\end{eqnarray}
Since $\sin\phi'(\eta')\geq\d_0$, we know
\begin{eqnarray}
\abs{\frac{\cos\phi'(\eta')}{\sin^2\phi'(\eta')}\frac{\p\phi'(\eta')}{\p\eta}}\leq \frac{\e}{\d_0^3}
\end{eqnarray}
Thus, we obtain
\begin{eqnarray}
\abs{\int_{\eta}^{\eta'}\frac{\cos\phi'(\xi)}{\sin\phi'^2(\xi)}\frac{\p\phi'(\xi)}{\p\eta}\ud{\xi}}\leq \frac{\e(\eta'-\eta)}{\d_0^3}.
\end{eqnarray}
Also, it is easy to see
\begin{eqnarray}
\exp(-G_{\eta',\eta})\leq \exp\bigg(-(\eta'-\eta)\bigg).
\end{eqnarray}
Since
\begin{eqnarray}
\sin\phi\geq \sin\phi'(\xi)\geq C\d_0,
\end{eqnarray}
we directly obtain
\begin{eqnarray}
\abs{JJ_2}\leq \frac{C\e}{\d_0^3}
\end{eqnarray}
\ \\
Step 3: Estimate of $JJ_3$.\\
We compute
\begin{eqnarray}
\frac{\p S(\eta',\phi'(\eta'))}{\p\eta}=\p_2S\frac{\p\phi'(\eta')}{\p\eta}=-\p_2S\frac{F(\eta)\cos\phi'(\eta')}{\sin\phi'(\eta')},
\end{eqnarray}
where $\p_2$ denotes derivative with respect to the second argument in $S(\eta,\phi)$. Hence, this implies
\begin{eqnarray}
\abs{\frac{\p S(\eta',\phi'(\eta'))}{\p\eta}}\leq \frac{\e}{\d_0}.
\end{eqnarray}
Thus, we have
\begin{eqnarray}
\abs{JJ_3}\leq \frac{\e}{\d_0}.
\end{eqnarray}
\ \\
Step 4: Estimate of $JJ_4$.\\
Since
\begin{eqnarray}
\ue^{-V(0)}\cos\phi'(0)=\ue^{-V(\eta)}\cos\phi,
\end{eqnarray}
we have
\begin{eqnarray}
\frac{\p\phi'(0)}{\p\eta}=\frac{\cos\phi'(0)F(\eta)}{\sin\phi'(0)}
\end{eqnarray}
Thus, we know
\begin{eqnarray}
\abs{\frac{\p p(\phi'(0))}{\p\eta}}\leq \nm{p}_{W^{1,\infty}}\abs{\frac{\p\phi'(0)}{\p\eta}}\leq \frac{C\e}{\d_0}.
\end{eqnarray}
Hence, we have
\begin{eqnarray}
\abs{JJ_4}\leq \frac{C\e}{\d_0}.
\end{eqnarray}
\ \\
Step 5: Estimate of $JJ_5$.\\
This is basically identical to the estimate of $JJ_2$. We have
\begin{eqnarray}
\abs{JJ_5}\leq \frac{C\e}{\d_0^3}
\end{eqnarray}
\ \\
Step 6: Estimate of $JJ_6$.\\
This is basically identical to the estimate of $JJ_2$. We have
\begin{eqnarray}
\abs{JJ_6}\leq \frac{C\e}{\d_0^3}
\end{eqnarray}
\ \\
Step 7: Estimate of $JJ_7$.\\
This is basically identical to the estimate of $JJ_2$. We have
\begin{eqnarray}
\abs{JJ_7}\leq \frac{C\e}{\d_0^3}
\end{eqnarray}
\ \\
Step 8: Synthesis.\\
Summarizing all above, we have
\begin{eqnarray}
\abs{\sin\phi\frac{\p\v}{\p\eta}(\eta,\phi)}\leq C\bigg(1+\frac{\e}{\d_0^3}\bigg).
\end{eqnarray}
\end{proof}

\begin{lemma}\label{pt lemma 6}
Assume $\lnnm{S}+\lnnm{\dfrac{\p S}{\p\eta}}\leq C$ and $\abs{\dfrac{\p\bar\v}{\p\eta}}\leq C(1+\abs{\ln(\e)}+\abs{\ln(\eta)})$. For $\sin\phi\leq0$ and $\abs{E(\eta,\phi)}\geq \ue^{-V(L)}$, we have
\begin{eqnarray}
\abs{\sin\phi\frac{\p\v}{\p\eta}(\eta,\phi)}\leq C(1+\abs{\ln(\e)}).
\end{eqnarray}
\end{lemma}
\begin{proof}
We use Formulation II to rewrite the equation as
\begin{eqnarray}\label{mt 02}
\v(\eta,\phi)&=&p(\phi_{\ast}(0))\exp(-H_{\phi,\phi_{\ast}(0)})+
\int_{\phi_{\ast}(0)}^{\phi}\frac{(S+\bar\v)(\eta_{\ast}(\phi_{\ast}),\phi_{\ast})}{F(\eta_{\ast}(\phi_{\ast}))\cos\phi_{\ast}}
\exp(-H_{\phi,\phi_{\ast}})\ud{\phi_{\ast}}
\end{eqnarray}
where
\begin{eqnarray}
H_{\phi,\phi_{\ast}}&=&\int_{\phi_{\ast}}^{\phi}\frac{1}{F(\eta_{\ast}(\varpi))\cos\varpi}\ud{\varpi}.
\end{eqnarray}
with $(\eta_{\ast}(\phi_{\ast}),\phi_{\ast})$, $(0,\phi_{\ast}(0))$ and $(\eta,\phi)$ are on the same characteristics. Then taking $\eta$ derivative on both sides of (\ref{mt 02}) to obtain
\begin{eqnarray}
\frac{\p\v}{\p\eta}(\eta,\phi)&=&JJJ=JJJ_1+JJJ_2+JJJ_3+JJJ_4+JJJ_5,
\end{eqnarray}
where
\begin{eqnarray}
JJJ_1&=&\frac{\p p(\phi_{\ast}(0))}{\p\eta}\exp(-H_{\phi,\phi_{\ast}(0)})\\
JJJ_2&=&p(\phi_{\ast}(0))\exp(-H_{\phi,\phi_{\ast}(0)})\bigg(-\frac{\p \phi_{\ast}(0)}{\p\eta}\frac{1}{F(0)\cos\phi_{\ast}(0)}-
\int_{\phi_{\ast}(0)}^{\phi}\frac{F'(\eta_{\ast}(\varpi))}{F^2(\eta_{\ast}(\varpi))\cos\varpi}\frac{\p\eta_{\ast}(\varpi)}{\p\eta}\ud{\varpi}\bigg)\\
JJJ_3&=&-\frac{(S+\bar\v)(0,\phi_{\ast}(0))}{F(0)\cos\phi_{\ast}(0)}
\exp(-H_{\phi,\phi_{\ast}(0)})\frac{\p \phi_{\ast}(0)}{\p\eta}\\
JJJ_4&=&\int_{\phi_{\ast}(0)}^{\phi}\frac{(S+\bar\v)(\eta_{\ast}(\phi_{\ast}),\phi_{\ast})}{F(\eta_{\ast}(\phi_{\ast}))\cos\phi_{\ast}}
\exp(-H_{\phi,\phi_{\ast}})\\
&&\bigg(-\frac{F'(\eta_{\ast}(\phi_{\ast})}{F^2(\eta_{\ast}(\phi_{\ast}))\cos\phi_{\ast}}\frac{\p\eta_{\ast}(\phi_{\ast})}{\p\eta}
-\int_{\phi_{\ast}}^{\phi}\frac{F'(\eta_{\ast}(\varpi))}{F^2(\eta_{\ast}(\varpi))\cos\varpi}\frac{\p\eta_{\ast}(\varpi)}{\p\eta}\ud{\varpi}\bigg)\ud{\phi_{\ast}},\no\\
JJJ_5&=&\int_{\phi_{\ast}(0)}^{\phi}\frac{\p (S+\bar\v)(\eta_{\ast}(\phi_{\ast}),\phi_{\ast})}{\p\eta}\frac{1}{F(\eta_{\ast}(\phi_{\ast}))\cos\phi_{\ast}}
\exp(-H_{\phi,\phi_{\ast}})\ud{\phi_{\ast}}.
\end{eqnarray}
Then we divide the proof into several steps:\\
\ \\
Step 1: Estimate of $JJJ_1$.\\
Since we know
\begin{eqnarray}\label{mt 31}
\ue^{-V(0)}\cos\phi_{\ast}(0)=\ue^{-V(\eta)}\cos\phi,
\end{eqnarray}
taking $\eta$ derivative on both sides of (\ref{mt 31}) implies
\begin{eqnarray}
\frac{\p\phi_{\ast}(0)}{\p\eta}=-\frac{\ue^{V(0)-V(\eta)}F(\eta)\cos\phi}{\sin\phi_{\ast}(0)}.
\end{eqnarray}
which further yields
\begin{eqnarray}
\abs{\frac{\p\phi_{\ast}(0)}{\p\eta}}\leq\abs{\frac{\ue^{V(0)-V(\eta)}F(\eta)\cos\phi}{\sin\phi}}\leq\abs{\frac{C\e}{\sin\phi}},
\end{eqnarray}
due to
\begin{eqnarray}
\abs{\sin\phi_{\ast}(0)}\geq\abs{\sin\phi}.
\end{eqnarray}
Hence, we have
\begin{eqnarray}\label{mt 32}
\abs{\frac{\p p(\phi_{\ast}(0))}{\p\eta}}\leq\nm{p}_{W^{1,\infty}}\abs{\frac{\p\phi_{\ast}(0)}{\p\eta}}
\leq\abs{\frac{C\e}{\sin\phi}}.
\end{eqnarray}
Also, for $\sin\phi\leq\dfrac{1}{2}$ and $\eta\in[0,L]$, we always have
\begin{eqnarray}
-H_{\phi,\phi_{\ast}(0)}=-\int_{\phi_{\ast}(0)}^{\phi}\frac{1}{F(\eta_{\ast}(\varpi))\cos\varpi}\ud{\varpi}\leq0,
\end{eqnarray}
which further yields
\begin{eqnarray}\label{mt 33}
\exp(-H_{\phi,\phi_{\ast}(0)})\leq1.
\end{eqnarray}
Then combining (\ref{mt 32}) and (\ref{mt 33}), we have
\begin{eqnarray}
\abs{JJJ_1}\leq \abs{\frac{C\e}{\sin\phi}}.
\end{eqnarray}
\ \\
Step 2: Estimate of $JJJ_2$.\\
Based on the results in Step 1, we can easily verify
\begin{eqnarray}
\abs{p(\phi_{\ast}(0))}&\leq&C,\\
\abs{\exp(-H_{\phi,\phi_{\ast}(0)})}&\leq&1,\\
\abs{-\frac{\p \phi_{\ast}(0)}{\p\eta}}&\leq&\abs{\frac{C\e}{\sin\phi}}.
\end{eqnarray}
Then since
\begin{eqnarray}
\ue^{-V(\eta_{\ast}(\phi_{\ast}))}\cos\phi_{\ast}=\ue^{-V(\eta)}\cos\phi,
\end{eqnarray}
taking $\eta$ derivative on both sides implies
\begin{eqnarray}
\frac{\p\eta_{\ast}(\phi_{\ast})}{\p\eta}=\frac{F(\eta)\cos\phi}{F(\eta_{\ast}(\phi_{\ast}))\cos\phi_{\ast}}\ue^{V(\eta_{\ast}(\phi_{\ast}))-V(\eta)}
=\frac{F(\eta)}{F(\eta_{\ast}(\phi_{\ast}))}.
\end{eqnarray}
Considering $0\leq\eta^{\ast}\leq L$, we may directly obtain
\begin{eqnarray}
\abs{\frac{F(\eta)}{F(\eta_{\ast}(\phi_{\ast}))}}\leq C,
\end{eqnarray}
which further leads to
\begin{eqnarray}
\abs{\frac{\p\eta_{\ast}(\phi_{\ast})}{\p\eta}}\leq C.
\end{eqnarray}
On the other hand, for $\sin\phi\leq 0$ with $\abs{E(\eta,\phi)}\geq \ue^{-V(L)}$, we know
\begin{eqnarray}
\abs{\ue^{-V(\eta_{\ast}(\phi_{\ast}))}\cos\phi_{\ast}}=\abs{\ue^{-V(\eta)}\cos\phi}\geq \ue^{-V(L)},
\end{eqnarray}
which implies
\begin{eqnarray}
\abs{\cos\phi_{\ast}}\geq \ue^{V(\eta_{\ast}(\phi_{\ast}))-V(L)}\geq\ue^{V(0)-V(L)}\geq C_0>0.
\end{eqnarray}
In total, we have shown
\begin{eqnarray}\label{mt 34}
\abs{\cos\phi_{\ast}}\geq C_0>0,
\end{eqnarray}
which naturally yields
\begin{eqnarray}
\abs{\frac{1}{F(0)\cos\phi_{\ast}(0)}}\leq \frac{C}{\e}.
\end{eqnarray}
Also, we have
\begin{eqnarray}
\abs{\frac{F'(\eta_{\ast}(\varpi))}{F^2(\eta_{\ast}(\varpi))}}=1.
\end{eqnarray}
Hence, we have
\begin{eqnarray}
\abs{JJJ_2}\leq \abs{\frac{C}{\sin\phi}}.
\end{eqnarray}
\ \\
Step 3: Estimate of $JJJ_3$.\\
We may directly estimate
\begin{eqnarray}
\abs{JJJ_3}&\leq&\abs{-\frac{(S+\bar\v)(0,\phi_{\ast}(0))}{F(0)\cos\phi_{\ast}(0)}}\abs{\exp(-H_{\phi,\phi_{\ast}(0)})}\abs{\frac{\p \phi_{\ast}(0)}{\p\eta}}\leq\frac{C}{\e}\cdot 1\cdot\abs{\frac{C\e}{\sin\phi}}\leq\abs{\frac{C}{\sin\phi}},
\end{eqnarray}
based on the estimates from Step 1.
Therefore, we have proved
\begin{eqnarray}
\abs{JJJ_3}\leq \abs{\frac{C}{\sin\phi}}.
\end{eqnarray}
\ \\
Step 4: Estimate of $JJJ_4$.\\
Using estimates in Step 1 and Step 2, we have
\begin{eqnarray}
&&\abs{-\frac{F'(\eta_{\ast}(\phi_{\ast})}{F^2(\eta_{\ast}(\phi_{\ast}))\cos\phi_{\ast}}\frac{\p\eta_{\ast}(\phi_{\ast})}{\p\eta}
-\int_{\phi_{\ast}}^{\phi}\frac{F'(\eta_{\ast}(\varpi))}{F^2(\eta_{\ast}(\varpi))\cos\varpi}\frac{\p\eta_{\ast}(\varpi)}{\p\eta}\ud{\varpi}}\\
&=&\abs{\frac{1}{\cos\phi_{\ast}}\frac{\p\eta_{\ast}(\phi_{\ast})}{\p\eta}}
+\abs{\int_{\phi_{\ast}}^{\phi}\frac{1}{\cos\varpi}\frac{\p\eta_{\ast}(\varpi)}{\p\eta}\ud{\varpi}}\leq C.\no
\end{eqnarray}
Then we know
\begin{eqnarray}
\abs{JJJ_4}\leq C\abs{\int_{\phi_{\ast}(0)}^{\phi}\frac{1}{F(\eta_{\ast}(\phi_{\ast}))\cos\phi_{\ast}}
\exp(-H_{\phi,\phi_{\ast}})\ud{\phi_{\ast}}}=\abs{\exp(-H_{\phi,\phi_{\ast}(0)})-1}\leq C.
\end{eqnarray}
\ \\
Step 5: Estimate of $JJJ_5$.\\
We decompose $JJJ_5$ as
\begin{eqnarray}
JJJ_5&=&\int_{\phi_{\ast}(0)}^{\phi}\frac{\p S(\eta_{\ast}(\phi_{\ast}),\phi_{\ast})}{\p\eta}\frac{1}{F(\eta_{\ast}(\phi_{\ast}))\cos\phi_{\ast}}
\exp(-H_{\phi,\phi_{\ast}})\ud{\phi_{\ast}}\\
&&+\int_{\phi_{\ast}(0)}^{\phi}\frac{\p \bar\v(\eta_{\ast}(\phi_{\ast}))}{\p\eta}\frac{1}{F(\eta_{\ast}(\phi_{\ast}))\cos\phi_{\ast}}
\exp(-H_{\phi,\phi_{\ast}})\ud{\phi_{\ast}}\no\\
&=&JJJ_{5,1}+JJJ_{5,2}.\no
\end{eqnarray}
We may direct estimate
\begin{eqnarray}
\abs{JJJ_{5,1}}&\leq&C\abs{\int_{\phi_{\ast}(0)}^{\phi}\frac{1}{F(\eta_{\ast}(\phi_{\ast}))\cos\phi_{\ast}}
\exp(-H_{\phi,\phi_{\ast}})\ud{\phi_{\ast}}}=C\abs{\exp(-H_{\phi,\phi_{\ast}(0)})-1}\leq C.
\end{eqnarray}
Note that we cannot estimate $JJJ_{5,2}$ as above since $\dfrac{\p \bar\v(\eta_{\ast}(\phi_{\ast}))}{\p\eta}$ involves derivative of $\bar\v$ in the normal direction, which might contain singularity when approaching the boundary. Fortunately, this term lies in the integral along the characteristics, so we may substitute the principle variable $\phi_{\ast}\rt\eta_{\ast}$ back into formulation I as $(\eta_{\ast},\phi_{\ast}(\eta_{\ast})))$. This implies the substitution for derivative
\begin{eqnarray}
\frac{\p}{\p\eta}\rt\frac{\p}{\p\eta}+\frac{\p}{\p\eta_{\ast}}\frac{\p\eta_{\ast}}{\p\eta}=\frac{\p}{\p\eta}+\frac{\p}{\p\eta_{\ast}}\frac{F(\eta)}{F(\eta_{\ast})}.
\end{eqnarray}
Hence, for $\eta^+$ denoting the intersection of characteristics with $\sin\phi=0$, we know
\begin{eqnarray}
JJJ_{5,2}&=&\int_{0}^{\eta^+}\bigg(\frac{\p\bar\v(\eta_{\ast})}{\p\eta}+\frac{\p\bar\v(\eta_{\ast})}{\p\eta_{\ast}}
\frac{F(\eta)}{F(\eta_{\ast})}\bigg)\frac{1}{F(\eta_{\ast})\cos(\phi_{\ast}(\eta_{\ast}))}
\exp(-H_{\phi,\phi_{\ast}})\frac{\p\phi_{\ast}}{\p\eta_{\ast}}\ud{\eta_{\ast}}\\
&&+\int_{\eta}^{\eta^+}\bigg(\frac{\p\bar\v(\eta_{\ast})}{\p\eta}+\frac{\p\bar\v(\eta_{\ast})}{\p\eta_{\ast}}
\frac{F(\eta)}{F(\eta_{\ast})}\bigg)\frac{1}{F(\eta_{\ast})\cos(\phi_{\ast}(\eta_{\ast}))}
\exp(-H_{\phi,\phi_{\ast}})\frac{\p\phi_{\ast}}{\p\eta_{\ast}}\ud{\eta_{\ast}}\no\\
&=&\int_{0}^{\eta^+}\frac{\p\bar\v(\eta_{\ast})}{\p\eta_{\ast}}
\frac{F(\eta)}{F(\eta_{\ast})}\frac{1}{F(\eta_{\ast})\cos(\phi_{\ast}(\eta_{\ast}))}
\exp(-H_{\phi,\phi_{\ast}})\frac{\p\phi_{\ast}}{\p\eta_{\ast}}\ud{\eta_{\ast}}\no\\
&&+\int_{\eta}^{\eta^+}\frac{\p\bar\v(\eta_{\ast})}{\p\eta_{\ast}}
\frac{F(\eta)}{F(\eta_{\ast})}\frac{1}{F(\eta_{\ast})\cos(\phi_{\ast}(\eta_{\ast}))}
\exp(-H_{\phi,\phi_{\ast}})\frac{\p\phi_{\ast}}{\p\eta_{\ast}}\ud{\eta_{\ast}}\no\\
&=&\int_{0}^{\eta^+}\frac{\p\bar\v(\eta_{\ast})}{\p\eta_{\ast}}
\frac{F(\eta)}{F(\eta_{\ast})}\frac{1}{\sin(\phi_{\ast}(\eta_{\ast}))}
\exp(-H_{\phi,\phi_{\ast}})\ud{\eta_{\ast}},\no\\
&&\int_{\eta}^{\eta^+}\frac{\p\bar\v(\eta_{\ast})}{\p\eta_{\ast}}
\frac{F(\eta)}{F(\eta_{\ast})}\frac{1}{\sin(\phi_{\ast}(\eta_{\ast}))}
\exp(-H_{\phi,\phi_{\ast}})\ud{\eta_{\ast}},\no
\end{eqnarray}
where
\begin{eqnarray}
\frac{\p\phi_{\ast}}{\p\eta_{\ast}}=\frac{F(\eta_{\ast})\cos(\phi_{\ast}(\eta_{\ast}))}{\sin(\phi_{\ast}(\eta_{\ast}))}.
\end{eqnarray}
Since
\begin{eqnarray}
\abs{\frac{\p\bar\v}{\p\eta}}\leq C(1+\abs{\ln(\e)}+\abs{\ln(\eta)}),
\end{eqnarray}
and $\sin\phi_{\ast}\sim \sqrt{\e(\eta_{\ast}-\eta^+)}$, we know above integral is finite, i.e.
\begin{eqnarray}
\abs{JJJ_{5,2}}\leq C(1+\abs{\ln(\e)}).
\end{eqnarray}
Therefore, we know
\begin{eqnarray}
\abs{JJJ_5}\leq C(1+\abs{\ln(\e)}).
\end{eqnarray}
\ \\
Step 6: Synthesis.\\
In summary, we have shown
\begin{eqnarray}
\abs{JJJ}\leq C(1+\abs{\ln(\e)})+\abs{\frac{C}{\sin\phi}}.
\end{eqnarray}
which implies
\begin{eqnarray}
\abs{\sin\phi\frac{\p\v}{\p\eta}(\eta,\phi)}\leq C(1+\abs{\ln(\e)}).
\end{eqnarray}
\end{proof}

\begin{remark}
Estimates in Lemma \ref{pt lemma 1}, Lemma \ref{pt lemma 2} and Lemma \ref{pt lemma 6} can provide pointwise bounds of derivatives. However, they are not uniform estimates due to presence of $\d_0$ and $\ln(\eta)$. We need weighted $L^{\infty}$ estimates of derivatives to close the proof.
\end{remark}

\subsection{Mild Formulation of Normal Derivative}

Consider the $\e$-transport problem for $\a=\zeta\dfrac{\p\v}{\p\eta}$ as
\begin{eqnarray}\label{Milne infinite problem LI}
\left\{
\begin{array}{rcl}\displaystyle
\sin\phi\frac{\p\a}{\p\eta}+F(\eta)\cos\phi\frac{\p
\a}{\p\phi}+\a&=&\tilde\a+S_{\a},\\
\a(0,\phi)&=&p_{\a}(\phi)\ \ \text{for}\ \ \sin\phi>0,\\
\a(L,\phi)&=&\a(L,R\phi),
\end{array}
\right.
\end{eqnarray}
where $p_{\a}$ and $S_{\a}$ will be specified later with
\begin{eqnarray}
\tilde\a(\eta,\phi)=\frac{1}{2\pi}\int_{-\pi}^{\pi}\frac{\zeta(\eta,\phi)}{\zeta(\eta,\phi_{\ast})}\a(\eta,\phi_{\ast})\ud{\phi_{\ast}}.
\end{eqnarray}

\begin{lemma}\label{pt lemma 3}
We have
\begin{eqnarray}
\lnnm{\a}&\leq&C\bigg(\lnm{p_{\a}}+\lnnm{S_{\a}}\bigg)\\
&&+C\abs{\ln(\e)}^8\bigg(\lnnm{\v}+\lnm{\frac{\p p}{\p\phi}}+\lnnm{S}+\lnnm{\frac{\p S}{\p\phi}}+\lnnm{\frac{\p S}{\p\eta}}\bigg).\no
\end{eqnarray}
\end{lemma}
The rest of this subsection will be devoted to the proof of this lemma. We first introduce some notation.
Define the energy as before
\begin{eqnarray}
E(\eta,\phi)=\ue^{-V(\eta)}\cos\phi .
\end{eqnarray}
Along the characteristics, where this energy is conserved and $\zeta$ is a constant, the equation can be simplified as follows:
\begin{eqnarray}
\sin\phi\frac{\ud{\a}}{\ud{\eta}}+\a=\tilde\a+S_{\a}.
\end{eqnarray}
An implicit function
$\eta^+(\eta,\phi)$ can be determined through
\begin{eqnarray}
\abs{E(\eta,\phi)}=\ue^{-V(\eta^+)}.
\end{eqnarray}
which means $(\eta^+,\phi_0)$ with $\sin\phi_0=0$ is on the same characteristics as $(\eta,\phi)$.
Define the quantities for $0\leq\eta'\leq\eta^+$ as follows:
\begin{eqnarray}
\phi'(\phi,\eta,\eta')&=&\cos^{-1}(\ue^{V(\eta')-V(\eta)}\cos\phi ),\\
R\phi'(\phi,\eta,\eta')&=&-\cos^{-1}(\ue^{V(\eta')-V(\eta)}\cos\phi )=-\phi'(\phi,\eta,\eta'),
\end{eqnarray}
where the inverse trigonometric function can be defined
single-valued in the domain $[0,\pi)$ and the quantities are always well-defined due to the monotonicity of $V$. Note that $\sin\phi'\geq0$, even if $\sin\phi<0$.
Finally we put
\begin{eqnarray}
G_{\eta,\eta'}(\phi)&=&\int_{\eta'}^{\eta}\frac{1}{\sin(\phi'(\phi,\eta,\xi))}\ud{\xi}.
\end{eqnarray}
Similar to $\e$-Milne problem, we can define the solution along the characteristics as follows:
\begin{eqnarray}
\a(\eta,\phi)=\k[p_{\a}]+\t[\tilde\a+S_{\a}],
\end{eqnarray}
where\\
\ \\
Region I:\\
For $\sin\phi>0$,
\begin{eqnarray}
\k[p_{\a}]&=&p_{\a}(\phi'(0))\exp(-G_{\eta,0})\\
\t[\tilde\a+S_{\a}]&=&\int_0^{\eta}\frac{(\tilde\a+S_{\a})(\eta',\phi'(\eta'))}{\sin(\phi'(\eta'))}\exp(-G_{\eta,\eta'})\ud{\eta'}.
\end{eqnarray}
\ \\
Region II:\\
For $\sin\phi<0$ and $\abs{E(\eta,\phi)}\leq \ue^{-V(L)}$,
\begin{eqnarray}
\k[p_{\a}]&=&p_{\a}(\phi'(0))\exp(-G_{L,0}-G_{L,\eta})\\
\t[\tilde\a+S_{\a}]&=&\int_0^{L}\frac{(\tilde\a+S)(\eta',\phi'(\eta'))}{\sin(\phi'(\eta'))}
\exp(-G_{L,\eta'}-G_{L,\eta})\ud{\eta'}\\
&&+\int_{\eta}^{L}\frac{(\tilde\a+S)(\eta',R\phi'(\eta'))}{\sin(\phi'(\eta'))}\exp(-G_{\eta',\eta})\ud{\eta'}\nonumber.
\end{eqnarray}
\ \\
Region III:\\
For $\sin\phi<0$ and $\abs{E(\eta,\phi)}\geq \ue^{-V(L)}$,
\begin{eqnarray}
\k[p_{\a}]&=&p_{\a}(\phi'(\phi,\eta,0))\exp(-G_{\eta^+,0}-G_{\eta^+,\eta})\\
\t[\tilde\a+S_{\a}]&=&\int_0^{\eta^+}\frac{(\tilde\a+S_{\a})(\eta',\phi'(\eta'))}{\sin(\phi'(\eta'))}
\exp(-G_{\eta^+,\eta'}-G_{\eta^+,\eta})\ud{\eta'}\\
&&+
\int_{\eta}^{\eta^+}\frac{(\tilde\a+S_{\a})(\eta',R\phi'(\eta'))}{\sin(\phi'(\eta'))}\exp(-G_{\eta',\eta})\ud{\eta'}\nonumber.
\end{eqnarray}
Then we need to estimate $\k[p_{\a}]$ and $\t[\tilde\a+S_{\a}]$ in each case. We assume $0<\d<<1$ and $0<\d_0<<1$ are small quantities which will be determined later.

\subsubsection{Region I: $\sin\phi>0$}

Based on \cite[Lemma 4.7, Lemma 4.8]{AA003},
we can directly obtain
\begin{eqnarray}
\abs{\k[p_{\a}]}&\leq&\lnm{p_{\a}},\\
\abs{\t[S_{\a}]}&\leq&\lnnm{S_{\a}}.
\end{eqnarray}
Hence, we only need to estimate $I=\t[\tilde\a]$. We divide it into several steps:\\
\ \\
Step 0: Preliminaries.\\
We have
\begin{eqnarray}
E(\eta',\phi')=\frac{\rk-\e\eta'}{\rk}\cos\phi'.
\end{eqnarray}
We can directly obtain
\begin{eqnarray}\label{pt 01}
\zeta(\eta',\phi')&=&\frac{1}{\rk}\sqrt{\rk^2-\bigg((\rk-\e\eta')\cos\phi'\bigg)^2}
=\frac{1}{\rk}\sqrt{\rk^2-(\rk-\e\eta')^2+(\rk-\e\eta')^2\sin^2\phi'},\\
&\leq& \sqrt{\rk^2-(\rk-\e\eta')^2}+\sqrt{(\rk-\e\eta')^2\sin^2\phi'}\leq C\bigg(\sqrt{\e\eta'}+\sin\phi'\bigg),\no
\end{eqnarray}
and
\begin{eqnarray}\label{pt 02}
\zeta(\eta',\phi')\geq\frac{1}{\rk}\sqrt{\rk^2-(\rk-\e\eta')^2}\geq C\sqrt{\e\eta'}.
\end{eqnarray}
Also, we know for $0\leq\eta'\leq\eta$,
\begin{eqnarray}
\sin\phi'&=&\sqrt{1-\cos^2\phi'}=\sqrt{1-\bigg(\frac{\rk-\e\eta}{\rk-\e\eta'}\bigg)^2\cos^2\phi}\\
&=&\frac{\sqrt{(\rk-\e\eta')^2\sin^2\phi+(2\rk-\e\eta-\e\eta')(\e\eta-\e\eta')\cos^2\phi}}{\rk-\e\eta'}.
\end{eqnarray}
Since
\begin{eqnarray}
0\leq(2\rk-\e\eta-\e\eta')(\e\eta-\e\eta')\cos^2\phi\leq 2\rk\e(\eta-\eta'),
\end{eqnarray}
we have
\begin{eqnarray}
\sin\phi\leq\sin\phi'
\leq2\sqrt{\sin^2\phi+\e(\eta-\eta')},
\end{eqnarray}
which means
\begin{eqnarray}
\frac{1}{2\sqrt{\sin^2\phi+\e(\eta-\eta')}}\leq\frac{1}{\sin\phi'}
\leq\frac{1}{\sin\phi}.
\end{eqnarray}
Therefore,
\begin{eqnarray}\label{pt 03}
-\int_{\eta'}^{\eta}\frac{1}{\sin\phi'(y)}\ud{y}&\leq& -\int_{\eta'}^{\eta}\frac{1}{2\sqrt{\sin^2\phi+\e(\eta-y)}}\ud{y}\\
&=&\frac{1}{\e}\bigg(\sin\phi-\sqrt{\sin^2\phi+\e(\eta-\eta')}\bigg)\no\\
&=&-\frac{\eta-\eta'}{\sin\phi+\sqrt{\sin^2\phi+\e(\eta-\eta')}}\no\\
&\leq&-\frac{\eta-\eta'}{2\sqrt{\sin^2\phi+\e(\eta-\eta')}}.\no
\end{eqnarray}
Define a cut-off function $\chi\in C^{\infty}[-\pi,\pi]$ satisfying
\begin{eqnarray}
\chi(\phi)=\left\{
\begin{array}{ll}
1&\text{for}\ \ \abs{\sin\phi}\leq\d,\\
0&\text{for}\ \ \abs{\sin\phi}\geq2\d,
\end{array}
\right.
\end{eqnarray}
In the following, we will divide the estimate of $I$ into several cases based on the value of $\sin\phi$, $\sin\phi'$, $\e\eta'$ and $\e(\eta-\eta')$. Let $\id$ denote the indicator function. We write
\begin{eqnarray}
I&=&\int_0^{\eta}\id_{\{\sin\phi\geq\d_0\}}+\int_0^{\eta}\id_{\{0\leq\sin\phi\leq\d_0\}}\id_{\{\chi(\phi_{\ast})<1\}}
+\int_0^{\eta}\id_{\{0\leq\sin\phi\leq\d_0\}}\id_{\{\chi(\phi_{\ast})=1\}}\id_{\{ \sqrt{\e\eta'}\geq\sin\phi'\}}\\
&&+\int_0^{\eta}\id_{\{0\leq\sin\phi\leq\d_0\}}\id_{\{\chi(\phi_{\ast})=1\}}\id_{\{\sqrt{\e\eta'}\leq\sin\phi'\}}\id_{\{\sin^2\phi\leq\e(\eta-\eta')\}}\no\\
&&+\int_0^{\eta}\id_{\{0\leq\sin\phi\leq\d_0\}}\id_{\{\chi(\phi_{\ast})=1\}}\id_{\{\sqrt{\e\eta'}\leq\sin\phi'\}}\id_{\{\sin^2\phi\geq\e(\eta-\eta')\}}\no\\
&=&I_1+I_2+I_3+I_4+I_5.\no
\end{eqnarray}
\ \\
Step 1: Estimate of $I_1$ for $\sin\phi\geq\d_0$.\\
Based on Lemma \ref{pt lemma 1}, we know
\begin{eqnarray}
\abs{\sin\phi\frac{\p\v}{\p\eta}}\leq C\bigg(1+\frac{1}{\d_0^3}\bigg)\bigg(\lnnm{\v}+\lnm{\frac{\p h}{\p\phi}}+\lnnm{\frac{\p S}{\p\phi}}+\lnnm{\frac{\p S}{\p\eta}}\bigg).
\end{eqnarray}
Hence, we have
\begin{eqnarray}
\abs{I_1}\leq C\abs{\frac{\p\v}{\p\eta}}\leq \frac{C}{\d_0^4}\bigg(\lnnm{\v}+\lnm{\frac{\p h}{\p\phi}}+\lnnm{\frac{\p S}{\p\phi}}+\lnnm{\frac{\p S}{\p\eta}}\bigg).
\end{eqnarray}
\ \\
Step 2: Estimate of $I_2$ for $0\leq\sin\phi\leq\d_0$ and $\chi(\phi_{\ast})<1$.\\
We have
\begin{eqnarray}
I_2&=&\frac{1}{2\pi}\int_0^{\eta}\bigg(\int_{-\pi}^{\pi}\frac{\zeta(\eta',\phi')}{\zeta(\eta',\phi_{\ast})}(1-\chi(\phi_{\ast}))
\a(\eta',\phi_{\ast})\ud{\phi_{\ast}}\bigg)
\frac{1}{\sin\phi'}\exp(-G_{\eta,\eta'})\ud{\eta'}\\
&=&\frac{1}{2\pi}\int_0^{\eta}\bigg(\int_{-\pi}^{\pi}\zeta(\eta',\phi')(1-\chi(\phi_{\ast}))
\frac{\v(\eta',\phi_{\ast})}{\p\eta'}\ud{\phi_{\ast}}\bigg)\frac{1}{\sin\phi'}\exp(-G_{\eta,\eta'})\ud{\eta'}.\no
\end{eqnarray}
Based on the $\e$-Milne problem of $\v$ as
\begin{eqnarray}
\sin\phi_{\ast}\frac{\p\v(\eta',\phi_{\ast})}{\p\eta'}+F(\eta')\cos\phi_{\ast}\frac{\p\v(\eta',\phi_{\ast})}{\p\phi_{\ast}}+\v(\eta',\phi_{\ast})-\bar\v(\eta')=S(\eta',\phi_{\ast}),
\end{eqnarray}
we have
\begin{eqnarray}
\frac{\p\v(\eta',\phi_{\ast})}{\p\eta'}=-\frac{1}{\sin\phi_{\ast}}
\bigg(F(\eta')\cos\phi_{\ast}\frac{\p\v(\eta',\phi_{\ast})}{\p\phi_{\ast}}+\v(\eta',\phi_{\ast})-\bar\v(\eta')-S(\eta',\phi_{\ast})\bigg)
\end{eqnarray}
Hence, we have
\begin{eqnarray}
\tilde\a&=&\int_{-\pi}^{\pi}\zeta(\eta',\phi')(1-\chi(\phi_{\ast}))
\frac{\p\v(\eta',\phi_{\ast})}{\p\eta'}\ud{\phi_{\ast}}\\
&=&-\int_{-\pi}^{\pi}\zeta(\eta',\phi')(1-\chi(\phi_{\ast}))
\frac{1}{\sin\phi_{\ast}}
\bigg(\v(\eta',\phi_{\ast})-\bar\v(\eta')-S(\eta',\phi_{\ast})\bigg)\ud{\phi_{\ast}}\no\\
&&-\int_{-\pi}^{\pi}\zeta(\eta',\phi')(1-\chi(\phi_{\ast}))
\frac{1}{\sin\phi_{\ast}}
F(\eta')\cos\phi_{\ast}\frac{\p\v(\eta',\phi_{\ast})}{\p\phi_{\ast}}\ud{\phi_{\ast}}\no\\
&=&\tilde\a_1+\tilde\a_2.\no
\end{eqnarray}
We may directly obtain
\begin{eqnarray}
\abs{\tilde\a_1}&\leq&\int_{-\pi}^{\pi}\zeta(\eta',\phi')(1-\chi(\phi_{\ast}))
\frac{1}{\sin\phi_{\ast}}
\bigg(\v(\eta',\phi_{\ast})-\bar\v(\eta')-S(\eta',\phi_{\ast})\bigg)\ud{\phi_{\ast}}\\
&\leq&\frac{\rk}{\d}\abs{\int_{-\pi}^{\pi}
\bigg(\v(\eta',\phi_{\ast})-\bar\v(\eta')-S(\eta',\phi_{\ast})\bigg)\ud{\phi_{\ast}}}\no\\
&\leq&C(\d)\bigg(\lnnm{\v}+\lnnm{S}\bigg).\no
\end{eqnarray}
On the other hand, an integration by parts yields
\begin{eqnarray}
\tilde\a_2&=&\int_{-\pi}^{\pi}\frac{\p}{\p\phi_{\ast}}\bigg(\zeta(\eta',\phi')(1-\chi(\phi_{\ast}))
\frac{1}{\sin\phi_{\ast}}
F(\eta')\cos\phi_{\ast}\bigg)\v(\eta',\phi_{\ast})\ud{\phi_{\ast}},
\end{eqnarray}
which further implies
\begin{eqnarray}
\abs{\tilde\a_2}&\leq&\frac{C\e}{\d^2}\lnnm{\v}\leq C(\d)\lnnm{\v}.
\end{eqnarray}
Since we can use substitution to show
\begin{eqnarray}
\int_0^{\eta}\frac{1}{\sin\phi'}\exp(-G_{\eta,\eta'})\ud{\eta'}\leq 1,
\end{eqnarray}
we have
\begin{eqnarray}
\abs{I_2}&\leq&C(\d)\bigg(\lnnm{\v}+\lnnm{S}\bigg)\int_0^{\eta}
\frac{1}{\sin\phi'}\exp(-G_{\eta,\eta'})\ud{\eta'}\\
&\leq&C(\d)\bigg(\lnnm{\v}+\lnnm{S}\bigg).\no
\end{eqnarray}
\ \\
Step 3: Estimate of $I_3$ for $0\leq\sin\phi\leq\d_0$, $\chi(\phi_{\ast})=1$ and $\sqrt{\e\eta'}\geq\sin\phi'$.\\
Based on (\ref{pt 01}), this implies
\begin{eqnarray}
\zeta(\eta',\phi')\leq C\sqrt{\e\eta'}.\no
\end{eqnarray}
Then combining this with (\ref{pt 02}), we can directly obtain
\begin{eqnarray}
\int_{-\pi}^{\pi}\frac{\zeta(\eta',\phi')}{\zeta(\eta',\phi_{\ast})}\chi(\phi_{\ast})
\a(\eta',\phi_{\ast})\ud{\phi_{\ast}}&\leq&C\int_{-\d}^{\d}
\a(\eta',\phi_{\ast})\ud{\phi_{\ast}}\leq C\d\lnnm{\a}.
\end{eqnarray}
Hence, we have
\begin{eqnarray}
\abs{I_3}&\leq&C\d\lnnm{\a}\int_0^{\eta}\frac{1}{\sin\phi'}\exp(-G_{\eta,\eta'})\ud{\eta'}\leq C\d\lnnm{\a}.
\end{eqnarray}
\ \\
Step 4: Estimate of $I_4$ for $0\leq\sin\phi\leq\d_0$, $\chi(\phi_{\ast})=1$, $\sqrt{\e\eta'}\leq\sin\phi'$ and $\sin^2\phi\leq\e(\eta-\eta')$.\\
Based on (\ref{pt 01}), this implies
\begin{eqnarray}
\zeta(\eta',\phi')\leq C\sin\phi'.
\end{eqnarray}
Based on (\ref{pt 03}), we have
\begin{eqnarray}
-G_{\eta,\eta'}=-\int_{\eta'}^{\eta}\frac{1}{\sin\phi'(y)}\ud{y}&\leq&-\frac{\eta-\eta'}{2\sqrt{\e(\eta-\eta')}}\leq-C\sqrt{\frac{\eta-\eta'}{\e}}.
\end{eqnarray}
Hence, we know
\begin{eqnarray}
\abs{I_4}&\leq&C\int_0^{\eta}\bigg(\int_{-\pi}^{\pi}\frac{\zeta(\eta',\phi')}{\zeta(\eta',\phi_{\ast})}\chi(\phi_{\ast})
\a(\eta',\phi_{\ast})\ud{\phi_{\ast}}\bigg)
\frac{1}{\sin\phi'}\exp(-G_{\eta,\eta'})\ud{\eta'}\\
&\leq&C\int_0^{\eta}\bigg(\int_{-\d}^{\d}\frac{1}{\zeta(\eta',\phi_{\ast})}
\a(\eta',\phi_{\ast})\ud{\phi_{\ast}}\bigg)
\frac{\zeta(\eta',\phi')}{\sin\phi'}\exp(-G_{\eta,\eta'})\ud{\eta'}\no\\
&\leq&C\d\lnnm{\a}\int_0^{\eta}\frac{1}{\sqrt{\e\eta'}}\exp(-G_{\eta,\eta'})\ud{\eta'}\no\\
&\leq&C\d\lnnm{\a}\int_0^{\eta}\frac{1}{\sqrt{\e\eta'}}\exp\bigg(-C\sqrt{\frac{\eta-\eta'}{\e}}\bigg)\ud{\eta'}\no
\end{eqnarray}
Define $z=\dfrac{\eta'}{\e}$, which implies $\ud{\eta'}=\e\ud{z}$. Substituting this into above integral, we have
\begin{eqnarray}
\abs{I_4}&\leq&C\d\lnnm{\a}\int_0^{\eta/\e}\frac{1}{\sqrt{z}}\exp\bigg(-C\sqrt{\frac{\eta}{\e}-z}\bigg)\ud{z}\\
&=&C\d\lnnm{\a}\Bigg(\int_0^{1}\frac{1}{\sqrt{z}}\exp\bigg(-C\sqrt{\frac{\eta}{\e}-z}\bigg)\ud{z}
+\int_1^{\eta/\e}\frac{1}{\sqrt{z}}\exp\bigg(-C\sqrt{\frac{\eta}{\e}-z}\bigg)\ud{z}\Bigg).\no
\end{eqnarray}
We can estimate these two terms separately.
\begin{eqnarray}
\int_0^{1}\frac{1}{\sqrt{z}}\exp\bigg(-C\sqrt{\frac{\eta}{\e}-z}\bigg)\ud{z}&\leq&\int_0^{1}\frac{1}{\sqrt{z}}\ud{z}=2.
\end{eqnarray}
\begin{eqnarray}
\int_1^{\eta/\e}\frac{1}{\sqrt{z}}\exp\bigg(-C\sqrt{\frac{\eta}{\e}-z}\bigg)\ud{z}&\leq&\int_1^{\eta/\e}\exp\bigg(-C\sqrt{\frac{\eta}{\e}-z}\bigg)\ud{z}
\overset{t^2=\frac{\eta}{\e}-z}{\leq}2\int_0^{\infty}t\ue^{-Ct}\ud{t}<\infty.
\end{eqnarray}
Hence, we know
\begin{eqnarray}
\abs{I_4}&\leq&C\d\lnnm{\a}.
\end{eqnarray}
\ \\
Step 5: Estimate of $I_5$ for $0\leq\sin\phi\leq\d_0$, $\chi(\phi_{\ast})=1$, $\sqrt{\e\eta'}\leq\sin\phi'$ and $\sin^2\phi\geq\e(\eta-\eta')$.\\
Based on (\ref{pt 01}), this implies
\begin{eqnarray}
\zeta(\eta',\phi')\leq C\sin\phi'.\no
\end{eqnarray}
Based on (\ref{pt 03}), we have
\begin{eqnarray}
-G_{\eta,\eta'}=-\int_{\eta'}^{\eta}\frac{1}{\sin\phi'(y)}\ud{y}&\leq&-\frac{C(\eta-\eta')}{\sin\phi}.
\end{eqnarray}
Hence, we have
\begin{eqnarray}
\abs{I_5}&\leq& C\lnnm{\a}\int_0^{\eta}\bigg(\int_{-\d}^{\d}\frac{1}{\zeta(\eta',\phi_{\ast})}
\ud{\phi_{\ast}}\bigg)
\exp\left(-\frac{C(\eta-\eta')}{\sin\phi}\right)\ud{\eta'}
\end{eqnarray}
Here, we use a different way to estimate the inner integral. We use substitution to find
\begin{eqnarray}
\int_{-\d}^{\d}\frac{1}{\zeta(\eta',\phi_{\ast})}
\ud{\phi_{\ast}}
&=&\int_{-\d}^{\d}\frac{1}{\bigg(\rk^2-(\rk-\e\eta')^2\cos\phi_{\ast}^2\bigg)^{1/2}}
\ud{\phi_{\ast}}\\
&\overset{\sin\phi_{\ast}\ small}{\leq}&C\int_{-\d}^{\d}\frac{\cos\phi_{\ast}}{\bigg(\rk^2-(\rk-\e\eta')^2\cos\phi_{\ast}^2\bigg)^{1/2}}
\ud{\phi_{\ast}}\no\\
&=&C\int_{-\d}^{\d}\frac{\cos\phi_{\ast}}{\bigg(\rk^2-(\rk-\e\eta')^2+(\rk-\e\eta')^2\sin\phi_{\ast}^2\bigg)^{1/2}}
\ud{\phi_{\ast}}\no\\
&\overset{y=\sin\phi_{\ast}}{=}&C\int_{-\d}^{\d}\frac{1}{\bigg(\rk^2-(\rk-\e\eta')^2+(\rk-\e\eta')^2y^2\bigg)^{1/2}}
\ud{y}.\no
\end{eqnarray}
Define
\begin{eqnarray}
p&=&\sqrt{\rk^2-(\rk-\e\eta')^2}=\sqrt{2\rk\e\eta'-\e^2\eta'^2}\leq C\sqrt{\e\eta'},\\
q&=&\rk-\e\eta'\geq C,\\
r&=&\frac{p}{q}\leq C\sqrt{\e\eta'}.
\end{eqnarray}
Then we have
\begin{eqnarray}
\int_{-\d}^{\d}\frac{1}{\zeta(\eta',\phi_{\ast})}\ud{\phi_{\ast}}&\leq&C\int_{-\d}^{\d}\frac{1}{(p^2+q^2y^2)^{1/2}}\ud{y}\\
&\leq&C\int_{-2}^{2}\frac{1}{(p^2+q^2y^2)^{1/2}}\ud{y}\leq C\int_{-2}^{2}\frac{1}{(r^2+y^2)^{1/2}}\ud{y}\no\\
&\leq&C\int_{0}^{2}\frac{1}{(r^2+y^2)^{1/2}}\ud{y}=\bigg(\ln(y+\sqrt{r^2+y^2})-\ln(r)\bigg)\bigg|_0^{2}\no\\
&\leq&C\bigg(\ln(2+\sqrt{r^2+4})-\ln{r}\bigg)\leq C\bigg(1+\ln(r)\bigg)\no\\
&\leq&C\bigg(1+\abs{\ln(\e)}+\abs{\ln(\eta')}\bigg).\no
\end{eqnarray}
Hence, we know
\begin{eqnarray}
\abs{I_5}&\leq&C\lnnm{\a}\int_0^{\eta}\bigg(1+\abs{\ln(\e)}+\abs{\ln(\eta')}\bigg)
\exp\left(-\frac{C(\eta-\eta')}{\sin\phi}\right)\ud{\eta'}
\end{eqnarray}
We may directly compute
\begin{eqnarray}
\abs{\int_0^{\eta}\bigg(1+\abs{\ln(\e)}\bigg)
\exp\left(-\frac{C(\eta-\eta')}{\sin\phi}\right)\ud{\eta'}}\leq C\sin\phi(1+\abs{\ln(\e)}).
\end{eqnarray}
Hence, we only need to estimate
\begin{eqnarray}
\abs{\int_0^{\eta}\abs{\ln(\eta')}
\exp\left(-\frac{C(\eta-\eta')}{\sin\phi}\right)\ud{\eta'}}.
\end{eqnarray}
If $\eta\leq 2$, using Cauchy's inequality, we have
\begin{eqnarray}
\abs{\int_0^{\eta}\abs{\ln(\eta')}
\exp\left(-\frac{C(\eta-\eta')}{\sin\phi}\right)\ud{\eta'}}
&\leq&\bigg(\int_0^{\eta}\ln^2(\eta')\ud{\eta'}\bigg)^{1/2}\bigg(\int_0^{\eta}
\exp\left(-\frac{2C(\eta-\eta')}{\sin\phi}\right)\ud{\eta'}\bigg)^{1/2}\\
&\leq&\bigg(\int_0^{2}\ln^2(\eta')\ud{\eta'}\bigg)^{1/2}\bigg(\int_0^{\eta}
\exp\left(-\frac{2C(\eta-\eta')}{\sin\phi}\right)\ud{\eta'}\bigg)^{1/2}\no\\
&\leq&\sqrt{\sin\phi}.\no
\end{eqnarray}
If $\eta\geq 2$, we decompose and apply Cauchy's inequality to obtain
\begin{eqnarray}
&&\abs{\int_0^{\eta}\abs{\ln(\eta')}
\exp\left(-\frac{C(\eta-\eta')}{\sin\phi}\right)\ud{\eta'}}\\
&\leq&\abs{\int_0^{2}\abs{\ln(\eta')}
\exp\left(-\frac{C(\eta-\eta')}{\sin\phi}\right)\ud{\eta'}}+\abs{\int_2^{\eta}\ln(\eta')
\exp\left(-\frac{C(\eta-\eta')}{\sin\phi}\right)\ud{\eta'}}\no\\
&\leq&\bigg(\int_0^{2}\ln^2(\eta')\ud{\eta'}\bigg)^{1/2}\bigg(\int_0^{2}
\exp\left(-\frac{2C(\eta-\eta')}{\sin\phi}\right)\ud{\eta'}\bigg)^{1/2}+\ln(2)\abs{\int_2^{\eta}
\exp\left(-\frac{C(\eta-\eta')}{\sin\phi}\right)\ud{\eta'}}\no\\
&\leq&C\bigg(\sqrt{\sin\phi}+\sin\phi\bigg)\leq C\sqrt{\sin\phi}.\no
\end{eqnarray}
Hence, we have
\begin{eqnarray}
\abs{I_5}\leq C(1+\abs{\ln(\e)})\sqrt{\d_0}\lnnm{\a}.
\end{eqnarray}
\ \\
Step 6: Synthesis.\\
Collecting all the terms in previous steps, we have proved
\begin{eqnarray}
\abs{I}&\leq&C(1+\abs{\ln(\e)})\sqrt{\d_0}\lnnm{\a}+C\d\lnnm{\a}\\
&&+\frac{C}{\d_0^4}\bigg(\lnnm{\v}+\lnm{\frac{\p p}{\p\phi}}+\lnnm{\frac{\p S}{\p\phi}}+\lnnm{\frac{\p S}{\p\eta}}\bigg)+C(\d)\bigg(\lnnm{\v}+\lnnm{S}\bigg).\no
\end{eqnarray}
Therefore, we know
\begin{eqnarray}
\abs{\a}_{I}&\leq&\lnm{p_{\a}}+\lnnm{S_{\a}}+C(1+\abs{\ln(\e)})\sqrt{\d_0}\lnnm{\a}+C\d\lnnm{\a}\\
&&+\frac{C}{\d_0^4}\bigg(\lnnm{\v}+\lnm{\frac{\p p}{\p\phi}}+\lnnm{\frac{\p S}{\p\phi}}+\lnnm{\frac{\p S}{\p\eta}}\bigg)+C(\d)\bigg(\lnnm{\v}+\lnnm{S}\bigg).\no
\end{eqnarray}

\subsubsection{Region II: $\sin\phi<0$ and $\abs{E(\eta,\phi)}\leq \ue^{-V_{L}}$}

\begin{eqnarray}
\k[p_{\a}]&=&p_{\a}(\phi'(0))\exp(-G_{L,0}-G_{L,\eta})\\
\t[\tilde\a+S_{\a}]&=&\int_0^{L}\frac{(\tilde\a+S)(\eta',\phi'(\eta'))}{\sin(\phi'(\eta'))}
\exp(-G_{L,\eta'}-G_{L,\eta})\ud{\eta'}\\
&&+\int_{\eta}^{L}\frac{(\tilde\a+S)(\eta',R\phi'(\eta'))}{\sin(\phi'(\eta'))}\exp(-G_{\eta',\eta})\ud{\eta'}\nonumber.
\end{eqnarray}
Based on \cite[Lemma 4.7, Lemma 4.8]{AA003},
we can directly obtain
\begin{eqnarray}
\abs{\k[p_{\a}]}&\leq&\lnm{p_{\a}},\\
\abs{\t[S_{\a}]}&\leq&\lnnm{S_{\a}}.
\end{eqnarray}
Hence, we only need to estimate $II=\t[\tilde\a]$. In particular, we can decompose
\begin{eqnarray}
\t[\tilde\a]&=&\int_0^{L}\frac{\tilde\a(\eta',\phi'(\eta'))}{\sin(\phi'(\eta'))}\exp(-G_{L,\eta'}-G_{L,\eta})\ud{\eta'}
+\int_{\eta}^{L}\frac{\tilde\a(\eta',R\phi'(\eta'))}{\sin(\phi'(\eta'))}\exp(-G_{\eta',\eta})\ud{\eta'}\\
&=&\int_0^{\eta}\frac{\tilde\a(\eta',\phi'(\eta'))}{\sin(\phi'(\eta'))}\exp(-G_{L,\eta'}-G_{L,\eta})\ud{\eta'}\no\\
&&+\int_{\eta}^{L}\frac{\tilde\a(\eta',\phi'(\eta'))}{\sin(\phi'(\eta'))}\exp(-G_{L,\eta'}-G_{L,\eta})\ud{\eta'}
+\int_{\eta}^{L}\frac{\tilde\a(\eta',R\phi'(\eta'))}{\sin(\phi'(\eta'))}\exp(-G_{\eta',\eta})\ud{\eta'}.\no
\end{eqnarray}
The integral $\displaystyle\int_0^{\eta}\cdots$ can be estimated as in Region I, so we only need to estimate the integral $\displaystyle\int_{\eta}^L\cdots$. Also, noting that fact that \begin{eqnarray}
\exp(-G_{L,\eta'}-G_{L,\eta})\leq \exp(-G_{\eta',\eta}),
\end{eqnarray}
we only need to estimate
\begin{eqnarray}
\int_{\eta}^{L}\frac{\tilde\a(\eta',R\phi'(\eta'))}{\sin(\phi'(\eta'))}\exp(-G_{\eta',\eta})\ud{\eta'}.
\end{eqnarray}
Here the proof is almost identical to Case I, so we only point out the key differences.\\
\ \\
Step 0: Preliminaries.\\
We need to update one key result. For $0\leq\eta\leq\eta'$,
\begin{eqnarray}
\sin\phi'&=&\sqrt{1-\cos^2\phi'}=\sqrt{1-\bigg(\frac{\rk-\e\eta}{\rk-\e\eta'}\bigg)^2\cos^2\phi}\\
&=&\frac{\sqrt{(\rk-\e\eta')^2\sin^2\phi+(2\rk-\e\eta-\e\eta')(\e\eta'-\e\eta)\cos^2\phi}}{\rk-\e\eta'}\no\\
&\leq&\abs{\sin\phi}.
\end{eqnarray}
Then we have
\begin{eqnarray}\label{pt 04}
-\int_{\eta}^{\eta'}\frac{1}{\sin\phi'(y)}\ud{y}&\leq&-\frac{\eta'-\eta}{\abs{\sin\phi}}.
\end{eqnarray}
In the following, we will divide the estimate of $II$ into several cases based on the value of $\sin\phi$, $\sin\phi'$ and $\e\eta'$. We write
\begin{eqnarray}
II&=&\int_{\eta}^L\id_{\{\sin\phi\leq-\d_0\}}+\int_{\eta}^L\id_{\{-\d_0\leq\sin\phi\leq0\}}\id_{\{\chi(\phi_{\ast})<1\}}\\
&&+\int_{\eta}^L\id_{\{-\d_0\leq\sin\phi\leq0\}}\id_{\{\chi(\phi_{\ast})=1\}}\id_{\{\sqrt{\e\eta'}\geq\sin\phi'\}}
+\int_{\eta}^L\id_{\{-\d_0\leq\sin\phi\leq0\}}\id_{\{\chi(\phi_{\ast})=1\}}\id_{\{\sqrt{\e\eta'}\leq\sin\phi'\}}\no\\
&=&II_1+II_2+II_3+II_4.\no
\end{eqnarray}
\ \\
Step 1: Estimate of $II_1$ for $\sin\phi\leq-\d_0$.\\
We first estimate $\sin\phi'$. Along the characteristics, we know
\begin{eqnarray}
\ue^{-V(\eta')}\cos\phi'=\ue^{-V(\eta)}\cos\phi,
\end{eqnarray}
which implies
\begin{eqnarray}
\cos\phi'&=&\ue^{V(\eta')-V(\eta)}\cos\phi\leq \ue^{V(L)-V(0)}\cos\phi= \ue^{V(L)-V(0)}\sqrt{1-\d_0^2}.
\end{eqnarray}
Based on Lemma \ref{rt lemma 1}, we can further deduce that
\begin{eqnarray}
\cos\phi'\leq \bigg(1-\frac{\e^{1/2}}{\rk}\bigg)^{-1}\sqrt{1-\d_0^2}.
\end{eqnarray}
Then we have
\begin{eqnarray}
\sin\phi'\geq\sqrt{1-\bigg(1-\frac{\e^{1/2}}{\rk}\bigg)^{-2}(1-\d_0^2)}\geq \d_0-\e^{1/4}>\frac{\d_0}{2},
\end{eqnarray}
when $\e$ is sufficiently small.
Based on Lemma \ref{pt lemma 2}, we know
\begin{eqnarray}
\abs{\sin\phi\frac{\p\v}{\p\eta}}\leq C\bigg(1+\frac{1}{\d_0^3}\bigg)\bigg(\lnnm{\v}+\lnnm{\frac{\p S}{\p\phi}}+\lnnm{\frac{\p S}{\p\eta}}\bigg).
\end{eqnarray}
Hence, we have
\begin{eqnarray}
\abs{II_1}\leq  \frac{1}{\abs{\sin\phi}}\abs{\frac{\p\v}{\p\eta}}\leq \frac{C}{\d_0^4}\bigg(\lnnm{\v}+\lnnm{\frac{\p S}{\p\phi}}+\lnnm{\frac{\p S}{\p\eta}}\bigg).
\end{eqnarray}
\ \\
Step 2: Estimate of $II_2$ for $-\d_0\leq\sin\phi\leq0$ and $\chi(\phi_{\ast})<1$.\\
This is similar to the estimate of $I_2$ based on the integral
\begin{eqnarray}
\int_{\eta}^{L}\frac{1}{\sin\phi'}\exp(-G_{\eta',\eta})\ud{\eta'}\leq 1.
\end{eqnarray}
Then we have
\begin{eqnarray}
\abs{II_2}
&\leq&C(\d)\bigg(\lnnm{\v}+\lnnm{S}\bigg).
\end{eqnarray}
\ \\
Step 3: Estimate of $II_3$ for $-\d_0\leq\sin\phi\leq0$, $\chi(\phi_{\ast})=1$ and $\sqrt{\e\eta'}\geq\sin\phi'$.\\
This is identical to the estimate of $I_4$, we have
\begin{eqnarray}
\abs{II_3}&\leq&C\d\lnnm{\a}.
\end{eqnarray}
\ \\
Step 4: Estimate of $II_4$ for $-\d_0\leq\sin\phi\leq0$, $\chi(\phi_{\ast})=1$ and $\sqrt{\e\eta'}\leq\sin\phi'$.\\
This step is different. We do not need to further decompose the cases.
Based on (\ref{pt 04}), we have,
\begin{eqnarray}
-G_{\eta,\eta'}&\leq&-\frac{\eta'-\eta}{\abs{\sin\phi}}.
\end{eqnarray}
Then following the same argument in estimating $I_5$, we obtain
\begin{eqnarray}
\abs{II_4}&\leq&C\lnnm{\a}\int_{\eta}^{L}\bigg(1+\abs{\ln(\e)}+\abs{\ln(\eta')}\bigg)
\exp\left(-\frac{\eta'-\eta}{\abs{\sin\phi}}\right)\ud{\eta'}
\end{eqnarray}
If $\eta\geq 2$, we directly obtain
\begin{eqnarray}
\abs{\int_{\eta}^{L}\abs{\ln(\eta')}
\exp\left(-\frac{\eta'-\eta}{\abs{\sin\phi}}\right)\ud{\eta'}}&\leq& \abs{\int_{2}^{L}\ln(\eta')
\exp\left(-\frac{\eta'-\eta}{\abs{\sin\phi}}\right)\ud{\eta'}}\\
&\leq&\ln(2)\abs{\int_{2}^{L}
\exp\left(-\frac{\eta'-\eta}{\abs{\sin\phi}}\right)\ud{\eta'}}\no\\
&\leq&C\sqrt{\abs{\sin\phi}}.\no
\end{eqnarray}
If $\eta\leq 2$, we decompose as
\begin{eqnarray}
&&\abs{\int_{\eta}^{L}\abs{\ln(\eta')}
\exp\left(-\frac{\eta'-\eta}{\abs{\sin\phi}}\right)\ud{\eta'}}\\
&\leq&\abs{\int_{\eta}^{2}\abs{\ln(\eta')}
\exp\left(-\frac{\eta'-\eta}{\abs{\sin\phi}}\right)\ud{\eta'}}+\abs{\int_{2}^{L}\abs{\ln(\eta')}
\exp\left(-\frac{\eta'-\eta}{\abs{\sin\phi}}\right)\ud{\eta'}}
\end{eqnarray}
The second term is identical to the estimate in $\eta\geq2$. We apply Cauchy's inequality to the first term
\begin{eqnarray}
\abs{\int_{\eta}^{2}\abs{\ln(\eta')}
\exp\left(-\frac{\eta'-\eta}{\abs{\sin\phi}}\right)\ud{\eta'}}
&\leq&\bigg(\int_{\eta}^{2}\ln^2(\eta')\ud{\eta'}\bigg)^{1/2}\bigg(\int_{\eta}^{2}
\exp\left(-\frac{2(\eta'-\eta)}{\abs{\sin\phi}}\right)\ud{\eta'}\bigg)^{1/2}\\
&\leq&\bigg(\int_0^{2}\ln^2(\eta')\ud{\eta'}\bigg)^{1/2}\bigg(\int_{\eta}^{2}
\exp\left(-\frac{2(\eta'-\eta)}{\abs{\sin\phi}}\right)\ud{\eta'}\bigg)^{1/2}\no\\
&\leq&C\sqrt{\abs{\sin\phi}}.\no
\end{eqnarray}
Hence, we have
\begin{eqnarray}
\abs{II_4}\leq C(1+\abs{\ln(\e)})\sqrt{\d_0}\lnnm{\a}.
\end{eqnarray}
\ \\
Step 6: Synthesis.\\
Collecting all the terms in previous steps, we have proved
\begin{eqnarray}
\abs{II}&\leq&C(1+\abs{\ln(\e)})\sqrt{\d_0}\lnnm{\a}+C\d\lnnm{\a}\\
&&+\frac{C}{\d_0^4}\bigg(\lnnm{\v}+\lnm{\frac{\p p}{\p\phi}}+\lnnm{\frac{\p S}{\p\phi}}+\lnnm{\frac{\p S}{\p\eta}}\bigg)+C(\d)\bigg(\lnnm{\v}+\lnnm{S}\bigg).\no
\end{eqnarray}
Therefore, we know
\begin{eqnarray}
\abs{\a}_{II}&\leq&\lnnm{S_{\a}}+\lnm{p_{\a}}+C(1+\abs{\ln(\e)})\sqrt{\d_0}\lnnm{\a}+C\d\lnnm{\a}\\
&&+\frac{C}{\d_0^4}\bigg(\lnnm{\v}+\lnm{\frac{\p p}{\p\phi}}+\lnnm{\frac{\p S}{\p\phi}}+\lnnm{\frac{\p S}{\p\eta}}\bigg)+C(\d)\bigg(\lnnm{\v}+\lnnm{S}\bigg).\no
\end{eqnarray}

\subsubsection{Region III: $\sin\phi<0$ and $\abs{E(\eta,\phi)}\geq \ue^{-V(L)}$}

Based on \cite[Lemma 4.7, Lemma 4.8]{AA003}, we still have
\begin{eqnarray}
\abs{\k[p_{\a}]}&\leq&\lnm{p_{\a}},\\
\abs{\t[S_{\a}]}&\leq&\lnnm{S_{\a}}.
\end{eqnarray}
Hence, we only need to estimate $III=\t[\tilde\a]$. Note that $\abs{E(\eta,\phi)}\geq \ue^{-V(L)}$ implies
\begin{eqnarray}
\ue^{-V(\eta)}\cos\phi\geq \ue^{-V(L)}.
\end{eqnarray}
Hence, based on Lemma \ref{rt lemma 1}, we can further deduce that
\begin{eqnarray}
\cos\phi&\geq&\ue^{V(\eta)-V(L)}\geq \ue^{V(0)-V_{\infty}}\geq \bigg(1-\frac{\e^{1/2}}{\rk}\bigg).
\end{eqnarray}
Hence, we know
\begin{eqnarray}
\abs{\sin\phi}\leq\sqrt{1-\bigg(1-\frac{\e^{1/2}}{\rk}\bigg)^2}\leq \e^{1/4}.
\end{eqnarray}
Hence, when $\e$ is sufficiently small, we always have
\begin{eqnarray}
\abs{\sin\phi}\leq \e^{1/4}\leq \d_0.
\end{eqnarray}
This means we do not need to bother with the estimate of $\sin\phi\leq-\d_0$ as Step 1 in estimating $I$ and $II$.
Since we can decompose
\begin{eqnarray}
\t[\tilde\a]&=&\int_0^{\eta}\frac{\tilde\a(\eta',\phi'(\eta'))}{\sin(\phi'(\eta'))}
\exp(-G_{\eta^+,\eta'}-G_{\eta^+,\eta})\ud{\eta'}\\
&&\bigg(\int_{\eta}^{\eta^+}\frac{\tilde\a(\eta',\phi'(\eta'))}{\sin(\phi'(\eta'))}
\exp(-G_{\eta^+,\eta'}-G_{\eta^+,\eta})\ud{\eta'}+
\int_{\eta}^{\eta^+}\frac{(\tilde\a+S_{\a})(\eta',R\phi'(\eta'))}{\sin(\phi'(\eta'))}\exp(-G_{\eta',\eta})\ud{\eta'}\bigg)\nonumber.
\end{eqnarray}
Then the integral $\displaystyle\int_0^{\eta}(\cdots)$ is similar to the argument in Region I, and the integral $\displaystyle\int_{\eta}^{\eta^+}(\cdots)$ is similar to the argument in Region II. Hence, combining the method in Region I and Region II, we can show the desired result, i.e.
\begin{eqnarray}
\abs{\a}_{III}&\leq&\lnm{p_{\a}}+\lnnm{S_{\a}}+C(1+\abs{\ln(\e)})\sqrt{\d_0}\lnnm{\a}+C\d\lnnm{\a}\\
&&+C(\d)\bigg(\lnnm{\v}+\lnnm{S}\bigg).\no
\end{eqnarray}

\subsubsection{Estimate of Normal Derivative}

Combining the analysis in these three regions, we have
\begin{eqnarray}
\abs{\a}&\leq&\lnm{p_{\a}}+\lnnm{S_{\a}}+C(1+\abs{\ln(\e)})\sqrt{\d_0}\lnnm{\a}+C\d\lnnm{\a}\\
&&+\frac{C}{\d_0^4}\bigg(\lnnm{\v}+\lnm{\frac{\p p}{\p\phi}}+\lnnm{\frac{\p S}{\p\phi}}+\lnnm{\frac{\p S}{\p\eta}}\bigg)+C(\d)\bigg(\lnnm{\v}+\lnnm{S}\bigg).\no
\end{eqnarray}
Taking supremum over all $(\eta,\phi)$, we have
\begin{eqnarray}\label{pt 05}
\lnnm{\a}&\leq&\lnm{p_{\a}}+\lnnm{S_{\a}}+C(1+\abs{\ln(\e)})\sqrt{\d_0}\lnnm{\a}+C\d\lnnm{\a}\\
&&+\frac{C}{\d_0^4}\bigg(\lnnm{\v}+\lnm{\frac{\p p}{\p\phi}}+\lnnm{\frac{\p S}{\p\phi}}+\lnnm{\frac{\p S}{\p\eta}}\bigg)\no\\
&&+C(\d)\bigg(\lnnm{\v}+\lnnm{S}\bigg).\no
\end{eqnarray}
Then we choose these constants to perform absorbing argument. First we choose $0<\d<<1$ sufficiently small such that
\begin{eqnarray}
C\d\leq\frac{1}{4}.
\end{eqnarray}
Then we take $\d_0=\d\abs{\ln(\e)}^{-2}$ such that
\begin{eqnarray}
C(1+\abs{\ln(\e)})\sqrt{\d_0}\leq 2C\d\leq\frac{1}{2}.
\end{eqnarray}
for $\e$ sufficiently small. Note that this mild decay of $\d_0$ with respect to $\e$ also justifies the assumption in Case III and the proof of Lemma \ref{pt lemma 2} that
\begin{eqnarray}
\e^{1/4}\leq \frac{\d_0}{2},
\end{eqnarray}
for $\e$ sufficiently small. Here since $\d$ and $C$ are independent of $\e$, there is no circulant argument. Hence, we can absorb all the term related to $\lnnm{\a}$ on the right-hand side of (\ref{pt 05}) to the left-hand side to obtain
\begin{eqnarray}
\lnnm{\a}&\leq&C\bigg(\lnm{p_{\a}}+\lnnm{S_{\a}}\bigg)\\
&&+C\abs{\ln(\e)}^8\bigg(\lnnm{\v}+\lnm{\frac{\p p}{\p\phi}}+\lnnm{S}+\lnnm{\frac{\p S}{\p\phi}}+\lnnm{\frac{\p S}{\p\eta}}\bigg).\no
\end{eqnarray}

\subsection{Mild Formulation of Velocity Derivative}

Consider the general $\e$-Milne problem for $\b=\zeta\dfrac{\p\v}{\p\phi}$ as
\begin{eqnarray}
\left\{
\begin{array}{rcl}\displaystyle
\sin\phi\frac{\p\b}{\p\eta}+F(\eta)\cos\phi\frac{\p
\b}{\p\phi}+\b&=&S_{\b},\\
\b(0,\phi)&=&p_{\b}(\phi)\ \ \text{for}\ \ \sin\phi>0,\\
\b(L,\phi)&=&\b(L,R\phi),
\end{array}
\right.
\end{eqnarray}
where $p_{\b}$ and $S_{\b}$ will be specified later. This is much simpler than normal derivative, since we do not have $\tilde\b$. Then by a direct argument that
\begin{eqnarray}
\abs{\k[p_{\b}]}&\leq&\lnm{p_{\b}},\\
\abs{\t[S_{\b}]}&\leq&\lnnm{S_{\b}}.
\end{eqnarray}
we can get the desired result.
\begin{lemma}\label{pt lemma 4}
We have
\begin{eqnarray}
\lnnm{\b}&\leq&\lnm{p_{\b}}+\lnnm{S_{\b}}.
\end{eqnarray}
\end{lemma}

\subsection{A Priori Estimate of Derivatives}

\begin{theorem}\label{pt theorem 1}
We have
\begin{eqnarray}
\lnnm{\zeta\frac{\p\v}{\p\eta}}+\lnnm{\zeta\frac{\p\v}{\p\eta}}\leq C\abs{\ln(\e)}^8.
\end{eqnarray}
\end{theorem}
\begin{proof}
Collecting the estimates for $\a$ and $\b$ in Lemma \ref{pt lemma 3} and Lemma \ref{pt lemma 4}, we have
\begin{eqnarray}\label{pt 06}
\lnnm{\a}&\leq&C\bigg(\lnm{p_{\a}}+\lnnm{S_{\a}}\bigg)\label{pt 06}\\
&&+C\abs{\ln(\e)}^8\bigg(\lnnm{\v}+\lnm{\frac{\p p}{\p\phi}}+\lnnm{S}+\lnnm{\frac{\p S}{\p\phi}}+\lnnm{\frac{\p S}{\p\eta}}\bigg),\no\\
\lnnm{\b}&\leq&\lnm{p_{\b}}+\lnnm{S_{\b}}.\label{pt 07}
\end{eqnarray}
Taking derivatives on both sides of (\ref{Milne decay}) and multiplying $\zeta$, based on Lemma \ref{rt lemma 2}, we have
\begin{eqnarray}
p_{\a}&=&\e\cos\phi\frac{\p p}{\p\phi}+p-\bar\v(0),\\
p_{\b}&=&\sin\phi\frac{\p p}{\p\phi},\\
S_{\a}&=&\frac{\p{F}}{\p{\eta}}\b\cos\phi+\zeta\frac{\p S}{\p\eta},\\
S_{\b}&=&\a\cos\phi+F\b\sin\phi+\zeta\frac{\p S}{\p\phi}.
\end{eqnarray}
Since $\abs{F(\eta)}+\abs{\dfrac{\p F}{\p\eta}}\leq \e$, by absorbing $\a$ and $\b$ on the right-hand side of (\ref{pt 06}) and (\ref{pt 07}), we derive
\begin{eqnarray}
\a&\leq&C\abs{\ln(\e)}^8,\\
\b&\leq&C\abs{\ln(\e)}^8.
\end{eqnarray}
\end{proof}

\begin{theorem}\label{pt theorem 2}
For $K_0>0$ sufficiently small, we have
\begin{eqnarray}
\lnnm{\ue^{K_0\eta}\zeta\frac{\p\v}{\p\eta}}+\lnnm{\ue^{K_0\eta}\zeta\frac{\p\v}{\p\eta}}\leq C\abs{\ln(\e)}^8.
\end{eqnarray}
\end{theorem}
\begin{proof}
This proof is almost identical to Theorem \ref{pt theorem 1}. The only difference is that $S_{\a}$ is added by $K_0\a\sin\phi$ and $S_{\b}$ added by $K_0\b\sin\phi$. When $K_0$ is sufficiently small, we can also absorb them into the left-hand side. Hence, this is obvious.
\end{proof}

\subsection{Iteration and Estimate of Derivatives}

So far, all the estimates are a priori. Hence, we first need to confirm the derivatives are well-defined. We start from continuity of solutions.
We consider the $\e$-transport equation for $\v$ as
\begin{eqnarray}
\left\{
\begin{array}{rcl}\displaystyle
\sin\phi\frac{\p\v}{\p\eta}+F(\eta)\cos\phi\frac{\p
\v}{\p\phi}+\v&=&H,\\
\v(0,\phi)&=&p(\phi)\ \ \text{for}\ \ \sin\phi>0,\\
f(L,\phi)&=&f(L,R\phi).
\end{array}
\right.
\end{eqnarray}
\begin{lemma}\label{pt additional 1}
Assume $H$ is continuous in $[0,L]\times[-\pi,\pi)$. Then we have $\v$ is continuous in $[0,L]\times[-\pi,\pi)$
\end{lemma}
\begin{proof}
As before, we can define the solution along the characteristics as follows:
\begin{eqnarray}
\v(\eta,\phi)=\k[p]+\t[H],
\end{eqnarray}
where\\
\ \\
Region I:\\
For $\sin\phi>0$,
\begin{eqnarray}
\k[p]&=&p(\phi'(0))\exp(-G_{\eta,0}),\\
\t[H]&=&\int_0^{\eta}\frac{H(\eta',\phi'(\eta'))}{\sin(\phi'(\eta'))}\exp(-G_{\eta,\eta'})\ud{\eta'}.
\end{eqnarray}
\ \\
Region II:\\
For $\sin\phi<0$ and $\abs{E(\eta,\phi)}\leq \ue^{-V(L)}$,
\begin{eqnarray}
\k[p]&=&p(\phi'(0))\exp(-G_{L,0}-G_{L,\eta})\nonumber\\
\t[H]&=&\bigg(\int_0^{L}\frac{H(\eta',\phi'(\eta'))}{\sin(\phi'(\eta'))}
\exp(-G_{L,\eta'}-G_{L,\eta})\ud{\eta'}
+\int_{\eta}^{L}\frac{H(\eta',R\phi'(\eta'))}{\sin(\phi'(\eta'))}\exp(G_{\eta,\eta'})\ud{\eta'}\bigg)\nonumber.
\end{eqnarray}
\ \\
Region III:\\
For $\sin\phi<0$ and $\abs{E(\eta,\phi)}\geq \ue^{-V(L)}$,
\begin{eqnarray}
\k[p]&=&p(\phi'(0))\exp(-G_{\eta^+,0}-G_{\eta^+,\eta})\nonumber\\
\t[H]&=&\bigg(\int_0^{\eta^+}\frac{H(\eta',\phi'(\eta'))}{\sin(\phi'(\eta'))}
\exp(-G_{\eta^+,\eta'}-G_{\eta^+,\eta})\ud{\eta'}
+\int_{\eta}^{\eta^+}\frac{H(\eta',R\phi'(\eta'))}{\sin(\phi'(\eta'))}\exp(G_{\eta,\eta'})\ud{\eta'}\bigg)\nonumber.
\end{eqnarray}
When $(\eta,\phi)$ does not touch the boundary of each cases, we can directly use above mild formulation to see the continuity. Hence, we concentrate on the separatrix between these regions. We divide the proof into several steps:\\
\ \\
Step 1: Separatrix between Region I and Case II.\\
In our formulation, there is no intersection between these two cases, so we do not need to worry about it.\\
\ \\
Step 2: Separatrix between Region II and Region III.\\
The separatrix is the curve satisfying $\abs{E(\eta,\phi)}=\ue^{-V(L)}$. We have in Region II:\\
\begin{eqnarray}
\k[p]&=&p(\phi'(0))\exp(-G_{L,0}-G_{L,\eta})\nonumber\\
\t[H]&=&\bigg(\int_0^{L}\frac{H(\eta',\phi'(\eta'))}{\sin(\phi'(\eta'))}
\exp(-G_{L,\eta'}-G_{L,\eta})\ud{\eta'}
+\int_{\eta}^{L}\frac{H(\eta',R\phi'(\eta'))}{\sin(\phi'(\eta'))}\exp(G_{\eta,\eta'})\ud{\eta'}\bigg)\nonumber.
\end{eqnarray}
and in Region III:\\
\begin{eqnarray}
\k[p]&=&p(\phi'(0))\exp(-G_{\eta^+,0}-G_{\eta^+,\eta})\nonumber\\
\t[H]&=&\bigg(\int_0^{\eta^+}\frac{H(\eta',\phi'(\eta'))}{\sin(\phi'(\eta'))}
\exp(-G_{\eta^+,\eta'}-G_{\eta^+,\eta})\ud{\eta'}
+\int_{\eta}^{\eta^+}\frac{H(\eta',R\phi'(\eta'))}{\sin(\phi'(\eta'))}\exp(G_{\eta,\eta'})\ud{\eta'}\bigg)\nonumber.
\end{eqnarray}
Since we know $\eta^+=L$ on this curve, above two formulations give exactly the same formula. Hence, it is continuous.\\
\ \\
Step 3: Separatrix between Region I and Region III.\\
This is actually the segment of line $(\eta,0)$ for $0<\eta<L$.\\
\ \\
Direction 1: Approaching from Region I.\\
Consider $(\ea,\pa)\rt(\eta,0)$. Assume $(\ea,\pa)$ and $(\eta',\phi')$ are on the same characteristics. Then we have
\begin{eqnarray}
\k[p]&=&p(\phi'(0))\exp(-G_{\ea,0})\\
\t[H]&=&\int_0^{\ea}\frac{H(\eta',\phi'(\eta'))}{\sin\phi'(\eta')}\exp(-G_{\ea,\eta'})\ud{\eta'}.
\end{eqnarray}
We can directly take limit $(\ea,\pa)\rt(\eta',\phi')$ and obtain
\begin{eqnarray}
\k[p]&\rt& p(\phi'(\eta,0;0))\exp\bigg(-\int_{0}^{\eta}\frac{1}{\sin\phi'(y)}\ud{y}\bigg),\\
\t[H]&\rt&\int_0^{\eta}\frac{H(\eta',\phi'(\eta'))}{\sin\phi'(\eta')}\exp\bigg(-\int_{\eta'}^{\eta}\frac{1}{\sin\phi'(y)}\ud{y}\bigg)\ud{\eta'}.
\end{eqnarray}
Here, we cannot further simplify these quantities.\\
\ \\
Direction 2: Approaching from Region III.\\
Consider $(\ea,\pa)\rt(\eta,0)$. Assume $(\ea,\pa)$ and $(\eta',\phi')$ are on the same characteristics. Then we have
\begin{eqnarray}
\k[p]&=&p(\phi'(0))\exp(-G_{\eta^+,0}-G_{\eta^+,\ea})\\
\t[H]&=&\int_0^{\eta^+}\frac{H(\eta',\phi'(\eta'))}{\sin\phi'(\eta')}
\exp(-G_{\eta^+,\eta'}-G_{\eta^+,\ea})\ud{\eta'}\\
&&+\int_{\ea}^{\eta^+}\frac{H(\eta',R\phi'(\eta'))}{\sin\phi'(\eta')}\exp(-G_{\eta',\ea})\ud{\eta'}\nonumber.
\end{eqnarray}
In this region, we always have $\eta^+<L$ and
\begin{eqnarray}
\ue^{-V(\eta^+)}=\ue^{-V(\ea)}\cos\pa.
\end{eqnarray}
Also, it is easy to see
\begin{eqnarray}
\exp(-G_{\eta^+,\ea})\leq \ue^0=1.
\end{eqnarray}
Hence, considering
\begin{eqnarray}
\ue^{-V(\eta^+)}=\ue^{-V(y)}\cos\phi'(y),
\end{eqnarray}
we have
\begin{eqnarray}
-\int_{\ea}^{\eta^+}\frac{1}{\sin\phi'(y)}\ud{y}&=&-\int_{\ea}^{\eta^+}\frac{1}{\sqrt{1-\ue^{2V(y)-2V(\eta^+)}}}\ud{y}\\
&=&-\int_{\ea}^{\eta^+}\frac{\rk-\e y}{\sqrt{(\rk-\e y)^2-(\rk-\e\eta^+)^2}}\ud{y}\no\\
&=&-\int_{\ea}^{\eta^+}\frac{\rk-\e y}{\sqrt{\e(\eta^+-y)(2\rk-\e y-\e\eta^+)}}\ud{y}\no\\
&\geq&-C\int_{\ea}^{\eta^+}\frac{1}{\sqrt{\e(\eta^+-y)}}\ud{y}\no\\
&=&-2C\sqrt{\frac{\eta^+-\ea}{\e}}.\no
\end{eqnarray}
Therefore, we know
\begin{eqnarray}
\exp(-G_{\eta^+,\ea})=\exp\bigg(-\int_{\ea}^{\eta^+}\frac{1}{\sin\phi'(y)}\ud{y}\bigg)\geq \exp\bigg(-2C\sqrt{\frac{\eta^+-\ea}{\e}}\bigg).
\end{eqnarray}
When $\pa\rt0$, since
\begin{eqnarray}
\ue^{-V(\eta^+)}=\ue^{-V(\ea)}\cos\pa,
\end{eqnarray}
we have $\eta^+\rt\ea$, which further implies
\begin{eqnarray}
\exp(-G_{\eta^+,\ea})\rt \ue^0=1.
\end{eqnarray}
Then we apply such result to $\k[p]$ to obtain when $(\ea,\pa)\rt(\eta,0)$
\begin{eqnarray}
\k[p]&=&p(\phi'(0))\exp(-G_{\eta^+,0}-G_{\eta^+,\ea})\rt p(\phi'(0))\exp(-G_{\eta,0}).
\end{eqnarray}
On the other hand, we consider $\t[H]$. We directly obtain
\begin{eqnarray}
\int_0^{\eta^+}\frac{H(\eta',\phi'(\eta'))}{\sin\phi'(\eta')}
\exp(-G_{\eta^+,\eta'}-G_{\eta^+,\ea})\ud{\eta'}\rt \int_0^{\eta}\frac{H(\eta',\phi'(\eta'))}{\sin\phi'(\eta')}
\exp(-G_{\eta,\eta'})\ud{\eta'}.
\end{eqnarray}
Also, we know
\begin{eqnarray}
\abs{\int_{\ea}^{\eta^+}\frac{H(\eta',R\phi'(\eta'))}{\sin\phi'(\eta')}\exp(-G_{\eta',\ea})\ud{\eta'}}
&\leq&\lnnm{H}\abs{\int_{\ea}^{\eta^+}\frac{1}{\sin\phi'(\eta')}\exp(-G_{\eta',\ea})\ud{\eta'}}\\
&\leq&\lnnm{H}\abs{\exp(-G_{\eta',\ea})\bigg|_{\ea}^{\eta^+}}\no\\
&=&\lnnm{H}\abs{\exp(-G_{\eta^+,\ea})-\ue^0}\rt 0.\no
\end{eqnarray}
Therefore, we have
\begin{eqnarray}
\t[H]\rt \int_0^{\eta}\frac{H(\eta',\phi'(\eta'))}{\sin\phi'(\eta')}
\exp(-G_{\ea,\eta'})\ud{\eta'}.
\end{eqnarray}
\ \\
Synthesis:\\
Summarizing above two cases, we always have
\begin{eqnarray}
\k[p]&\rt& p(\phi'(0))\exp(-G_{\eta,0}),\\
\t[H]&\rt& \int_0^{\eta}\frac{H(\eta',\phi'(\eta'))}{\sin\phi'(\eta')}
\exp(-G_{\eta,\eta'})\ud{\eta'}.
\end{eqnarray}
Hence, the solution is continuous.\\
\ \\
Step 4: Triple Point $(L,0)$.\\
This is the only point that three cases can be applied simultaneously. However, based on previous analysis, we know at this point, Case II and Case III provides exactly the same formula. Also, Case I and Case III is equivalent when taking limit $(\ea,\pa)\rt(\eta,0)$. Then this point is also continuous.
\end{proof}

\begin{theorem}\label{pt additional 2}
The derivatives of $\v$ are well-defined a.e. and satisfies
\begin{eqnarray}
\lnnm{\ue^{K_0\eta}\zeta\frac{\p\v}{\p\eta}}+\lnnm{\ue^{K_0\eta}\zeta\frac{\p\v}{\p\eta}}\leq C\abs{\ln(\e)}^8.
\end{eqnarray}
\end{theorem}
\begin{proof}
Based on the a priori estimate, it suffices to show the derivatives are well-defined. Consider the iteration of penalized $\e$-Milne problem for $\{\v_{\l}^m\}_{m=0}^{\infty}$ with $\v_{\l}^0=0$ and for $m\geq1$
\begin{eqnarray}
\left\{
\begin{array}{rcl}\displaystyle
\sin\phi\frac{\p \v_{\l}^{m}}{\p\eta}+F(\eta)\cos\phi\frac{\p
\v_{\l}^{m}}{\p\phi}+(1+\l)\v_{\l}^{m}-\bar\v_{\l}^{m-1}&=&S(\eta,\phi),\\
\v_{\l}^{m}(0,\phi)&=&p(\phi)\ \ \text{for}\ \ \sin\phi>0,\\
\v_{\l}^{m}(L,\phi)&=&\v_{\l}^{m}(L,R\phi).
\end{array}
\right.
\end{eqnarray}
Here we require $\l>0$. We divide the proof into several steps:\\
\ \\
Step 1: $m\rt\infty$ convergence.\\
Tracking along the characteristics, as we have shown in $\e$-Milne problem, we have $\v_{\l}^{m}\in L^{\infty}([0,L]\times[-\pi,\pi))$. Hence, it is easy to see each $\v_{\l}^{m}$ is uniquely determined.
Define $\z^m=\v_{\l}^{m}-\v_{\l}^{m-1}$ for $m\geq1$. Then $\z^m$ satisfies the equation
\begin{eqnarray}
\left\{
\begin{array}{rcl}\displaystyle
\sin\phi\frac{\p \z^{m}}{\p\eta}+F(\eta)\cos\phi\frac{\p
\z^{m}}{\p\phi}+(1+\l)\z^{m}-\bar\z^{m-1}&=&0,\\
\z^{m}(0,\phi)&=&0\ \ \text{for}\ \ \sin\phi>0,\\
\z^{m}(L,\phi)&=&\z^{m}(L,R\phi).
\end{array}
\right.
\end{eqnarray}
Based on previous analysis, we know
\begin{eqnarray}\label{pt 08}
\lnnm{\z^m}\leq\frac{1}{1+\l}\lnnm{\z^{m-1}}\leq\bigg(\frac{1}{1+\l}\bigg)^{m-1}\lnnm{\z^1}.
\end{eqnarray}
Since $\v_{\l}^0=0$, we have $\z^1=\v_{\l}^1$. Applying Lemma \ref{pt additional 1} for $H=0$, we know $\z^1$ is continuous. Using the proofs of Lemma \ref{pt lemma 1}, Lemma \ref{pt lemma 2} and Lemma \ref{pt lemma 6} with
\begin{eqnarray}
\v=\v_{\l}^1,\quad
\bar\v=0,\quad
p=0,
\end{eqnarray}
we get $\dfrac{\p\z^1}{\p\eta}$ and $\dfrac{\p\z^1}{\p\phi}$ are a.e. well-defined. However, the estimates from these lemmas are not strong enough to show the convergence of this iteration. Then we can use estimates of $\e$-Milne problem and the proofs of Lemma \ref{pt lemma 3} and Lemma \ref{pt lemma 4} with
\begin{eqnarray}
&&\a=\zeta\dfrac{\p\v_{\l}^1}{\p\eta},\quad
\tilde\a=0,\quad
p_{\a}=\cos\phi\frac{\p p}{\p\phi}+\v_{\l}^1(0,\phi),\quad
S_{\a}=\zeta\frac{\p S}{\p\eta},\\
&&\b=\zeta\dfrac{\p\v_{\l}^1}{\p\phi},\quad
p_{\b}=\frac{\p p}{\p\phi},\quad
S_{\b}=\zeta\frac{\p S}{\p\phi},
\end{eqnarray}
to see
\begin{eqnarray}
&&\lnnm{\zeta\frac{\p\z^{1}}{\p\eta}}+\lnnm{\zeta\frac{\p\z^{1}}{\p\phi}}+\lnnm{\z^1}\\
&\leq& C\bigg(\lnnm{S}+\lnm{p}+\lnnm{\frac{\p S}{\p\eta}}+\lnnm{\frac{\p S}{\p\phi}}+\lnm{\frac{\p p}{\p\phi}}\bigg).\no
\end{eqnarray}
and further
\begin{eqnarray}
\abs{\dfrac{\p\bar\z^{1}}{\p\eta}}\leq C\abs{\ln(\e)}^8(1+\abs{\ln(\e)}+\abs{\ln(\eta)}).
\end{eqnarray}
Note that here the extra $\l$ will not affect the result. Similarly, for each $m>1$, $\bar\z^{m-1}$ can be regarded as known. Applying Lemma \ref{pt additional 1} for $H=\bar\z^{m-1}$, we know $\z^m$ is continuous. Then we use the proofs of Lemma \ref{pt lemma 1}, Lemma \ref{pt lemma 2} and Lemma \ref{pt lemma 6} with
\begin{eqnarray}
\v=\z^m,\quad
S+\bar\v=\z^{m-1},\quad
p=0,
\end{eqnarray}
to confirm the derivatives $\dfrac{\p\z^m}{\p\eta}$ and $\dfrac{\p\z^m}{\p\phi}$ are a.e. well-defined. Then we utilize the proofs of Lemma \ref{pt lemma 3} and Lemma \ref{pt lemma 4} with
\begin{eqnarray}
\\
&&\a=\zeta\dfrac{\p\z^m}{\p\eta},\quad
\tilde\a=\frac{1}{2\pi}\int_{-\pi}^{\pi}\frac{\zeta(\eta,\phi)}{\zeta(\eta,\phi_{\ast})}\frac{\p\z^{m-1}(\eta,\phi_{\ast})}{\p\eta}\ud{\phi_{\ast}},\quad
p_{\a}=\z^m(0,\phi)-\bar\z^{m-1},\quad
S_{\a}=0,\no\\
&&\b=\zeta\dfrac{\p\z^m}{\p\phi},\quad
p_{\b}=0,\quad
S_{\b}=0,
\end{eqnarray}
to show
\begin{eqnarray}\label{pt 09}
\\
\lnnm{\zeta\frac{\p\z^m}{\p\eta}}+\lnnm{\zeta\frac{\p\z^m}{\p\phi}}&\leq&\d\bigg(\lnnm{\zeta\frac{\p\z^{m-1}}{\p\eta}}+\lnnm{\zeta\frac{\p\z^{m-1}}{\p\phi}}\bigg)
+C\abs{\ln(\e)}^8\lnnm{\z^m}.\no
\end{eqnarray}
for $0<\d<<1$ and
\begin{eqnarray}
\abs{\dfrac{\p\bar\z^{m}}{\p\eta}}\leq C\abs{\ln(\e)}^8(1+\abs{\ln(\e)}+\abs{\ln(\eta)}).
\end{eqnarray}
Therefore, combining (\ref{pt 08}) and (\ref{pt 09}), for fixed $\d\leq\dfrac{1}{1+\l}$, we have
\begin{eqnarray}
&&\lnnm{\zeta\frac{\p\z^m}{\p\eta}}+\lnnm{\zeta\frac{\p\z^m}{\p\phi}}\\
&\leq& \d^{m-1}\bigg(\lnnm{\zeta\frac{\p\z^{1}}{\p\eta}}+\lnnm{\zeta\frac{\p\z^{1}}{\p\phi}}\bigg)+Cm\abs{\ln(\e)}^8\bigg(\frac{1}{1+\l}\bigg)^{m-1}\lnnm{\z^1}\no\\
&\leq&Cm\abs{\ln(\e)}^8\bigg(\frac{1}{1+\l}\bigg)^{m-1}\bigg(\lnnm{\z^1}+\lnnm{\zeta\frac{\p\z^{1}}{\p\eta}}+\lnnm{\zeta\frac{\p\z^{1}}{\p\phi}}\bigg).\no
\end{eqnarray}
For fixed $\e$ and $\l>0$, when $m\rt\infty$, we know
\begin{eqnarray}
\lnnm{\z^m}+\lnnm{\zeta\frac{\p\z^m}{\p\eta}}+\lnnm{\zeta\frac{\p\z^m}{\p\phi}}\rt0,
\end{eqnarray}
and further for any $N>1$,
\begin{eqnarray}
\sum_{k=m}^{m+N}\bigg(\lnnm{\z^m}+\lnnm{\zeta\frac{\p\z^m}{\p\eta}}+\lnnm{\zeta\frac{\p\z^m}{\p\phi}}\bigg)\rt0.
\end{eqnarray}
Hence, $\v_{\l}^m$ is a Cauchy sequence. Thus we have $\v_{\l}^m\rt\v_{\l}$ strongly which satisfies
\begin{eqnarray}
&&\lnnm{\v_{\l}}+\lnnm{\zeta\frac{\p\v_{\l}}{\p\eta}}+\lnnm{\zeta\frac{\p\v_{\l}}{\p\phi}}\\
&\leq&\sum_{k=1}^{\infty}\bigg(\lnnm{\z^m}+\lnnm{\zeta\frac{\p\z^m}{\p\eta}}+\lnnm{\zeta\frac{\p\z^m}{\p\phi}}\bigg)\no\\
&\leq&\frac{C}{\l}\abs{\ln(\e)}^8\lnnm{\z^1}\no\\
&\leq& \frac{C}{\l}\abs{\ln(\e)}^8\bigg(\lnnm{S}+\lnm{p}+\lnnm{\frac{\p S}{\p\eta}}+\lnnm{\frac{\p S}{\p\phi}}+\lnm{\frac{\p h}{\p\phi}}\bigg).\no
\end{eqnarray}
Hence, we know $\dfrac{\p\v_{\l}}{\p\eta}$ and $\dfrac{\p\v_{\l}}{\p\phi}$ are a.e. well-defined.\\
\ \\
Step 2: $\l\rt0$ convergence.\\
We know $\v_{\l}$ satisfies the equation
\begin{eqnarray}
\left\{
\begin{array}{rcl}\displaystyle
\sin\phi\frac{\p \v_{\l}}{\p\eta}+F(\eta)\cos\phi\frac{\p
\v_{\l}}{\p\phi}+(1+\l)\v_{\l}-\bar\v_{\l}&=&S(\eta,\phi),\\
\v_{\l}(0,\phi)&=&p(\phi)\ \ \text{for}\ \ \sin\phi>0,\\
\v_{\l}(L,\phi)&=&\v_{\l}(L,R\phi).
\end{array}
\right.
\end{eqnarray}
Since its derivatives are a.e. well-defined, we can use the proof of Lemma \ref{pt lemma 3} and Lemma \ref{pt lemma 4} to show
\begin{eqnarray}
&&\lnnm{\v_{\l}}+\lnnm{\zeta\frac{\p\v_{\l}}{\p\eta}}+\lnnm{\zeta\frac{\p\v_{\l}}{\p\phi}}\\
&\leq& C\abs{\ln(\e)}^8\bigg(\lnnm{S}+\lnm{p}+\lnnm{\zeta\frac{\p S}{\p\eta}}+\lnnm{\zeta\frac{\p S}{\p\phi}}+\lnm{\frac{\p p}{\p\phi}}\bigg),\no
\end{eqnarray}
which is uniform in $\l$. Then we can define weak-$\ast$ limit $\v_{\l}\rt\v$ in weighted $W^{1,\infty}$, up to extracting a subsequence as $\l\rt0$. Also, the analysis of $\e$-Milne problem in \cite[Section 4]{AA003} reveals that $\v_{\l}\rt\v$ weakly in $L^2L^2$ as $\l\rt0$.
Hence, $\dfrac{\p\v}{\p\eta}$ and $\dfrac{\p\v}{\p\phi}$ are a.e. well-defined.
Therefore, we can apply the a priori estimates in Theorem \ref{pt theorem 1} and Theorem \ref{pt theorem 2} to obtain the desired result.
\end{proof}

\begin{corollary}\label{pt corollary}
We have
\begin{eqnarray}
\lnnm{\ue^{K_0\eta}\sin\phi\frac{\p\v}{\p\eta}(\eta,\phi)}\leq C\abs{\ln(\e)}^8.
\end{eqnarray}
\end{corollary}
\begin{proof}
This is a natural result of Theorem \ref{pt additional 2} since $\zeta(\eta,\phi)\geq \abs{\sin\phi}$.
\end{proof}

Now we pull $\tau$ dependence back and study the tangential derivative.
\begin{theorem}\label{Milne tangential}
We have
\begin{eqnarray}
\lnnm{\ue^{K_0\eta}\frac{\p\v}{\p\tau}(\eta,\tau,\phi)}\leq C\abs{\ln(\e)}^8.
\end{eqnarray}
\end{theorem}
\begin{proof}
Following a similar fashion in proof of Lemma \ref{pt additional 2}, using iteration and characteristics, we can show $\dfrac{\p\v}{\p\tau}$ is a.e. well-defined, so here we focus on the a priori estimate. Let $\w=\dfrac{\p\v}{\p\tau}$. Taking $\tau$ derivative on both sides of (\ref{Milne difference problem}), we have $\w$ satisfies the equation
\begin{eqnarray}\label{Milne tangential problem}
\left\{
\begin{array}{rcl}\displaystyle
\sin\phi\frac{\p \w}{\p\eta}+F(\eta)\cos\phi\frac{\p
\w}{\p\phi}+\w-\bar\w&=&\dfrac{\p S}{\p\tau}(\eta,\tau,\phi)+\dfrac{R'_{\kappa}(\tau)}{R_{\kappa}(\tau)-\e\eta}\bigg(F(\eta)\cos\phi\dfrac{\p
\v}{\p\phi}\bigg),\\\rule{0ex}{1.5em}
\w(0,\tau,\phi)&=&\dfrac{\p p}{\p\tau}(\tau,\phi)\ \ \text{for}\ \ \sin\phi>0,\\\rule{0ex}{1.5em}
\w(L,\tau,\phi)&=&\w(L,\tau,R\phi),
\end{array}
\right.
\end{eqnarray}
where $R'_{\kappa}$ is the $\tau$ derivative of $R_{\kappa}$. Our assumptions on $S$ verify
\begin{eqnarray}
\lnnm{\ue^{K_0\eta}\frac{\p S}{\p\tau}(\eta,\tau,\phi)}&\leq&C.
\end{eqnarray}
For $\eta\in[0,L]$, we have
\begin{eqnarray}
\dfrac{R'_{\kappa}(\tau)}{R_{\kappa}(\tau)-\e\eta}\leq C\max_{\tau}R'_{\kappa}(\tau)\leq C.
\end{eqnarray}
Based on Corollary \ref{pt corollary} and the equation (\ref{Milne difference problem}), we know
\begin{eqnarray}
\lnnm{\ue^{K_0\eta}\bigg(F(\eta)\cos\phi\frac{\p\v}{\p\phi}\bigg)(\eta,\tau,\phi)}\leq C\abs{\ln(\e)}^8.
\end{eqnarray}
Therefore, the source term in the equation (\ref{Milne tangential problem}) is in $L^{\infty}$ and decays exponentially. By Theorem \ref{Milne theorem 3}, we have
\begin{eqnarray}
\lnnm{\ue^{K_0\eta}\w(\eta,\tau,\phi)}\leq C\abs{\ln(\e)}^8,
\end{eqnarray}
which is the desired estimate.
\end{proof}

\subsection{Diffusive Boundary}

In this subsection, we come back to the $\e$-Milne problem with diffusive boundary.
In \cite[Section 6]{AA003}, it has been proved that
\begin{lemma}\label{Milne lemma 1.}
In order for the equation (\ref{Milne problem.}) to have a solution
$f(\eta,\tau,\phi)\in L^{\infty}([0,L]\times[-\pi,\pi)\times[-\pi,\pi))$, the boundary data $h$
and the source term $S$ must satisfy the compatibility condition
\begin{eqnarray}\label{Milne compatibility condition}
\int_{\sin\phi>0}h(\tau,\phi)\sin\phi\ud{\phi}
+\int_0^{L}\int_{-\pi}^{\pi}\ue^{-V(s)}S(s,\tau,\phi)\ud{\phi}\ud{s}=0.
\end{eqnarray}
In particular, if $S=0$, then the compatibility condition reduces to
\begin{eqnarray}\label{Milne reduced compatibility condition}
\int_{\sin\phi>0}h(\tau,\phi)\sin\phi\ud{\phi}=0.
\end{eqnarray}
\end{lemma}
It is easy to see if $f$ is a solution to (\ref{Milne problem.}),
then $f+C$ is also a solution for any constant $C$. Hence, in order
to obtain a unique solution, we need a normalization condition
\begin{eqnarray}\label{Milne normalization}
\pp[f](0,\tau)=0.
\end{eqnarray}
The following lemma in \cite[Section 6]{AA003} tells us the problem (\ref{Milne problem.}) can
be reduced to the $\e$-Milne problem with in-flow boundary (\ref{Milne problem}).
\begin{lemma}\label{Milne lemma 2.}
If the boundary data $h$ and $S$ satisfy the compatibility condition
(\ref{Milne compatibility condition}), then the solution $f$ to the
$\e$-Milne problem (\ref{Milne problem}) with in-flow boundary as
$f=h$ on $\sin\phi>0$ is also a solution to the $\e$-Milne problem
(\ref{Milne problem.}) with diffusive boundary, which satisfies the
normalization condition (\ref{Milne normalization}). Furthermore,
this is the unique solution to (\ref{Milne problem.}) among the
functions satisfying (\ref{Milne normalization}) and
$\tnnm{f(\eta,\tau,\phi)-f_L(\tau)}\leq C$.
\end{lemma}
In summary, based on above analysis, we can utilize the known result
for $\e$-Milne problem (\ref{Milne problem}) to obtain the
desired results of the solution to the
$\e$-Milne problem (\ref{Milne problem.}).
\begin{theorem}\label{Milne theorem 1.}
There exists a unique solution $f(\eta,\tau,\phi)$ to the $\e$-Milne problem
(\ref{Milne problem.}) with the normalization condition (\ref{Milne
normalization}) satisfying
\begin{eqnarray}
\tnnm{f(\eta,\tau,\phi)-f_L(\tau)}\leq C.
\end{eqnarray}
\end{theorem}
\begin{theorem}\label{Milne theorem 2.}
The unique solution $f(\eta,\tau,\phi)$ to the $\e$-Milne problem
(\ref{Milne problem.}) with the normalization condition (\ref{Milne
normalization}) satisfying
\begin{eqnarray}
\lnnm{f(\eta,\tau,\phi)-f_L(\tau)}\leq C.
\end{eqnarray}
\end{theorem}
\begin{theorem}\label{Milne theorem 3.}
There exists $K_0>0$ such that the solution $f(\eta,\tau,\phi)$ to the
$\e$-Milne problem (\ref{Milne problem.}) with the normalization
condition (\ref{Milne normalization}) satisfies
\begin{eqnarray}
\lnnm{\ue^{K_0\eta}\bigg(f(\eta,\tau,\phi)-f_L(\tau)\bigg)}\leq C.
\end{eqnarray}
\end{theorem}
\begin{theorem}\label{Milne tangential.}
The unique solution $f(\eta,\tau,\phi)$ to the
$\e$-Milne problem (\ref{Milne problem.}) with the normalization
condition (\ref{Milne normalization}) satisfies
\begin{eqnarray}
\lnnm{\ue^{K_0\eta}\frac{\p(f-f_L)}{\p\tau}(\eta,\tau,\phi)}\leq C\abs{\ln(\e)}^8.
\end{eqnarray}
\end{theorem}

\section{Remainder Estimate}

In this section, we consider the the remainder equation for $u(\vx,\vw)$ as
\begin{eqnarray}\label{neutron}
\left\{
\begin{array}{rcl}\displaystyle
\e\vw\cdot\nx u+u-\bar
u&=&f(\vx,\vw)\ \ \text{in}\ \ \Omega,\\
u(\vx_0,\vw)&=&\pp[u](\vx_0)+h(\vx_0,\vw)\ \ \text{for}\ \
\vw\cdot\vn<0\ \ \text{and}\ \ \vx_0\in\p\Omega,
\end{array}
\right.
\end{eqnarray}
where
\begin{eqnarray}
\bar u(\vx)=\frac{1}{2\pi}\int_{\s^1}u(\vx,\vw)\ud{\vw},
\end{eqnarray}
\begin{eqnarray}
\pp[u](\vx_0)=\frac{1}{2}\int_{\vw\cdot\vn>0}u(\vx_0,\vw)(\vw\cdot\vn)\ud{\vw},
\end{eqnarray}
$\vn$ is the outward unit normal vector, with the Knudsen number $0<\e<<1$. To guarantee uniqueness, we need the normalization condition
\begin{eqnarray}\label{normalization.}
\int_{\Omega\times\s^1}u(\vx,\vw)\ud{\vw}\ud{\vx}=0.
\end{eqnarray}
Also, the data $f$ and $h$ satisfy the compatibility condition
\begin{eqnarray}\label{compatibility.}
\int_{\Omega\times\s^1}f(\vx,\vw)\ud{\vw}\ud{\vx}+\e\int_{\p\Omega}\int_{\vw\cdot\vn<0}h(\vx_0,\vw)(\vw\cdot\vn)\ud{\vw}\ud{\vx_0}=0.
\end{eqnarray}
We define the $L^p$ norm with $1\leq p<\infty$ and $L^{\infty}$ norms in $\Omega\times\s^1$ as
usual:
\begin{eqnarray}
\nm{f}_{L^p(\Omega\times\s^1)}&=&\bigg(\int_{\Omega}\int_{\s^1}\abs{f(\vx,\vw)}^p\ud{\vw}\ud{\vx}\bigg)^{1/p},\\
\nm{f}_{L^{\infty}(\Omega\times\s^1)}&=&\sup_{(\vx,\vw)\in\Omega\times\s^1}\abs{f(\vx,\vw)}.
\end{eqnarray}
Define the $L^p$ norm with $1\leq p<\infty$ and $L^{\infty}$ norms on the boundary as follows:
\begin{eqnarray}
\nm{f}_{L^p(\Gamma)}&=&\bigg(\iint_{\Gamma}\abs{f(\vx,\vw)}^p\abs{\vw\cdot\vn}\ud{\vw}\ud{\vx}\bigg)^{1/p},\\
\nm{f}_{L^p(\Gamma^{\pm})}&=&\bigg(\iint_{\Gamma^{\pm}}\abs{f(\vx,\vw)}^p\abs{\vw\cdot\vn}\ud{\vw}\ud{\vx}\bigg)^{1/p},\\
\nm{f}_{L^{\infty}(\Gamma)}&=&\sup_{(\vx,\vw)\in\Gamma}\abs{f(\vx,\vw)},\\
\nm{f}_{L^{\infty}(\Gamma^{\pm})}&=&\sup_{(\vx,\vw)\in\Gamma^{\pm}}\abs{f(\vx,\vw)}.
\end{eqnarray}

\subsection{Preliminaries}

In order to show the $L^{\infty}$ estimates of the
equation (\ref{neutron}), we start with some preparations with the
transport equation.
\begin{lemma}\label{well-posedness lemma 1}
Assume $f(\vx,\vw)\in L^{\infty}(\Omega\times\s^1)$ and
$h(x_0,\vw)\in L^{\infty}(\Gamma^-)$. Then for the
transport equation
\begin{eqnarray}\label{penalty equation}
\left\{
\begin{array}{rcl}
\e\vw\cdot\nx u+u&=&f(\vx,\vw)\ \ in\ \ \Omega\\
u(\vx_0,\vw)&=&h(\vx_0,\vw)\ \ for\ \ \vx_0\in\p\Omega\ \ and\ \vw\cdot\vn<0,
\end{array}
\right.
\end{eqnarray}
there exists a unique
solution $u(\vx,\vw)\in L^{\infty}(\Omega\times\s^1)$
satisfying
\begin{eqnarray}
\im{u}{\Omega\times\s^1}\leq
\im{f}{\Omega\times\s^1}+\im{h}{\Gamma^-}
\end{eqnarray}
\end{lemma}
\begin{proof}
The characteristics $(X(s),W(s))$ of the equation (\ref{penalty
equation}) which goes through $(\vx,\vw)$ is defined by
\begin{eqnarray}\label{character}
\left\{
\begin{array}{lll}
\dfrac{\ud{X(s)}}{\ud{s}}=\e W(s),
\quad\dfrac{\ud{W(s)}}{\ud{s}}=0,\\ \rule{0ex}{2.0em}
(X(0),W(0))=(\vx,\vw).
\end{array}
\right.
\end{eqnarray}
which implies
\begin{eqnarray}
X(s)=\vx+(\e\vw)s,\quad W(s)=\vw.
\end{eqnarray}
Along the characteristics, the equation (\ref{penalty equation}) takes the form
\begin{eqnarray}
\left\{
\begin{array}{rcl}
\dfrac{\ud{u}}{\ud{s}}+ u&=&f(\vx,\vw)\ \ \text{in}\ \ \Omega\\\rule{0ex}{1.0em}
u(\vx_b,\vw)&=&h(\vx_b,\vw)\ \ \text{for}\ \ \vw\cdot\vn<0,
\end{array}
\right.
\end{eqnarray}
where
\begin{eqnarray}
t_b(\vx,\vw)&=&\inf\{t\geq0: \vx-\e t\vw\in\p\Omega\},\\
x_b(\vx,\vw)&=&\vx-\e t_b\vw.
\end{eqnarray}
We rewrite the equation (\ref{penalty equation}) along the characteristics as
\begin{eqnarray}
u(\vx,\vw)&=&h(\vx-\e
t_b\vw,\vw)\ue^{- t_b}+\int_{0}^{t_b}f(\vx-\e(t_b-s)\vw,\vw)\ue^{- (t_b-s)}\ud{s}.
\end{eqnarray}
The existence and uniqueness directly follows from above formulation. Also, we have
\begin{eqnarray}
\im{u}{\Omega\times\s^1}\leq
\im{h}{\Gamma^-}+\im{f}{\Omega\times\s^1}.
\end{eqnarray}
Hence, our desired result is obvious.
\end{proof}

\subsection{$L^2$ Estimate}

In this subsection, we start from the preliminary equation (\ref{penalty equation}) and take $\bar u$ and $\pp[u]$ into consideration.
\begin{lemma}\label{well-posedness lemma 2}
Define the near-grazing set of $\Gamma^+$ or $\Gamma^-$ as
\begin{eqnarray}
\Gamma_{\pm}^{\delta}=\left\{(\vx,\vw)\in\Gamma^{\pm}:
\abs{\vn(\vx)\cdot\vw}\leq\delta\right\}.
\end{eqnarray}
Then
\begin{eqnarray}
\nm{f{\bf{1}}_{\Gamma^{\pm}\backslash\Gamma_{\pm}^{\delta}}}_{L^1(\Gamma^{\pm})}\leq
C(\delta)\bigg(\nm{f}_{L^1(\Omega\times\s^1)}+\nm{\vw\cdot\nx f}_{L^1(\Omega\times\s^1)}\bigg).
\end{eqnarray}
\end{lemma}
\begin{proof}
See the proof of \cite[Lemma 2.1]{Esposito.Guo.Kim.Marra2013}.
\end{proof}
\begin{lemma}(Green's Identity)\label{well-posedness lemma 3}
Assume $f(\vx,\vw),\ g(\vx,\vw)\in L^2(\Omega\times\s^1)$ and
$\vw\cdot\nx f,\ \vw\cdot\nx g\in L^2(\Omega\times\s^1)$ with $f,\
g\in L^2(\Gamma)$. Then
\begin{eqnarray}
\iint_{\Omega\times\s^1}\bigg((\vw\cdot\nx f)g+(\vw\cdot\nx
g)f\bigg)\ud{\vx}\ud{\vw}=\int_{\Gamma}fg\ud{\gamma},
\end{eqnarray}
where $\ud{\gamma}=(\vw\cdot\vn)\ud{s}$ on the boundary.
\end{lemma}
\begin{proof}
See the proof of \cite[Chapter 9]{Cercignani.Illner.Pulvirenti1994} and
\cite{Esposito.Guo.Kim.Marra2013}.
\end{proof}
\begin{lemma}\label{LT estimate}
Assume $f(\vx,\vw)\in L^{\infty}(\Omega\times\s^1)$ and $h(x_0,\vw)\in
L^{\infty}(\Gamma^-)$. Then for the transport equation (\ref{neutron}),
there exists a unique solution $u(\vx,\vw)\in L^2(\Omega\times\s^1)$
satisfying
\begin{eqnarray}
\frac{1}{\e}\nm{(1-\pp)[u]}^2_{L^2(\Gamma^+)}+\nm{u}_{L^2(\Omega\times\s^1)}\leq C\bigg(\frac{1}{\e^2}\nm{f}_{L^2(\Omega\times\s^1)}+\frac{1}{\e}\nm{h}_{L^2(\Gamma^-)}\bigg),
\end{eqnarray}
\end{lemma}
\begin{proof}
We divide the proof into several steps:\\
\ \\
Step 1: Penalized equation.\\
We first consider the penalized equation
\begin{eqnarray}\label{penalty neutron}
\\
\left\{
\begin{array}{rcl}\displaystyle
\e\vw\cdot\nx u_{j,\l}+ (1+\l)u_{j,\l}-\bar
u_{j,\l}&=&f(\vx,\vw)\ \ \text{in}\ \ \Omega,\\
u_{j,\l}(\vx_0,\vw)&=&\left(1-\dfrac{1}{j}\right)\pp[u_{j,\l}](\vx_0)+h(\vx_0,\vw)\ \ \text{for}\ \
\vw\cdot\vn<0\ \ \text{and}\ \ \vx_0\in\p\Omega,\no
\end{array}
\right.
\end{eqnarray}
for $\l>0$, $j\in\mathbb{N}$ and $j\geq \dfrac{2}{\l}$. We iteratively construct
an approximating sequence $\{u^k_j\}_{k=0}^{\infty}$ where $u^0_j=0$ and
\begin{eqnarray}\label{penalty iteration}
\\
\left\{
\begin{array}{rcl}\displaystyle
\e\vw\cdot\nx u_{j,\l}^k+(1+\l)u_{j,\l}^k-\bar u_{j,\l}^{k-1}&=&f(\vx,\vw)\ \ \text{in}\ \ \Omega,\\
u_{j,\l}^k(\vx_0,\vw)&=&\left(1-\dfrac{1}{j}\right)\pp[u_{j,\l}^{k-1}](\vx_0)+h(\vx_0,\vw)\ \
\text{for}\ \ \vx_0\in\p\Omega\ \ \text{and}\ \vw\cdot\vn<0.\no
\end{array}
\right.
\end{eqnarray}
By Lemma \ref{well-posedness lemma 1}, this sequence is
well-defined and $\im{u_{j,\l}^k}{\Omega\times\s^1}<\infty$.
We rewrite equation (\ref{penalty iteration}) along the
characteristics as
\begin{eqnarray}
u_{j,\l}^k(\vx,\vw)&=&\left(h+\left(1-\dfrac{1}{j}\right)\pp[u_{j,\l}^{k-1}]\right)(\vx-\e
t_b\vw,\vw)\ue^{-(1+\l)t_b}\\
&&+\int_{0}^{t_b}(f+\bar
u_{j,\l}^{k-1})(\vx-\e(t_b-s)\vw,\vw)\ue^{-(1+\l)(t_b-s)}\ud{s}\no.
\end{eqnarray}
We define the difference $v^k=u^{k}_{j,\l}-u^{k-1}_{j,\l}$ for $k\geq1$. Then
$v^k_{j,\l}$ satisfies
\begin{eqnarray}
v^{k+1}_{j,\l}(\vx,\vw)&=&\left(1-\dfrac{1}{j}\right)\pp[v^{k}_{j,\l}](\vx-\e
t_b\vw,\vw)\ue^{-(1+\l)t_b}+\int_{0}^{t_b}\bar
v^{k}_{j,\l}(\vx-\e(t_b-s)\vw,\vw)\ue^{-(1+\l)(t_b-s)}\ud{s}.
\end{eqnarray}
Since $\im{\bar
v^k_{j,\l}}{\Omega\times\s^1}\leq\im{v^k_{j,\l}}{\Omega\times\s^1}$ and $\im{\pp
[v^k_{j,\l}]}{\Gamma^+}\leq \im{v^k_{j,\l}}{\Omega\times\s^1}$, we can directly estimate
\begin{eqnarray}
\im{v^{k+1}_{j,\l}}{\Omega\times\s^1}&\leq&\ue^{-(1+\l)t_b}\left(1-\dfrac{1}{j}\right)\im{v^{k+1}_{j,\l}}{\Omega\times\s^1}
+\im{v^{k}_{j,\l}}{\Omega\times\s^1}\int_0^{t_b}\ue^{-(1+\l)(t_b-s)}\ud{s}\\
&\leq&\ue^{-(1+\l)t_b}\left(1-\dfrac{1}{j}\right)\im{v^{k+1}_{j,\l}}{\Omega\times\s^1}+\frac{1}{1+\l}(1-\ue^{-(1+\l)t_b})\im{v^{k}_{j,\l}}{\Omega\times\s^1}\no\\
&\leq&\frac{1}{1+\l}\im{v^{k}_{j,\l}}{\Omega\times\s^1},\no
\end{eqnarray}
since $j\geq\dfrac{2}{\l}$.
Hence, we naturally have
\begin{eqnarray}
\im{v^{k+1}_{j,\l}}{\Omega\times\s^1}&\leq&\left(1-\dfrac{1}{j}\right)\im{v^{k}_{j,\l}}{\Omega\times\s^1}.
\end{eqnarray}
Thus, this is a contraction iteration. Considering $v^1=u^1$, we
have
\begin{eqnarray}
\im{v^{k}_{j,\l}}{\Omega\times\s^1}\leq\left(1-\dfrac{1}{j}\right)^{k-1}\im{u^{1}_{j,\l}}{\Omega\times\s^1}.
\end{eqnarray}
for $k\geq1$. Therefore, $u^k_{j,\l}$ converges strongly in $L^{\infty}$ to
the limiting solution $u_{j,\l}$ satisfying
\begin{eqnarray}\label{well-posedness temp 1}
\im{u_{j,\l}}{\Omega\times\s^1}\leq\sum_{k=1}^{\infty}\im{v^{k}_{j,\l}}{\Omega\times\s^1}\leq j\im{u^1_{j,\l}}{\Omega\times\s^1}.
\end{eqnarray}
Since $u^1_{j,\l}$ satisfies the equation
\begin{eqnarray}
u^1_{j,\l}(\vx,\vw)&=&h(\vx-\e
t_b\vw,\vw)\ue^{-(1+\l)t_b}+\int_{0}^{t_b}f(\vx-\e(t_b-s)\vw,\vw)\ue^{-(1+\l)(t_b-s)}\ud{s}\nonumber.
\end{eqnarray}
Based on Lemma \ref{well-posedness lemma 1}, we can directly
estimate
\begin{eqnarray}\label{well-posedness temp 2}
\im{u^1_{j,\l}}{\Omega\times\s^1}\leq
\im{f}{\Omega\times\s^1}+\im{h}{\Gamma^-}.
\end{eqnarray}
Combining (\ref{well-posedness temp 1}) and (\ref{well-posedness
temp 2}), we can naturally obtain the existence and the estimate
\begin{eqnarray}
\im{u_{j,\l}}{\Omega\times\s^1}\leq j\bigg(\im{f}{\Omega\times\s^1}+\im{h}{\Gamma^-}\bigg).
\end{eqnarray}
This justify the well-posedness of $u_{j,\l}$.
Note that when $\l\rt0$ or $j\rt\infty$, this estimate blows up. Hence, we have to find a uniform estimate in $\l$ and $j$.\\
\ \\
Step 2: Energy Estimate of $u_{\l,j}$.\\
Multiplying $u_{j,\l}$ on both sides of (\ref{penalty neutron}) and integrating over $\Omega\times\s^1$, by Lemma \ref{well-posedness lemma 3}, we get the energy estimate
\begin{eqnarray}
\half\e\int_{\Gamma}\abs{u_{j,\l}}^2\ud{\gamma}+\l\nm{u_{j,\l}}_{L^2(\Omega\times\s^1)}^2+\nm{u_{j,\l}-\bar
u_{j,\l}}_{L^2(\Omega\times\s^1)}^2=\iint_{\Omega\times\s^1}fu_{j,\l}.
\end{eqnarray}
A direct computation shows
\begin{eqnarray}
&&\half\e\int_{\Gamma}\abs{u_{j,\l}}^2\ud{\gamma}\\
&=&\half\e\nm{u_{j,\l}}^2_{L^2(\Gamma^+)}-\half\e\nm{\left(1-\dfrac{1}{j}\right)\pp[u_{j,\l}]
+h}^2_{L^2(\Gamma^-)}\no\\
&=&\half\e\bigg(\nm{u_{j,\l}}^2_{L^2(\Gamma^+)}-\nm{\left(1-\dfrac{1}{j}\right)\pp[u_{j,\l}]}^2_{L^2(\Gamma^-)}\bigg)
-\half\e\nm{h}_{L^2(\Gamma^-)}^2-\e\left(1-\dfrac{1}{j}\right)\int_{\Gamma^-}h\pp[u_{j,\l}]\abs{\vw\cdot\vn}\ud{\gamma}.\nonumber
\end{eqnarray}
Hence, we have
\begin{eqnarray}
&&\half\e\bigg(\nm{u_{j,\l}}^2_{L^2(\Gamma^+)}-\nm{\left(1-\dfrac{1}{j}\right)\pp[u_{j,\l}]}^2_{L^2(\Gamma^-)}\bigg)+\l\nm{u_{j,\l}}_{L^2(\Omega\times\s^1)}^2+\nm{u_{j,\l}-\bar
u_{j,\l}}_{L^2(\Omega\times\s^1)}^2\\
&=&\iint_{\Omega\times\s^1}fu_{j,\l}+\half\e\nm{h}_{L^2(\Gamma^-)}^2+\e\left(1-\dfrac{1}{j}\right)\int_{\Gamma^-}h\pp[u_{j,\l}]\abs{\vw\cdot\vn}\ud{\gamma}.\no
\end{eqnarray}
Noting the fact that
\begin{eqnarray}
\e\bigg(\nm{u_{j,\l}}^2_{L^2(\Gamma^+)}-\nm{\pp[u_{j,\l}]}^2_{L^2(\Gamma^-)}\bigg)=\e\nm{(1-\pp)[u_{j,\l}]}^2_{L^2(\Gamma^+)},
\end{eqnarray}
we deduce
\begin{eqnarray}
&&\half\e\nm{(1-\pp)[u_{j,\l}]}^2_{L^2(\Gamma^+)}+\l\nm{u_{j,\l}}_{L^2(\Omega\times\s^1)}^2+\nm{u_{j,\l}-\bar
u_{j,\l}}_{L^2(\Omega\times\s^1)}^2\\
&\leq&\iint_{\Omega\times\s^1}fu_{j,\l}+\half\e\nm{h}_{L^2(\Gamma^-)}^2+\e\int_{\Gamma^-}h\pp[u_{j,\l}]\abs{\vw\cdot\vn}\ud{\gamma}.\no
\end{eqnarray}
Applying Cauchy's inequality, we obtain for $\eta>0$ sufficiently small,
\begin{eqnarray}
\e\int_{\Gamma^-}h\pp[u_{j,\l}]\abs{\vw\cdot\vn}\ud{\gamma}\leq \frac{4}{\eta}\nm{h}_{L^2(\Gamma^-)}^2+\e^2\eta\nm{\pp[u_{j,\l}]}^2_{L^2(\Gamma^-)},
\end{eqnarray}
which further implies
\begin{eqnarray}\label{wt 02}
&&\half\e\nm{(1-\pp)[u_{j,\l}]}^2_{L^2(\Gamma^+)}+\l\nm{u_{j,\l}}_{L^2(\Omega\times\s^1)}^2+\nm{u_{j,\l}-\bar
u_{j,\l}}_{L^2(\Omega\times\s^1)}^2\\
&\leq&\iint_{\Omega\times\s^1}fu_{j,\l}+\bigg(1+\frac{4}{\eta}\bigg)\nm{h}_{L^2(\Gamma^-)}^2+\e^2\eta\nm{\pp[u_{j,\l}]}^2_{L^2(\Gamma^-)}.\no
\end{eqnarray}
Now the only difficulty is $\e^2\eta\nm{\pp[u_{j,\l}]}^2_{L^2(\Gamma^-)}$, which we cannot bound directly.\\
\ \\
Step 3: Estimate of $\nm{\pp[u_{j,\l}]}^2_{L^2(\Gamma^-)}$.\\
Multiplying $u_{j,\l}$ on both sides of (\ref{penalty neutron}), we have
\begin{eqnarray}\label{wt 03}
\half\e\vw\cdot\nx(u_{j,\l}^2)=-\l u_{j,\l}^2-u_{j,\l}(u_{j,\l}-\bar u_{j,\l})+fu_{j,\l}.
\end{eqnarray}
Taking absolute value on both sides of (\ref{wt 03}) and integrating over $\Omega\times\s^1$, we get
\begin{eqnarray}
\nm{\vw\cdot\nx(u_{j,\l}^2)}_{L^1(\Omega\times\s^1)}\leq \frac{2\l}{\e}\nm{u_{j,\l}}_{L^2(\Omega\times\s^1)}^2+\frac{2}{\e}\nm{u_{j,\l}-\bar
u_{j,\l}}_{L^2(\Omega\times\s^1)}^2+\frac{2}{\e}\iint_{\Omega\times\s^1}fu_{j,\l}.
\end{eqnarray}
Based on (\ref{wt 02}), we can further obtain
\begin{eqnarray}\label{wt 11}
\nm{\vw\cdot\nx(u_{j,\l}^2)}_{L^1(\Omega\times\s^1)}\leq \frac{1}{\e}\bigg(1+\frac{4}{\eta}\bigg)\nm{h}_{L^2(\Gamma^-)}^2+\e\eta\nm{\pp[u_{j,\l}]}^2_{L^2(\Gamma^-)}+\frac{4}{\e}\iint_{\Omega\times\s^1}fu.
\end{eqnarray}
Hence, by Lemma \ref{well-posedness lemma 2}, (\ref{wt 02}) and (\ref{wt 11}), we know for given $\d>0$
\begin{eqnarray}\label{wt 04}
\nm{u_{j,\l}^2{\bf{1}}_{\Gamma^{\pm}\backslash\Gamma_{\pm}^{\delta}}}_{L^1(\Gamma^{\pm})}&\leq&
C(\delta)\bigg(\nm{u_{j,\l}}_{L^2(\Omega\times\s^1)}^2+\nm{\vw\cdot\nx(u_{j,\l}^2)}_{L^1(\Omega\times\s^1)}\bigg)\\
&\leq&C(\delta)\Bigg(\frac{1}{\l}\bigg(1+\frac{4}{\eta}\bigg)\nm{h}_{L^2(\Gamma^-)}^2
+\frac{\e^2\eta}{\l}\nm{\pp[u_{j,\l}]}^2_{L^2(\Gamma^-)}+\frac{1}{\l}\iint_{\Omega\times\s^1}fu_{j,\l}.\Bigg).\no
\end{eqnarray}
Noting the fact that
\begin{eqnarray}
\nm{\pp[u_{j,\l}{\bf{1}}_{\Gamma_{+}\backslash\Gamma_{+}^{\delta}}]}_{L^2(\Gamma^-)}\leq \nm{
u_{j,\l}{\bf{1}}_{\Gamma_{+}\backslash\Gamma_{+}^{\delta}}}_{L^2(\Gamma^-)},
\end{eqnarray}
and for $\d$ sufficiently small, we have
\begin{eqnarray}
\nm{\pp[u_{j,\l}{\bf{1}}_{\Gamma_{+}\backslash\Gamma_{+}^{\delta}}]}_{L^2(\Gamma^-)}\geq \half\nm{\pp[u_{j,\l}]}_{L^2(\Gamma^-)}.
\end{eqnarray}
Combining with (\ref{wt 04}), we naturally obtain
\begin{eqnarray}
\nm{\pp[u_{j,\l}]}_{L^2(\Gamma^-)}^2&\leq&2\nm{\pp[u_{j,\l}{\bf{1}}_{\Gamma_{+}\backslash\Gamma_{+}^{\delta}}]}_{L^2(\Gamma^-)}^2\leq2\nm{
u_{j,\l}{\bf{1}}_{\Gamma_{+}\backslash\Gamma_{+}^{\delta}}}_{L^2(\Gamma^-)}^2\\
&\leq&C(\delta)\Bigg(\frac{1}{\l}\bigg(1+\frac{1}{\eta}\bigg)\nm{h}_{L^2(\Gamma^-)}^2+\frac{\e^2\eta}{\l}\nm{\pp[u_{j,\l}]}^2_{L^2(\Gamma^-)}
+\frac{1}{\l}\iint_{\Omega\times\s^1}fu_{j,\l}.\Bigg).\no
\end{eqnarray}
For fixed $\d$, taking $\eta>0$ sufficiently small, we obtain
\begin{eqnarray}\label{wt 13}
\nm{\pp[u_{j,\l}]}_{L^2(\Gamma^-)}^2
&\leq&C\bigg(\frac{1}{\l}\nm{h}_{L^2(\Gamma^-)}^2+\frac{1}{\l}\iint_{\Omega\times\s^1}fu_{j,\l}\bigg).
\end{eqnarray}
Plugging (\ref{wt 13}) into (\ref{wt 02}), we deduce
\begin{eqnarray}\label{wt 08}
\\
&&\half\e\nm{(1-\pp)[u_{j,\l}]}^2_{L^2(\Gamma^+)}+\l\nm{u_{j,\l}}_{L^2(\Omega\times\s^1)}^2+\nm{u_{j,\l}-\bar
u_{j,\l}}_{L^2(\Omega\times\s^1)}^2\leq \frac{C\e^2}{\l}\bigg(\iint_{\Omega\times\s^1}fu_{j,\l}+\nm{h}_{L^2(\Gamma^-)}^2\bigg).\no
\end{eqnarray}
\ \\
Step 4: Limit $j\rt\infty$.\\
Naturally, based on (\ref{wt 08}), we deduce
\begin{eqnarray}\label{wt 09}
\nm{u_{j,\l}}_{L^2(\Omega\times\s^1)}^2\leq C\bigg(\frac{1}{\l}\iint_{\Omega\times\s^1}fu+\frac{1}{\l}\nm{h}_{L^2(\Gamma^-)}^2\bigg).
\end{eqnarray}
Applying Cauchy's inequality, we obtain for $C_0>0$ sufficiently small
\begin{eqnarray}\label{wt 10}
\frac{1}{\l}\iint_{\Omega\times\s^1}fu_{j,\l}\leq \frac{4}{C_0\l^2}\nm{f}_{L^2(\Omega\times\s^1)}^2+C_0\nm{u_{j,\l}}_{L^2(\Omega\times\s^1)}^2.
\end{eqnarray}
Combining (\ref{wt 09}) and (\ref{wt 10}), we obtain
\begin{eqnarray}
\nm{u_{j,\l}}_{L^2(\Omega\times\s^1)}^2\leq C\bigg(\frac{1}{\l^2}\nm{f}_{L^2(\Omega\times\s^1)}^2+\frac{1}{\l}\nm{h}_{L^2(\Gamma^-)}^2\bigg),
\end{eqnarray}
which further implies
\begin{eqnarray}\label{wt 12}
\nm{u_{j,\l}}_{L^2(\Omega\times\s^1)}\leq C\bigg(\frac{1}{\l}\nm{f}_{L^2(\Omega\times\s^1)}+\frac{1}{\l^{1/2}}\nm{h}_{L^2(\Gamma^-)}\bigg).
\end{eqnarray}
Since this estimate is uniform in $j$, we may take weak limit $u_{j,\l}\rightharpoonup u_{\l}$ in $L^2(\Omega\times\s^1)$ as $j\rt\infty$. By Lemma \ref{well-posedness lemma 3} and weak semi-continuity, there exists a unique solution $u_{\l}$ to the equation
\begin{eqnarray}\label{penalty neutron.}
\left\{
\begin{array}{rcl}\displaystyle
\e\vw\cdot\nx u_{\l}+ (1+\l)u_{\l}-\bar
u_{\l}&=&f(\vx,\vw)\ \ \text{in}\ \ \Omega,\\
u_{\l}(\vx_0,\vw)&=&\pp[u_{\l}](\vx_0)+h(\vx_0,\vw)\ \ \text{for}\ \
\vw\cdot\vn<0\ \ \text{and}\ \ \vx_0\in\p\Omega,
\end{array}
\right.
\end{eqnarray}
and satisfies
\begin{eqnarray}
\nm{u_{\l}}_{L^2(\Omega\times\s^1)}\leq C\bigg(\frac{1}{\l}\nm{f}_{L^2(\Omega\times\s^1)}+\frac{1}{\l^{1/2}}\nm{h}_{L^2(\Gamma^-)}\bigg).
\end{eqnarray}
However, this estimate still blows up when $\l\rt0$, so we need to find a uniform estimate for $u_{\l}$.\\
\ \\
Step 5: Kernel Estimate.\\
Applying Lemma \ref{well-posedness lemma 3} to the solution of the
equation (\ref{penalty neutron}). Then for any
$\phi\in L^2(\Omega\times\s^1)$ satisfying $\vw\cdot\nx\phi\in
L^2(\Omega\times\s^1)$ and $\phi\in L^2(\Gamma)$, we have
\begin{eqnarray}\label{well-posedness temp 4}
\l\iint_{\Omega\times\s^1}u_{\l}\phi+\e\int_{\Gamma}u_{\l}\phi\ud{\gamma}
-\e\iint_{\Omega\times\s^1}(\vw\cdot\nx\phi)u_{\l}+\iint_{\Omega\times\s^1}(u_{\l}-\bar
u_{\l})\phi=\iint_{\Omega\times\s^1}f\phi.
\end{eqnarray}
Our goal is to choose a particular test function $\phi$. We first
construct an auxiliary function $\zeta$. Since $u_{\l}\in
L^{\infty}(\Omega\times\s^1)$, it naturally implies $\bar u_{\l}\in
L^{\infty}(\Omega)$ which further leads to $\bar u_{\l}\in
L^2(\Omega)$. We define $\zeta(\vx)$ on $\Omega$ satisfying
\begin{eqnarray}\label{test temp 1}
\left\{
\begin{array}{rcl}
\Delta \zeta&=&\bar u_{\l}\ \ \text{in}\ \
\Omega,\\\rule{0ex}{1.0em} \dfrac{\p\zeta}{\p\vn}&=&0\ \ \text{on}\ \ \p\Omega.
\end{array}
\right.
\end{eqnarray}
A direct integration over $\Omega\times\s^1$ in (\ref{penalty neutron.}) implies
\begin{eqnarray}
\int_{\Omega\times\s^1}u_{\l}(\vx,\vw)\ud{\vw}\ud{\vx}=0.
\end{eqnarray}
Hence, in the bounded domain $\Omega$, based on the standard elliptic
estimate, there exists $\zeta\in H^2(\Omega)$ such that
\begin{eqnarray}\label{test temp 3}
\nm{\zeta}_{H^2(\Omega)}\leq C(\Omega)\nm{\bar
u_{\l}}_{L^2(\Omega)}.
\end{eqnarray}
We plug the test function
\begin{eqnarray}\label{test temp 2}
\phi=-\vw\cdot\nx\zeta
\end{eqnarray}
into the weak formulation (\ref{well-posedness temp 4}) and estimate
each term there. Naturally, we have
\begin{eqnarray}\label{test temp 4}
\nm{\phi}_{L^2(\Omega)}\leq C\nm{\zeta}_{H^1(\Omega)}\leq
C(\Omega)\nm{\bar u_{\l}}_{L^2(\Omega)}.
\end{eqnarray}
Easily we can decompose
\begin{eqnarray}\label{test temp 5}
-\e\iint_{\Omega\times\s^1}(\vw\cdot\nx\phi)u_{\l}&=&-\e\iint_{\Omega\times\s^1}(\vw\cdot\nx\phi)\bar
u_{\l}-\e\iint_{\Omega\times\s^1}(\vw\cdot\nx\phi)(u_{\l}-\bar
u_{\l}).
\end{eqnarray}
We estimate the two term on the right-hand side separately. By
(\ref{test temp 1}) and (\ref{test temp 2}), we have
\begin{eqnarray}\label{wellposed temp 1}
-\e\iint_{\Omega\times\s^1}(\vw\cdot\nx\phi)\bar
u_{\l}&=&\e\iint_{\Omega\times\s^1}\bar
u_{\l}\bigg(w_1(w_1\p_{11}\zeta+w_2\p_{12}\zeta)+w_2(w_1\p_{12}\zeta+w_2\p_{22}\zeta)\bigg)\\
&=&\e\iint_{\Omega\times\s^1}\bar
u_{\l}\bigg(w_1^2\p_{11}\zeta+w_2^2\p_{22}\zeta\bigg)\nonumber\\
&=&\e\pi\int_{\Omega}\bar u_{\l}(\p_{11}\zeta+\p_{22}\zeta)\nonumber\\
&=&\e\pi\nm{\bar u_{\l}}_{L^2(\Omega)}^2\nonumber\\
&=&\half\e\nm{\bar u_{\l}}_{L^2(\Omega\times\s^1)}^2\nonumber.
\end{eqnarray}
In the second equality, above cross terms vanish due to the symmetry
of the integral over $\s^1$. On the other hand, for the second term
in (\ref{test temp 5}), H\"older's inequality and the elliptic
estimate imply
\begin{eqnarray}\label{wellposed temp 2}
-\e\iint_{\Omega\times\s^1}(\vw\cdot\nx\phi)(u_{\l}-\bar
u_{\l})&\leq&C(\Omega)\e\nm{u_{\l}-\bar u_{\l}}_{L^2(\Omega\times\s^1)}\nm{\zeta}_{H^2(\Omega)}\\
&\leq&C(\Omega)\e\nm{u_{\l}-\bar
u_{\l}}_{L^2(\Omega\times\s^1)}\nm{\bar
u_{\l}}_{L^2(\Omega\times\s^1)}\nonumber.
\end{eqnarray}
We may decompose $\vw=(\vw\cdot\vn)\vn+\vw_{\perp}$ to obtain
\begin{eqnarray}
\e\int_{\Gamma}u_{\l}\phi\ud{\gamma}&=&\e\int_{\Gamma}u_{\l}(\vw\cdot\nx\zeta)\ud{\gamma}\\
&=&\e\int_{\Gamma}u_{\l}(\vn\cdot\nx\zeta)(\vw\cdot\vn)\ud{\gamma}+\e\int_{\Gamma}u_{\l}(\vw_{\perp}\cdot\nx\zeta)\ud{\gamma}\no\\
&=&\e\int_{\Gamma}u_{\l}(\vw_{\perp}\cdot\nx\zeta)\ud{\gamma}.\no
\end{eqnarray}
Based on (\ref{test temp 3}), (\ref{test temp 4}), the boundary
condition of the penalized neutron transport equation
(\ref{penalty neutron.}), the trace theorem,
H\"older's inequality and the elliptic estimate, we have
\begin{eqnarray}\label{wellposed temp 3}
\e\int_{\Gamma}u_{\l}\phi\ud{\gamma}&=&\e\int_{\Gamma}u_{\l}(\vw_{\perp}\cdot\nx\zeta)\ud{\gamma}\\
&=&\e\int_{\Gamma}\pp[u_{\l}](\vw_{\perp}\cdot\nx\zeta)\ud{\gamma}+\e\int_{\Gamma^+}(1-\pp)[u_{\l}](\vw_{\perp}\cdot\nx\zeta)\ud{\gamma}
+\e\int_{\Gamma^-}h(\vw_{\perp}\cdot\nx\zeta)\ud{\gamma}\no\\
&=&\e\int_{\Gamma^+}(1-\pp)[u_{\l}](\vw_{\perp}\cdot\nx\zeta)\ud{\gamma}
+\e\int_{\Gamma^-}h(\vw_{\perp}\cdot\nx\zeta)\ud{\gamma}\no\\
&\leq&\e\nm{\bar u_{\l}}_{L^2(\Omega\times\s^1)}\bigg(\nm{(1-\pp)[u_{j,\l}]}_{L^2(\Gamma^+)}+\nm{h}_{L^2(\Gamma^-)}\bigg).\no
\end{eqnarray}
Also, we obtain
\begin{eqnarray}\label{wellposed temp 4}
\l\iint_{\Omega\times\s^1}u_{\l}\phi&=&\l\iint_{\Omega\times\s^1}\bar
u_{\l}\phi+\l\iint_{\Omega\times\s^1}(u_{\l}-\bar u_{\l})\phi=\l\iint_{\Omega\times\s^1}(u_{\l}-\bar u_{\l})\phi\\
&\leq&C(\Omega)\l\nm{\bar
u_{\l}}_{L^2(\Omega\times\s^1)}\nm{u_{\l}-\bar
u_{\l}}_{L^2(\Omega\times\s^1)}\nonumber,
\end{eqnarray}
\begin{eqnarray}\label{wellposed temp 5}
\iint_{\Omega\times\s^1}(u_{\l}-\bar u_{\l})\phi\leq
C(\Omega)\nm{\bar u_{\l}}_{L^2(\Omega\times\s^1)}\nm{u_{\l}-\bar
u_{\l}}_{L^2(\Omega\times\s^1)},
\end{eqnarray}
\begin{eqnarray}\label{wellposed temp 6}
\iint_{\Omega\times\s^1}f\phi\leq C(\Omega)\nm{\bar
u_{\l}}_{L^2(\Omega\times\s^1)}\nm{f}_{L^2(\Omega\times\s^1)}.
\end{eqnarray}
Collecting terms in (\ref{wellposed temp 1}), (\ref{wellposed temp
2}), (\ref{wellposed temp 3}), (\ref{wellposed temp 4}),
(\ref{wellposed temp 5}) and (\ref{wellposed temp 6}), we obtain
\begin{eqnarray}
\e\nm{\bar u_{\l}}_{L^2(\Omega\times\s^1)}&\leq&
C(\Omega)\bigg((1+\e+\l)\nm{u_{\l}-\bar
u_{\l}}_{L^2(\Omega\times\s^1)}+\e\nm{u_{\l}}_{L^2(\Gamma^+)}+\nm{f}_{L^2(\Omega\times\s^1)}\\
&&+\e\nm{(1-\pp)[u_{j,\l}]}_{L^2(\Gamma^+)}+\e\tm{h}{\Gamma^-}\bigg)\nonumber,
\end{eqnarray}
When $0\leq\l<1$ and $0<\e<1$, we get the desired uniform estimate
with respect to $\lambda$ as
\begin{eqnarray}\label{well-posedness temp 8}
\e\nm{\bar u_{\l}}_{L^2(\Omega\times\s^1)}&\leq&
C(\Omega)\bigg(\nm{u_{\l}-\bar
u_{\l}}_{L^2(\Omega\times\s^1)}+\e\nm{u_{\l}}_{L^2(\Gamma^+)}+\nm{f}_{L^2(\Omega\times\s^1)}\\
&&+\e\nm{(1-\pp)[u_{j,\l}]}_{L^2(\Gamma^+)}+\e\tm{h}{\Gamma^-}\bigg)\nonumber,
\end{eqnarray}
\ \\
Step 6: Limit $\l\rt0$.\\
In the weak formulation (\ref{well-posedness temp 4}), we may take
the test function $\phi=u_{\l}$ to get the energy estimate
\begin{eqnarray}
\l\nm{u_{\l}}_{L^2(\Omega\times\s^1)}^2+\half\e\int_{\Gamma}\abs{u_{\l}}^2\ud{\gamma}+\nm{u_{\l}-\bar
u_{\l}}_{L^2(\Omega\times\s^1)}^2=\iint_{\Omega\times\s^1}fu_{\l}.
\end{eqnarray}
Similar to (\ref{wt 02}), we have
\begin{eqnarray}
&&\half\e\nm{(1-\pp)[u_{\l}]}^2_{L^2(\Gamma^+)}+\l\nm{u_{\l}}_{L^2(\Omega\times\s^1)}^2+\nm{u_{\l}-\bar
u_{\l}}_{L^2(\Omega\times\s^1)}^2\\
&\leq& \iint_{\Omega\times\s^1}fu_{\l}+\bigg(1+\frac{4}{\eta}\bigg)\nm{h}_{L^2(\Gamma^-)}^2+\e^2\eta\nm{\pp[u_{\l}]}^2_{L^2(\Gamma^-)}.\no
\end{eqnarray}
Also, based on Step 3, we know
\begin{eqnarray}\label{wt 14}
\nm{\pp[u_{\l}]}_{L^2(\Gamma^-)}^2&\leq&C\bigg(\nm{u_{\l}}_{L^2(\Omega\times\s^1)}^2+\nm{\vw\cdot\nx(u_{\l}^2)}_{L^1(\Omega\times\s^1)}\bigg)\\
&\leq&C\bigg(\nm{u_{\l}-\bar
u_{\l}}_{L^2(\Omega\times\s^1)}^2+\nm{\bar
u_{\l}}_{L^2(\Omega\times\s^1)}^2+\frac{1}{\e}\bigg(1+\frac{4}{\eta}\bigg)\nm{h}_{L^2(\Gamma^-)}^2+\frac{4}{\e}\iint_{\Omega\times\s^1}fu\bigg).\no
\end{eqnarray}
Hence, this naturally implies
\begin{eqnarray}\label{well-posedness temp 5}
&&\e\nm{(1-\pp)[u_{\l}]}^2_{L^2(\Gamma^+)}+\nm{u_{\l}-\bar
u_{\l}}_{L^2(\Omega\times\s^1)}^2\leq C\bigg(\e^2\eta\nm{\bar
u_{\l}}_{L^2(\Omega\times\s^1)}^2+\iint_{\Omega\times\s^1}fu_{\l}+\nm{h}_{L^2(\Gamma^-)}^2\bigg).
\end{eqnarray}
On the other hand, we can square on both sides of
(\ref{well-posedness temp 8}) to obtain
\begin{eqnarray}\label{well-posedness temp 6}
\e^2\nm{\bar u_{\l}}_{L^2(\Omega\times\s^1)}^2&\leq&
C(\Omega)\bigg(\nm{u_{\l}-\bar
u_{\l}}_{L^2(\Omega\times\s^1)}^2+\e^2\nm{u_{\l}}_{L^2(\Gamma^+)}^2+\nm{f}_{L^2(\Omega\times\s^1)}^2\\
&&+\e^2\nm{(1-\pp)[u_{j,\l}]}_{L^2(\Gamma^+)}^2+\e^2\tm{h}{\Gamma^-}^2\bigg).\nonumber
\end{eqnarray}
Taking $\eta$ sufficiently small, multiplying a sufficiently small constant on both sides of
(\ref{well-posedness temp 6}) and adding it to (\ref{well-posedness
temp 5}) to absorb $\nm{(1-\pp)[u_{j,\l}]}_{L^2(\Gamma^+)}^2$, $\nm{u_{\l}}_{L^2(\Omega\times\s^1)}^2$ and
$\nm{u_{\l}-\bar u_{\l}}_{L^2(\Omega\times\s^1)}^2$, we deduce
\begin{eqnarray}
&&\e\nm{(1-\pp)[u_{j,\l}]}_{L^2(\Gamma^+)}^2+\e^2\nm{\bar
u_{\l}}_{L^2(\Omega\times\s^1)}^2+\nm{u_{\l}-\bar
u_{\l}}_{L^2(\Omega\times\s^1)}^2\\&&\qquad\qquad\qquad\leq
C(\Omega)\bigg(\tm{f}{\Omega\times\s^1}^2+
\iint_{\Omega\times\s^1}fu_{\l}+\nm{h}_{L^2(\Gamma^-)}^2\bigg).\nonumber
\end{eqnarray}
Hence, we have
\begin{eqnarray}\label{well-posedness temp 7}
\e\nm{(1-\pp)[u_{j,\l}]}_{L^2(\Gamma^+)}^2+\e^2\nm{u_{\l}}_{L^2(\Omega\times\s^1)}^2\leq
C(\Omega)\bigg(\tm{f}{\Omega\times\s^1}^2+
\iint_{\Omega\times\s^1}fu_{\l}+\nm{h}_{L^2(\Gamma^-)}^2\bigg).
\end{eqnarray}
A simple application of Cauchy's inequality leads to
\begin{eqnarray}
\iint_{\Omega\times\s^1}fu_{\l}\leq\frac{1}{4C\e^2}\tm{f}{\Omega\times\s^1}^2+C\e^2\tm{u_{\l}}{\Omega\times\s^1}^2.
\end{eqnarray}
Taking $C$ sufficiently small, we can divide (\ref{well-posedness
temp 7}) by $\e^2$ to obtain
\begin{eqnarray}\label{well-posedness temp 21}
\frac{1}{\e}\nm{(1-\pp)[u_{j,\l}]}_{L^2(\Gamma^+)}^2+\nm{u_{\l}}_{L^2(\Omega\times\s^2)}^2\leq
C(\Omega)\bigg(
\frac{1}{\e^4}\nm{f}_{L^2(\Omega\times\s^2)}^2+\frac{1}{\e^2}\nm{h}_{L^2(\Gamma^-)}^2\bigg).
\end{eqnarray}
Since above estimate does not depend on $\l$, it gives a uniform
estimate for the penalized neutron transport equation
(\ref{penalty neutron}). Thus, we can extract a
weakly convergent subsequence $u_{\l}\rt u$ as $\l\rt0$. The weak
lower semi-continuity of norms $\nm{\cdot}_{L^2(\Omega\times\s^2)}$
and $\nm{\cdot}_{L^2(\Gamma^+)}$ implies $u$ also satisfies the
estimate (\ref{well-posedness temp 21}). Hence, in the weak
formulation (\ref{well-posedness temp 4}), we can take $\l\rt0$ to
deduce that $u$ satisfies equation (\ref{neutron}). Also $u_{\l}-u$
satisfies the equation
\begin{eqnarray}
\left\{
\begin{array}{rcl}
\epsilon\vec w\cdot\nabla_x(u_{\l}-u)+(u_{\l}-u)-(\bar u_{\l}-\bar
u)&=&-\l u_{\l}\ \ \text{in}\ \
\Omega\label{remainder},\\\rule{0ex}{1.0em} (u_{\l}-u)(\vec x_0,\vec
w)&=&0\ \ \text{for}\ \ \vec x_0\in\p\Omega\ \ and\ \vw\cdot\vec
n<0.
\end{array}
\right.
\end{eqnarray}
By a similar argument as above, we can achieve
\begin{eqnarray}
\nm{u_{\l}-u}_{L^2(\Omega\times\s^2)}^2\leq
C(\Omega)\bigg(\frac{\l}{\e^4}\nm{u_{\l}}_{L^2(\Omega\times\s^2)}^2\bigg).
\end{eqnarray}
When $\l\rt0$, the right-hand side approaches zero, which implies
the convergence is actually in the strong sense. The uniqueness
easily follows from the energy estimates.

\end{proof}

\subsection{$L^{\infty}$ Estimate - First Round}

In this subsection, we prove the $L^2$-$L^{\infty}$ estimate. We consider the characteristics that reflect several times on the boundary.
\begin{definition}(Stochastic Cycle)
For fixed point $(t,\vx,\vw)$ with $(\vx,\vw)\notin\Gamma^0$, let
$(t_0,\vx_0,\vw_0)=(0,\vx,\vw)$. For $\vw_{k+1}$ such that
$\vw_{k+1}\cdot\vn(\vx_{k+1})>0$, define the $(k+1)$-component of
the back-time cycle as
\begin{eqnarray}
(t_{k+1},\vx_{k+1},\vw_{k+1})=(t_k+t_b(\vx_k,\vw_k),\vx_b(\vx_k,\vw_k),\vw_{k+1})
\end{eqnarray}
where
\begin{eqnarray}
t_b(\vx,\vw)&=&\inf\{t>0:\vx-\e t\vw\notin\Omega\}\\
x_b(\vx,\vw)&=&\vx-\e t_b(\vx,\vw)\vw\notin\Omega
\end{eqnarray}
Set
\begin{eqnarray}
\xc(s;t,\vx,\vw)&=&\sum_{k}\chi_{t_{k+1}\leq s<t_k}\bigg(\vx_k-\e(t_k-s)\vw_k\bigg)\\
\wc(s;t,\vx,\vw)&=&\sum_{k}\chi_{t_{k+1}\leq s<t_k}\vw_k
\end{eqnarray}
Define $\mu_{k+1}=\{\vw\in \s^1:\vw\cdot\vn(\vx_{k+1})>0\}$, and
let the iterated integral for $k\geq2$ be defined as
\begin{eqnarray}
\int_{\prod_{k=1}^{k-1}\mu_j}\prod_{j=1}^{k-1}\ud{\sigma_j}=\int_{\mu_1}\ldots\bigg(\int_{\mu_{k-1}}\ud{\sigma_{k-1}}\bigg)\ldots\ud{\sigma_1}
\end{eqnarray}
where $\ud{\sigma_j}=(\vn(\vx_j)\cdot\vw)\ud{\vw}$ is a
probability measure.
\end{definition}
\begin{lemma}\label{well-posedness lemma 6}
For $T_0>0$ sufficiently large, there exists constants $C_1,C_2>0$
independent of $T_0$, such that for $k=C_1T_0^{5/4}$,
\begin{eqnarray}
\int_{\prod_{j=1}^{k-1}\mu_j}{\bf{1}}_{t_k(t,\vx,\vw,\vw_1,\ldots,\vw_{k-1})<T_0}\prod_{j=1}^{k-1}\ud{\sigma_j}\leq
\bigg(\frac{1}{2}\bigg)^{C_2T_0^{5/4}}
\end{eqnarray}
\end{lemma}
\begin{proof}
See \cite[Lemma 4.1]{Esposito.Guo.Kim.Marra2013}.
\end{proof}
\begin{theorem}\label{LI estimate.}
Assume $f(\vx,\vw)\in L^{\infty}(\Omega\times\s^1)$ and
$h(x_0,\vw)\in L^{\infty}(\Gamma^-)$. Then the solution $u(\vx,\vw)$ to the transport
equation (\ref{neutron}) satisfies
\begin{eqnarray}
\im{u}{\Omega\times\s^1}\leq C(\Omega)\bigg(\frac{1}{\e^{3}}\tm{f}{\Omega\times\s^1}+\frac{1}{\e^{2}}\tm{h}{\Gamma^-}+\im{f}{\Omega\times\s^1}+\im{h}{\Gamma^-}\bigg).\no
\end{eqnarray}
\end{theorem}
\begin{proof}
We divide the proof into several steps:\\
\ \\
Step 1: Mild formulation.\\
We rewrite the equation (\ref{neutron}) along the characteristics as
\begin{eqnarray}
u(\vx,\vw)&=&h(\vx-\e t_1\vw,\vw)\ue^{- t_1}+\pp[u](\vx-\e t_1\vw,\vw)\ue^{- t_1}\\
&&+\int_0^{t_1}f(\vx-\e(t_1-s_1)\vw,\vw)\ue^{- (t_1-s_1)}\ud{s_1}+\int_0^{t_1}\bar u(\vx-\e(t_1-s_1)\vw)\ue^{- (t_1-s_1)}\ud{s_1}.\no
\end{eqnarray}
Note that here $\pp[u]$ is an integral over $\mu_1$ at $\vx_1$, using stochastic cycle, we may rewrite it again along the characteristics to $\vx_2$. This process can continue to arbitrary $\vx_k$. Then we get
\begin{eqnarray}\label{wt 21}
\\
u(\vx,\vw)&=&\ue^{- t_1}H+\sum_{l=1}^{k-1}\int_{\prod_{j=1}^l}\ue^{- t_{l+1}}G\prod_{j=1}^l\ud{\sigma_j}
+\sum_{l=1}^{k-1}\int_{\prod_{j=1}^l}\ue^{- t_{l+1}}\pp[u](\vx_k,\vw_{k-1})\prod_{j=1}^l\ud{\sigma_j}\no\\
&=&I+II+III.\no
\end{eqnarray}
where
\begin{eqnarray}
H&=&h(\vx-\e t_1\vw,\vw)\\
&&+\int_0^{t_1}f(\vx-\e(t_1-s_1)\vw,\vw)\ue^{ s_1}\ud{s_1}+\int_0^{t_1}\bar u(\vx-\e(t_1-s_1)\vw)\ue^{ s_1}\ud{s_1},\no\\
G&=&h(\vx_l-\e t_{l+1}\vw_l,\vw_l)\\
&&+\int_0^{t_l}f(\vx_l-\e(t_{l+1}-s_{l+1})\vw_l,\vw_l)\ue^{ s_{l+1}}\ud{s_{l+1}}+\int_0^{t_l}\bar u(\vx_l-\e(t_{l+1}-s_{l+1})\vw_l)\ue^{ s_{l+1}}\ud{s_{l+1}}.\no
\end{eqnarray}
We need to estimate each term on the right-hand side of (\ref{wt 21}).\\
\ \\
Step 2: Estimate of mild formulation.\\
We first consider $III$. We may decompose it as
\begin{eqnarray}
III&=&\sum_{l=1}^{k-1}\int_{\prod_{j=1}^l}\pp[u](\vx_k,\vw_{k-1})\ue^{- t_{l+1}}\prod_{j=1}^l\ud{\sigma_j}\\
&=&\sum_{l=1}^{k-1}\int_{\prod_{j=1}^l}{\bf{1}}_{t_k\leq T_0}\pp[u](\vx_k,\vw_{k-1})\ue^{- t_{l+1}}\prod_{j=1}^l\ud{\sigma_j}\no\\
&&+\sum_{l=1}^{k-1}\int_{\prod_{j=1}^l}{\bf{1}}_{t_k\geq T_0}\pp[u](\vx_k,\vw_{k-1})\ue^{- t_{l+1}}\prod_{j=1}^l\ud{\sigma_j},\no\\
&=&III_1+III_2,\no
\end{eqnarray}
where $T_0>0$ is defined as in Lemma \ref{well-posedness lemma 6}. Then we take $k=C_1T_0^{5/4}$. By Lemma \ref{well-posedness lemma 6}, we deduce
\begin{eqnarray}
\abs{III_1}&\leq&C\bigg(\frac{1}{2}\bigg)^{C_2T_0^{5/4}}\im{u}{\Omega\times\s^1}.
\end{eqnarray}
Also, we may directly estimate
\begin{eqnarray}
\abs{III_2}&\leq&C\ue^{- T_0}\im{u}{\Omega\times\s^1}.
\end{eqnarray}
Then taking $T_0$ sufficiently large, we know
\begin{eqnarray}\label{wt 22}
\abs{III}\leq \delta\im{u}{\Omega\times\s^1},
\end{eqnarray}
for $\d>0$ small.

On the other hand, we may directly estimate the terms in $I$ and $II$ related to $h$ and $f$, which we denote as $I_1$ and $II_1$. For fixed $T$, it is easy to see
\begin{eqnarray}\label{wt 23}
\abs{I_1}+\abs{II_1}\leq \im{f}{\Omega\times\s^1}+\im{h}{\Gamma^-}.
\end{eqnarray}
\ \\
Step 3: Estimate of $\bar u$ term.\\
The most troubling terms are related to $\bar u$. Here, we use the trick as in \cite{AA006}. Collecting the results in (\ref{wt 22}) and (\ref{wt 23}), we obtain
\begin{eqnarray}
\abs{u}&\leq&A+\abs{\int_0^{t_1}\bar u(\vx-\e(t_1-s_1)\vw)\ue^{- (t_1-s_1)}\ud{s_1}}\\
&&+\abs{\sum_{l=1}^{k-1}\int_{\prod_{j=1}^l}\bigg(\int_0^{t_l}\bar u(\vx_l-\e(t_{l+1}-s_{l+1})\vw_l)\ue^{- (t_{l+1}-s_{l+1})}\ud{s_{l+1}}\bigg)\prod_{j=1}^l\ud{\sigma_j}},\no\\
&=&A+I_2+II_2,\no
\end{eqnarray}
where
\begin{eqnarray}
A=\im{f}{\Omega\times\s^1}+\im{h}{\Gamma^-}+\delta\im{u}{\Omega\times\s^1}.
\end{eqnarray}
By definition, we know
\begin{eqnarray}
\abs{I_2}&=&\abs{\int_0^{t_1}\bigg(\int_{\s^1}u(\vx-\e(t_1-s_1)\vw,\vw_{s_1})\ud{\vw_{s_1}}\bigg)\ue^{- (t_1-s_1)}\ud{s_1}},
\end{eqnarray}
where $\vw_{s_1}\in\s^1$ is a dummy variable.
Then we can utilize the mild formulation (\ref{wt 21}) to rewrite $u(\vx-\e(t_1-s_1)\vw,\vw_{s_1})$ along the characteristics. We denote the stochastic cycle as $(t_k',\vx_k',\vw_k')$ correspondingly and $(t_0',\vx_0',\vw_0')=(0,\vx-\e(t_1-s_1)\vw,\vw_{s_1})$. Then
\begin{eqnarray}
\abs{I_2}&\leq& \abs{\int_0^{t_1}\bigg(\int_{\s^1}A\ud{\vw_{s_1}}\bigg)\ue^{- (t_1-s_1)}\ud{s_1}}\\
&&+\abs{\int_0^{t_1}\bigg(\int_{\s^1}\int_0^{t_1'}\bar u(\vx'-\e(t_1'-s_1')\vw_{s_1})\ue^{- (t_1'-s_1')}\ud{s_1'}\ud{\vw_{s_1}}\bigg)\ue^{- (t_1-s_1)}\ud{s_1}}\no\\
&&+\Bigg\vert\int_0^{t_1}\bigg(\int_{\s^1}\sum_{l'=1}^{k-1}\int_{\prod_{j'=1}^{l'}}\bigg(\int_0^{t_{l'}'}\bar u(\vx_{l'}-\e(t_{l'+1}'-s_{l'+1}')\vw_{l'})\ue^{- (t_{l'+1}'-s_{l'+1}')}\ud{s_{l'+1}'}\bigg)\prod_{j'=1}^{l'}
\ud{\sigma_{j'}}\ud{\vw_{s_1}}\bigg)\no\\
&&\ue^{- (t_1-s_1)}\ud{s_1}\Bigg\vert,\no\\
&=&\abs{I_{2,1}}+\abs{I_{2,2}}+\abs{I_{2,3}}.\no
\end{eqnarray}
It is obvious that
\begin{eqnarray}
\abs{I_{2,1}}&=&\abs{\int_0^{t_1}\bigg(\int_{\s^1}A\ud{\vw_{s_1}}\bigg)\ue^{- (t_1-s_1)}\ud{s_1}}\leq A\\
&\leq& \im{f}{\Omega\times\s^1}+\im{h}{\Gamma^-}+\delta\im{u}{\Omega\times\s^1}.\no
\end{eqnarray}
Then by definition, we know
\begin{eqnarray}
\abs{I_{2,2}}&=&\abs{\int_0^{t_1}\bigg(\int_{\s^1}\int_0^{t_1'}\bar u(\vx'-\e(t_1'-s_1')\vw_{s_1})\ue^{- (t_1'-s_1')}\ud{s_1'}\ud{\vw_{s_1}}\bigg)\ue^{- (t_1-s_1)}\ud{s_1}}.
\end{eqnarray}
We may decompose this integral
\begin{eqnarray}
\int_0^{t_1}\int_{\s^1}\int_0^{t_1'}&=&\int_0^{t_1}\int_{\s^1}\int_{t_1'-s_1'\leq\d}
+\int_0^{t_1}\int_{\s^1}\int_{t_1'-s_1'\geq\d}=I_{2,2,1}+I_{2,2,2}.
\end{eqnarray}
For $I_{2,2,1}$, since the integral is defined in the small domain $[t_1'-\d, t_1']$, it is easy to see
\begin{eqnarray}
\abs{I_{2,2,1}}\leq \delta\im{u}{\Omega\times\s^1}.
\end{eqnarray}
For $I_{2,2,2}$, applying H\"{o}lder's inequality, we get
\begin{eqnarray}
\abs{I_{2,2,2}}&\leq&\abs{\int_0^{t_1}\int_{\s^1}\int_{t_1'-s_1'\geq\d}\bar u(\vx'-\e(t_1'-s_1')\vw_{s_1})\ue^{- (t_1'-s_1')}\ue^{- (t_1-s_1)}\ud{s_1'}\ud{\vw_{s_1}}\ud{s_1}}\\
&\leq&\bigg(\int_0^{t_1}\int_{\s^1}\int_{t_1'-s_1'\geq\d}\ue^{- (t_1'-s_1')}\ue^{- (t_1-s_1)}\ud{s_1'}\ud{\vw_{s_1}}\ud{s_1}\bigg)^{1/2}\no\\
&&\bigg(\int_0^{t_1}\int_{\s^1}\int_{t_1'-s_1'\geq\d}{\bf{1}}_{\vx'-\e(t_1'-s_1')\vw_{s_1}\in\Omega}\abs{\bar u}^2(\vx'-\e(t_1'-s_1')\vw_{s_1})\ue^{- (t_1'-s_1')}\ue^{- (t_1-s_1)}\ud{s_1'}\ud{\vw_{s_1}}\ud{s_1}\bigg)^{1/2}\no\\
&\leq&\bigg(\int_0^{t_1}\int_{\s^1}\int_{t_1'-s_1'\geq\d}{\bf{1}}_{\vx'-\e(t_1'-s_1')\vw_{s_1}\in\Omega}\abs{\bar u}^2(\vx'-\e(t_1'-s_1')\vw_{s_1})\ue^{- (t_1'-s_1')}\ue^{- (t_1-s_1)}\ud{s_1'}\ud{\vw_{s_1}}\ud{s_1}\bigg)^{1/2}.\no
\end{eqnarray}
Since $\vw_{s_1}\in\s^1$, we can express it as $(\cos\phi,\sin\phi)$. Then we consider the substitution $(\phi,s'_1)\rt(y_1,y_2)$ as
\begin{eqnarray}
\vec y=\vx'-\e(t_1'-s_1')\vw_{s_1},
\end{eqnarray}
whose Jacobian is
\begin{eqnarray}
\abs{\frac{\p(y_1,y_2)}{\p(\phi,r')}}=\abs{\abs{\begin{array}{cc}
\e(t_1'-s_1')\sin\phi&\cos\phi\\
-\e(t_1'-s_1')\cos\phi&\sin\phi
\end{array}}}=\e^2(t_1'-s_1')\geq\e^2\d.
\end{eqnarray}
Therefore, we know
\begin{eqnarray}
\abs{I_{2,2,2}}&\leq&\frac{1}{\e\d^{\frac{1}{2}}}\nm{\bar u}_{L^2(\Omega\times\s^1)}.
\end{eqnarray}
Therefore, we have shown
\begin{eqnarray}
\abs{I_{2,2}}\leq \delta\im{u}{\Omega\times\s^1}+\frac{1}{\d^{\frac{1}{2}}\e}\nm{\bar u}_{L^2(\Omega\times\s^1)}.
\end{eqnarray}
After a similar but tedious computation, we can show
\begin{eqnarray}
\abs{I_{2,3}}\leq \delta\im{u}{\Omega\times\s^1}+\frac{1}{\d^{\frac{1}{2}}\e}\nm{\bar u}_{L^2(\Omega\times\s^1)}.
\end{eqnarray}
Hence, we have proved
\begin{eqnarray}
\abs{I_{2}}\leq \delta\im{u}{\Omega\times\s^1}+\frac{1}{\d^{\frac{1}{2}}\e}\nm{\bar u}_{L^2(\Omega\times\s^1)}+\im{f}{\Omega\times\s^1}+\im{h}{\Gamma^-}.
\end{eqnarray}
In a similar fashion, we can show
\begin{eqnarray}
\abs{II_{2}}\leq \delta\im{u}{\Omega\times\s^1}+\frac{1}{\d^{\frac{1}{2}}\e}\nm{\bar u}_{L^2(\Omega\times\s^1)}+\im{f}{\Omega\times\s^1}+\im{h}{\Gamma^-}.
\end{eqnarray}
\ \\
Step 4: Synthesis.\\
Summarizing all above, we have shown
\begin{eqnarray}
\abs{u}&\leq& \delta\im{u}{\Omega\times\s^1}+\frac{1}{\d^{\frac{1}{2}}\e}\nm{\bar u}_{L^2(\Omega\times\s^1)}+\im{f}{\Omega\times\s^1}+\im{h}{\Gamma^-}\\
&\leq&\delta\im{u}{\Omega\times\s^1}+\frac{1}{\d^{\frac{1}{2}}\e}\nm{u}_{L^2(\Omega\times\s^1)}+\im{f}{\Omega\times\s^1}+\im{h}{\Gamma^-}.\no
\end{eqnarray}
Since $(\vx,\vw)$ are arbitrary and $\d$ is small, taking supremum on both sides and applying Lemma \ref{LT estimate}, we have
\begin{eqnarray}
\im{u}{\Omega\times\s^1}&\leq&C(\Omega)\bigg(\frac{1}{\e}\nm{u}_{L^2(\Omega\times\s^1)}+\im{f}{\Omega\times\s^1}+\im{h}{\Gamma^-}\bigg)\\
&\leq&C(\Omega)\bigg(\frac{1}{\e^{3}}\tm{f}{\Omega\times\s^1}+\frac{1}{\e^{2}}\tm{h}{\Gamma^-}+\im{f}{\Omega\times\s^1}+\im{h}{\Gamma^-}\bigg).\no
\end{eqnarray}
This is the desired result.
\end{proof}

\subsection{$L^{2m}$ Estimate}

In this subsection, we try to improve previous estimates. In the following, we assume $m>2$ is an integer and let $o(1)$ denote a sufficiently small constant.
\begin{theorem}\label{LN estimate}
Assume $f(\vx,\vw)\in L^{\infty}(\Omega\times\s^1)$ and
$h(x_0,\vw)\in L^{\infty}(\Gamma^-)$. Then
$u(\vx,\vw)$ satisfies
\begin{eqnarray}
&&\frac{1}{\e^{\frac{1}{2}}}\nm{(1-\pp)[u]}_{L^2(\Gamma^+)}+\nm{
\bar u}_{L^{2m}(\Omega\times\s^1)}+\frac{1}{\e}\nm{u-\bar
u}_{L^2(\Omega\times\s^1)}\\
&\leq&
C\bigg(o(1)\e^{\frac{1}{m}}\nm{u}_{L^{\infty}(\Gamma^+)}+\frac{1}{\e}\tm{f}{\Omega\times\s^1}+
\frac{1}{\e^2}\nm{f}_{L^{\frac{2m}{2m-1}}(\Omega\times\s^1)}+\frac{1}{\e}\nm{h}_{L^2(\Gamma^-)}+\nm{h}_{L^{m}(\Gamma^-)}\bigg).\nonumber
\end{eqnarray}
\end{theorem}
\begin{proof}
We divide the proof into several steps:\\
\ \\
Step 1: Kernel Estimate.\\
Applying Green's identity to the solution of the
equation (\ref{neutron}). Then for any
$\phi\in L^2(\Omega\times\s^1)$ satisfying $\vw\cdot\nx\phi\in
L^2(\Omega\times\s^1)$ and $\phi\in L^2(\Gamma)$, we have
\begin{eqnarray}\label{well-posedness temp 4.}
\e\int_{\Gamma}u\phi\ud{\gamma}
-\e\iint_{\Omega\times\s^1}(\vw\cdot\nx\phi)u+\iint_{\Omega\times\s^1}(u-\bar
u)\phi=\iint_{\Omega\times\s^1}f\phi.
\end{eqnarray}
Our goal is to choose a particular test function $\phi$. We first
construct an auxiliary function $\zeta$. Naturally $u\in
L^{2m}(\Omega\times\s^1)$ implies $\bar u\in
L^{2m}(\Omega)$ which further leads to $(\bar u)^{2m-1}\in
L^{\frac{2m}{2m-1}}(\Omega)$. We define $\zeta(\vx)$ on $\Omega$ satisfying
\begin{eqnarray}\label{test temp 1.}
\left\{
\begin{array}{rcl}
\Delta \zeta&=&(\bar u)^{2m-1}-\displaystyle\frac{1}{\abs{\Omega}}\int_{\Omega}(\bar u)^{2m-1}\ud{\vx}\ \ \text{in}\ \
\Omega,\\\rule{0ex}{1.0em} \dfrac{\p\zeta}{\p\vn}&=&0\ \ \text{on}\ \ \p\Omega.
\end{array}
\right.
\end{eqnarray}
In the bounded domain $\Omega$, based on the standard elliptic
estimate, we have
\begin{eqnarray}\label{test temp 3.}
\nm{\zeta}_{W^{2,\frac{2m}{2m-1}}(\Omega)}\leq C\nm{(\bar
u)^{2m-1}}_{L^{\frac{2m}{2m-1}}(\Omega)}= C\nm{\bar
u}_{L^{2m}(\Omega)}^{2m-1}.
\end{eqnarray}
We plug the test function
\begin{eqnarray}\label{test temp 2.}
\phi=-\vw\cdot\nx\zeta
\end{eqnarray}
into the weak formulation (\ref{well-posedness temp 4.}) and estimate
each term there. By Sobolev embedding theorem, we have
\begin{eqnarray}
\nm{\phi}_{L^2(\Omega)}&\leq& C\nm{\zeta}_{H^1(\Omega)}\leq C\nm{\zeta}_{W^{2,\frac{2m}{2m-1}}(\Omega)}\leq
C\nm{\bar
u}_{L^{2m}(\Omega)}^{2m-1},\label{test temp 6.}\\
\nm{\phi}_{L^{\frac{2m}{2m-1}}(\Omega)}&\leq&C\nm{\zeta}_{W^{1,\frac{2m}{2m-1}}(\Omega)}\leq
C\nm{\bar
u}_{L^{2m}(\Omega)}^{2m-1}.\label{test temp 4.}
\end{eqnarray}
Easily we can decompose
\begin{eqnarray}\label{test temp 5.}
-\e\iint_{\Omega\times\s^1}(\vw\cdot\nx\phi)u_{\l}&=&-\e\iint_{\Omega\times\s^1}(\vw\cdot\nx\phi)\bar
u_{\l}-\e\iint_{\Omega\times\s^1}(\vw\cdot\nx\phi)(u_{\l}-\bar
u_{\l}).
\end{eqnarray}
We estimate the two term on the right-hand side of (\ref{test temp 5.}) separately. By
(\ref{test temp 1.}) and (\ref{test temp 2.}), we have
\begin{eqnarray}\label{wellposed temp 1.}
-\e\iint_{\Omega\times\s^1}(\vw\cdot\nx\phi)\bar
u&=&\e\iint_{\Omega\times\s^1}\bar
u\bigg(w_1(w_1\p_{11}\zeta+w_2\p_{12}\zeta)+w_2(w_1\p_{12}\zeta+w_2\p_{22}\zeta)\bigg)\\
&=&\e\iint_{\Omega\times\s^1}\bar
u\bigg(w_1^2\p_{11}\zeta+w_2^2\p_{22}\zeta\bigg)\nonumber\\
&=&\e\pi\int_{\Omega}\bar u(\p_{11}\zeta+\p_{22}\zeta)\nonumber\\
&=&\e\pi\nm{\bar u}_{L^{2m}(\Omega)}^{2m}\nonumber.
\end{eqnarray}
In the second equality, above cross terms vanish due to the symmetry
of the integral over $\s^1$. On the other hand, for the second term
in (\ref{test temp 5.}), H\"older's inequality and the elliptic
estimate imply
\begin{eqnarray}\label{wellposed temp 2.}
-\e\iint_{\Omega\times\s^1}(\vw\cdot\nx\phi)(u-\bar
u)&\leq&C\e\nm{u-\bar u}_{L^{2m}(\Omega\times\s^1)}\nm{\nx\phi}_{L^{\frac{2m}{2m-1}}(\Omega)}\\
&\leq&C\e\nm{u-\bar u}_{L^{2m}(\Omega\times\s^1)}\nm{\zeta}_{W^{2,\frac{2m}{2m-1}}(\Omega)}\no\\
&\leq&C\e\nm{u-\bar
u}_{L^{2m}(\Omega\times\s^1)}\nm{\bar
u}_{L^{2m}(\Omega)}^{2m-1}\nonumber.
\end{eqnarray}
Based on (\ref{test temp 3.}), (\ref{test temp 6.}), (\ref{test temp 4.}), Sobolev embedding theorem and the trace theorem, we have
\begin{eqnarray}
\\
\nm{\nx\zeta}_{L^{\frac{m}{m-1}}(\Gamma)}\leq C\nm{\nx\zeta}_{W^{\frac{1}{2m},\frac{2m}{2m-1}}(\Gamma)}\leq C\nm{\nx\zeta}_{W^{1,\frac{2m}{2m-1}}(\Omega)}\leq C\nm{\zeta}_{W^{2,\frac{2m}{2m-1}}(\Omega)}\leq
C\nm{\bar
u}_{L^{2m}(\Omega)}^{2m-1}.\no
\end{eqnarray}
We may also decompose $\vw=(\vw\cdot\vn)\vn+\vw_{\perp}$ to obtain
\begin{eqnarray}
\e\int_{\Gamma}u\phi\ud{\gamma}&=&\e\int_{\Gamma}u(\vw\cdot\nx\zeta)\ud{\gamma}\\
&=&\e\int_{\Gamma}u(\vn\cdot\nx\zeta)(\vw\cdot\vn)\ud{\gamma}+\e\int_{\Gamma}u(\vw_{\perp}\cdot\nx\zeta)\ud{\gamma}\no\\
&=&\e\int_{\Gamma}u(\vw_{\perp}\cdot\nx\zeta)\ud{\gamma}.\no
\end{eqnarray}
Based on (\ref{test temp 3.}), (\ref{test temp 4.}) and
H\"older's inequality, we have
\begin{eqnarray}
\e\int_{\Gamma}u\phi\ud{\gamma}&=&\e\int_{\Gamma}u(\vw_{\perp}\cdot\nx\zeta)\ud{\gamma}\\
&=&\e\int_{\Gamma}\pp[u](\vw_{\perp}\cdot\nx\zeta)\ud{\gamma}+\e\int_{\Gamma^+}(1-\pp)[u](\vw_{\perp}\cdot\nx\zeta)\ud{\gamma}
+\e\int_{\Gamma^-}h(\vw_{\perp}\cdot\nx\zeta)\ud{\gamma}\no\\
&=&\e\int_{\Gamma^+}(1-\pp)[u](\vw_{\perp}\cdot\nx\zeta)\ud{\gamma}
+\e\int_{\Gamma^-}h(\vw_{\perp}\cdot\nx\zeta)\ud{\gamma}\no\\
&\leq&C\e\nm{\nx\zeta}_{L^{\frac{m}{m-1}}(\Gamma)}\bigg(\nm{(1-\pp)[u]}_{L^{m}(\Gamma^+)}+\nm{h}_{L^{m}(\Gamma^-)}\bigg)\no\\
&\leq&C\e\nm{\bar u}_{L^{2m}(\Omega\times\s^1)}^{2m-1}\bigg(\nm{(1-\pp)[u]}_{L^{m}(\Gamma^+)}+\nm{h}_{L^{m}(\Gamma^-)}\bigg).\no
\end{eqnarray}
Hence, we know
\begin{eqnarray}\label{wellposed temp 3.}
\e\int_{\Gamma}u\phi\ud{\gamma}&\leq&C\e\nm{\bar u}_{L^{2m}(\Omega\times\s^1)}^{2m-1}\bigg(\nm{(1-\pp)[u]}_{L^{m}(\Gamma^+)}
+\nm{h}_{L^{m}(\Gamma^-)}\bigg).
\end{eqnarray}
Also, we have
\begin{eqnarray}\label{wellposed temp 5.}
\iint_{\Omega\times\s^1}(u-\bar u)\phi\leq
C\nm{\phi}_{L^2(\Omega\times\s^1)}\nm{u-\bar
u}_{L^2(\Omega\times\s^1)}\leq
C\nm{\bar
u}_{L^{2m}(\Omega)}^{2m-1}\nm{u-\bar
u}_{L^2(\Omega\times\s^1)},
\end{eqnarray}
\begin{eqnarray}\label{wellposed temp 6.}
\iint_{\Omega\times\s^1}f\phi\leq C\nm{\phi}_{L^2(\Omega\times\s^1)}\nm{f}_{L^2(\Omega\times\s^1)}\leq C\nm{\bar
u}_{L^{2m}(\Omega)}^{2m-1}\nm{f}_{L^2(\Omega\times\s^1)}.
\end{eqnarray}
Collecting terms in (\ref{wellposed temp 1.}), (\ref{wellposed temp
2.}), (\ref{wellposed temp 3.}),
(\ref{wellposed temp 5.}) and (\ref{wellposed temp 6.}), we obtain
\begin{eqnarray}\label{improve temp 1.}
\e\nm{\bar u}_{L^{2m}(\Omega\times\s^1)}&\leq&
C\bigg(\e\nm{u-\bar
u}_{L^{2m}(\Omega\times\s^1)}+\nm{u-\bar
u}_{L^2(\Omega\times\s^1)}+\nm{f}_{L^2(\Omega\times\s^1)}\\
&&+\e\nm{(1-\pp)[u]}_{L^{m}(\Gamma^+)}+\e\nm{h}_{L^{m}(\Gamma^-)}\bigg)\nonumber,
\end{eqnarray}
\ \\
Step 2: Energy Estimate.\\
In the weak formulation (\ref{well-posedness temp 4.}), we may take
the test function $\phi=u$ to get the energy estimate
\begin{eqnarray}\label{well-posedness temp 9.}
\half\e\int_{\Gamma}\abs{u}^2\ud{\gamma}+\nm{u-\bar
u}_{L^2(\Omega\times\s^1)}^2=\iint_{\Omega\times\s^1}fu.
\end{eqnarray}
Hence, by (\ref{wt 14}), this naturally implies
\begin{eqnarray}\label{well-posedness temp 5.}
\e\nm{(1-\pp)[u]}_{L^2(\Gamma^+)}^2+\nm{u-\bar
u}_{L^2(\Omega\times\s^1)}^2\leq o(1)\e^2\nm{\bar u}_{L^2(\Omega\times\s^1)}^2+\iint_{\Omega\times\s^1}fu+\nm{h}_{L^2(\Gamma^-)}^2.
\end{eqnarray}
On the other hand, we can square on both sides of
(\ref{improve temp 1.}) to obtain
\begin{eqnarray}\label{well-posedness temp 6.}
\e^2\nm{\bar u}_{L^{2m}(\Omega\times\s^1)}^2&\leq&
C\bigg(\e^2\nm{u-\bar
u}_{L^{2m}(\Omega\times\s^1)}^2+\nm{u-\bar
u}_{L^2(\Omega\times\s^1)}^2+\nm{f}_{L^2(\Omega\times\s^1)}^2\\
&&+\e^2\nm{(1-\pp)[u]}_{L^{m}(\Gamma^+)}+\e^2\nm{h}_{L^{m}(\Gamma^-)}^2\bigg)\nonumber,
\end{eqnarray}
Multiplying a sufficiently small constant on both sides of
(\ref{well-posedness temp 6.}) and adding it to (\ref{well-posedness
temp 5.}) to absorb
$\nm{u-\bar u}_{L^2(\Omega\times\s^1)}^2$ and $\e^2\nm{\bar u}_{L^2(\Omega\times\s^1)}^2$, we deduce
\begin{eqnarray}\label{wt 03.}
&&\e\nm{(1-\pp)[u]}_{L^2(\Gamma^+)}^2+\e^2\nm{\bar
u}_{L^{2m}(\Omega\times\s^1)}^2+\nm{u-\bar
u}_{L^2(\Omega\times\s^1)}^2\\
&\leq&
C\bigg(\e^2\nm{u-\bar
u}_{L^{2m}(\Omega\times\s^1)}^2+\e^2\nm{(1-\pp)[u]}_{L^{m}(\Gamma^+)}\no\\
&&+\tm{f}{\Omega\times\s^1}^2+
\iint_{\Omega\times\s^1}fu+\nm{h}_{L^2(\Gamma^-)}^2+\e^2\nm{h}_{L^{m}(\Gamma^-)}^2\bigg).\nonumber
\end{eqnarray}
By interpolation estimate and Young's inequality, we have
\begin{eqnarray}
\nm{(1-\pp)[u]}_{L^{m}(\Gamma^+)}&\leq&\nm{(1-\pp)[u]}_{L^2(\Gamma^+)}^{\frac{2}{m}}\nm{(1-\pp)[u]}_{L^{\infty}(\Gamma^+)}^{\frac{m-2}{m}}\\
&=&\bigg(\frac{1}{\e^{\frac{m-2}{m^2}}}\nm{(1-\pp)[u]}_{L^2(\Gamma^+)}^{\frac{2}{m}}\bigg)
\bigg(\e^{\frac{m-2}{m^2}}\nm{(1-\pp)[u]}_{L^{\infty}(\Gamma^+)}^{\frac{m-2}{m}}\bigg)\no\\
&\leq&C\bigg(\frac{1}{\e^{\frac{m-2}{m^2}}}\nm{(1-\pp)[u]}_{L^2(\Gamma^+)}^{\frac{2}{m}}\bigg)^{\frac{m}{2}}+o(1)
\bigg(\e^{\frac{m-2}{m^2}}\nm{(1-\pp)[u]}_{L^{\infty}(\Gamma^+)}^{\frac{m-2}{m}}\bigg)^{\frac{m}{m-2}}\no\\
&\leq&\frac{C}{\e^{\frac{m-2}{2m}}}\nm{(1-\pp)[u]}_{L^2(\Gamma^+)}+o(1)\e^{\frac{1}{m}}\nm{(1-\pp)[u]}_{L^{\infty}(\Gamma^+)}\no\\
&\leq&\frac{C}{\e^{\frac{m-2}{2m}}}\nm{(1-\pp)[u]}_{L^2(\Gamma^+)}+o(1)\e^{\frac{1}{m}}\nm{u}_{L^{\infty}(\Omega\times\s^1)}.\no
\end{eqnarray}
Similarly, we have
\begin{eqnarray}
\nm{u-\bar u}_{L^{2m}(\Omega\times\s^1)}&\leq&\nm{u-\bar u}_{L^2(\Omega\times\s^1)}^{\frac{1}{m}}\nm{u-\bar u}_{L^{\infty}(\Omega\times\s^1)}^{\frac{m-1}{m}}\\
&=&\bigg(\frac{1}{\e^{\frac{m-1}{m^2}}}\nm{u-\bar u}_{L^2(\Omega\times\s^1)}^{\frac{1}{m}}\bigg)\bigg(\e^{\frac{m-1}{m^2}}\nm{u-\bar u}_{L^{\infty}(\Omega\times\s^1)}^{\frac{m-1}{m}}\bigg)\no\\
&\leq&C\bigg(\frac{1}{\e^{\frac{m-1}{m^2}}}\nm{u-\bar u}_{L^2(\Omega\times\s^1)}^{\frac{1}{m}}\bigg)^{m}+o(1)\bigg(\e^{\frac{m-1}{m^2}}\nm{u-\bar u}_{L^{\infty}(\Omega\times\s^1)}^{\frac{m-1}{m}}\bigg)^{\frac{m}{m-1}}\no\\
&\leq&\frac{C}{\e^{\frac{m-1}{m}}}\nm{u-\bar u}_{L^2(\Omega\times\s^1)}+o(1)\e^{\frac{1}{m}}\nm{u-\bar u}_{L^{\infty}(\Omega\times\s^1)}.\no
\end{eqnarray}
We need this extra $\e^{\frac{1}{m}}$ for the convenience of $L^{\infty}$ estimate.
Then we know for sufficiently small $\e$,
\begin{eqnarray}
\e^2\nm{(1-\pp)[u]}_{L^{m}(\Gamma^+)}^2
&\leq&C\e^{2-\frac{m-2}{m}}\nm{(1-\pp)[u]}_{L^2(\Gamma^+)}^2+o(1)\e^{2+\frac{2}{m}}\nm{u}_{L^{\infty}(\Gamma^+)}^2\\
&\leq&o(1)\e\nm{(1-\pp)[u]}_{L^2(\Gamma^+)}^2+o(1)\e^{2+\frac{2}{m}}\nm{u}_{L^{\infty}(\Gamma^+)}^2.\no
\end{eqnarray}
Similarly, we have
\begin{eqnarray}
\e^2\nm{u-\bar
u}_{L^{2m}(\Omega\times\s^1)}^2&\leq&\e^{2-\frac{2m-2}{m}}\nm{u-\bar u}_{L^2(\Omega\times\s^1)}^2+o(1)\e^{2+\frac{2}{m}}\nm{u}_{L^{\infty}(\Omega\times\s^1)}^2\\
&\leq& o(1)\nm{u-\bar u}_{L^2(\Omega\times\s^1)}^2+o(1)\e^{2+\frac{2}{m}}\nm{u}_{L^{\infty}(\Omega\times\s^1)}^2.\no
\end{eqnarray}
In (\ref{wt 03}), we can absorb $\nm{u-\bar u}_{L^2(\Omega\times\s^1)}$ and $\e\nm{(1-\pp)[u]}_{L^2(\Gamma^+)}^2$ into left-hand side to obtain
\begin{eqnarray}\label{wt 04.}
&&\e\nm{(1-\pp)[u]}_{L^2(\Gamma^+)}^2+\e^2\nm{\bar
u}_{L^{2m}(\Omega\times\s^1)}^2+\nm{u-\bar
u}_{L^2(\Omega\times\s^1)}^2\\
&\leq&
C\bigg(o(1)\e^{2+\frac{2}{m}}\nm{u}_{L^{\infty}(\Omega\times\s^1)}^2+\tm{f}{\Omega\times\s^1}^2+
\iint_{\Omega\times\s^1}fu+\nm{h}_{L^2(\Gamma^-)}^2+\e^2\nm{h}_{L^{m}(\Gamma^-)}^2\bigg).\no
\end{eqnarray}
We can decompose
\begin{eqnarray}
\iint_{\Omega\times\s^1}fu=\iint_{\Omega\times\s^1}f\bar u+\iint_{\Omega\times\s^1}f(u-\bar u).
\end{eqnarray}
H\"older's inequality and Cauchy's inequality imply
\begin{eqnarray}
\iint_{\Omega\times\s^1}f\bar u\leq\nm{f}_{L^{\frac{2m}{2m-1}}(\Omega\times\s^1)}\nm{\bar u}_{L^{2m}(\Omega\times\s^1)}
\leq\frac{C}{\e^{2}}\nm{f}_{L^{\frac{2m}{2m-1}}(\Omega\times\s^1)}^2+o(1)\e^2\nm{\bar u}_{L^{2m}(\Omega\times\s^1)}^2,
\end{eqnarray}
and
\begin{eqnarray}
\iint_{\Omega\times\s^1}f(u-\bar u)\leq C\nm{f}_{L^{2}(\Omega\times\s^1)}^2+o(1)\nm{u-\bar u}_{L^2(\Omega\times\s^1)}^2.
\end{eqnarray}
Hence, absorbing $\e^2\nm{\bar u}_{L^{2m}(\Omega\times\s^1)}^2$ and $\nm{u-\bar u}_{L^2(\Omega\times\s^1)}^2$ into left-hand side of (\ref{wt 04.}), we get
\begin{eqnarray}\label{wt 06.}
&&\e\nm{(1-\pp)[u]}_{L^2(\Gamma^+)}^2+\e^2\nm{\bar
u}_{L^{2m}(\Omega\times\s^1)}^2+\nm{u-\bar
u}_{L^2(\Omega\times\s^1)}^2\\
&\leq&
C\bigg(o(1)\e^{2+\frac{2}{m}}\nm{u}_{L^{\infty}(\Gamma^+)}^2+\tm{f}{\Omega\times\s^1}^2+
\frac{1}{\e^2}\nm{f}_{L^{\frac{2m}{2m-1}}(\Omega\times\s^1)}^2+\nm{h}_{L^2(\Gamma^-)}^2+\e^2\nm{h}_{L^{m}(\Gamma^-)}^2\bigg),\nonumber
\end{eqnarray}
which implies
\begin{eqnarray}\label{wt 07.}
&&\frac{1}{\e^{\frac{1}{2}}}\nm{(1-\pp)[u]}_{L^2(\Gamma^+)}+\nm{
\bar u}_{L^{2m}(\Omega\times\s^1)}+\frac{1}{\e}\nm{u-\bar
u}_{L^2(\Omega\times\s^1)}\\
&\leq&
C\bigg(o(1)\e^{\frac{1}{m}}\nm{u}_{L^{\infty}(\Gamma^+)}+\frac{1}{\e}\tm{f}{\Omega\times\s^1}+
\frac{1}{\e^2}\nm{f}_{L^{\frac{2m}{2m-1}}(\Omega\times\s^1)}+\frac{1}{\e}\nm{h}_{L^2(\Gamma^-)}+\nm{h}_{L^{m}(\Gamma^-)}\bigg).\nonumber
\end{eqnarray}

\end{proof}

\subsection{$L^{\infty}$ Estimate - Second Round}

\begin{theorem}\label{LI estimate}
Assume $f(\vx,\vw)\in L^{\infty}(\Omega\times\s^1)$ and
$h(x_0,\vw)\in L^{\infty}(\Gamma^-)$. Then for the steady neutron
transport equation (\ref{neutron}), there exists a unique solution
$u(\vx,\vw)\in L^{\infty}(\Omega\times\s^1)$ satisfying
\begin{eqnarray}
\im{u}{\Omega\times\s^1}
&\leq& C\bigg(\frac{1}{\e^{1+\frac{1}{m}}}\tm{f}{\Omega\times\s^1}+
\frac{1}{\e^{2+\frac{1}{m}}}\nm{f}_{L^{\frac{2m}{2m-1}}(\Omega\times\s^1)}+\im{f}{\Omega\times\s^1}\\
&&+\frac{1}{\e^{1+\frac{1}{m}}}\nm{h}_{L^2(\Gamma^-)}+\frac{1}{\e^{\frac{1}{m}}}\nm{h}_{L^{m}(\Gamma^-)}+\im{h}{\Gamma^-}\bigg).\no
\end{eqnarray}
\end{theorem}
\begin{proof}
Following the argument in the proof of Theorem \ref{LI estimate.}, by double Duhamel's principle along the characteristics,
the key step is
\begin{eqnarray}
\abs{I_{2,2,2}}&\leq&\abs{\int_0^{t_1}\int_{\s^1}\int_{t_1'-s_1'\geq\d}\bar u(\vx'-\e(t_1'-s_1')\vw_{s_1})\ue^{- (t_1'-s_1')}\ue^{- (t_1-s_1)}\ud{s_1'}\ud{\vw_{s_1}}\ud{s_1}}\\
&\leq&\bigg(\int_0^{t_1}\int_{\s^1}\int_{t_1'-s_1'\geq\d}\ue^{- (t_1'-s_1')}\ue^{- (t_1-s_1)}\ud{s_1'}\ud{\vw_{s_1}}\ud{s_1}\bigg)^{\frac{2m-1}{2m}}\no\\
&&\bigg(\int_0^{t_1}\int_{\s^1}\int_{t_1'-s_1'\geq\d}{\bf{1}}_{\vx'-\e(t_1'-s_1')\vw_{s_1}\in\Omega}\abs{\bar u}^{2m}(\vx'-\e(t_1'-s_1')\vw_{s_1})\ue^{- (t_1'-s_1')}\ue^{- (t_1-s_1)}\ud{s_1'}\ud{\vw_{s_1}}\ud{s_1}\bigg)^{\frac{1}{2m}}\no\\
&\leq&\bigg(\int_0^{t_1}\int_{\s^1}\int_{t_1'-s_1'\geq\d}{\bf{1}}_{\vx'-\e(t_1'-s_1')\vw_{s_1}\in\Omega}\abs{\bar u}^{2m}(\vx'-\e(t_1'-s_1')\vw_{s_1})\ue^{- (t_1'-s_1')}\ue^{- (t_1-s_1)}\ud{s_1'}\ud{\vw_{s_1}}\ud{s_1}\bigg)^{\frac{1}{2m}}.\no
\end{eqnarray}
Then using the same substitution $(\phi,s'_1)\rt(y_1,y_2)$ as
\begin{eqnarray}
\vec y=\vx'-\e(t_1'-s_1')\vw_{s_1},
\end{eqnarray}
whose Jacobian is larger than $\e^2\d$, we know
\begin{eqnarray}
\abs{I_{2,2,2}}&\leq&\frac{1}{\e^{\frac{1}{m}}\d^{\frac{1}{2m}}}\nm{\bar u}_{L^{2m}(\Omega\times\s^1)}.
\end{eqnarray}
Therefore, we can show
\begin{eqnarray}
\im{u}{\Omega\times\s^1}\leq C\bigg(
\frac{1}{\e^{\frac{1}{m}}}\nm{\bar u}_{L^{2m}(\Omega\times\s^1)}+\im{f}{\Omega\times\s^1}+\im{g}{\Gamma^-}\bigg).
\end{eqnarray}
Considering Theorem \ref{LN estimate}, we obtain
\begin{eqnarray}
\im{u}{\Omega\times\s^1}&\leq& C\bigg(\frac{1}{\e^{1+\frac{1}{m}}}\tm{f}{\Omega\times\s^1}+
\frac{1}{\e^{2+\frac{1}{m}}}\nm{f}_{L^{\frac{2m}{2m-1}}(\Omega\times\s^1)}+\im{f}{\Omega\times\s^1}\\
&&+\frac{1}{\e^{1+\frac{1}{m}}}\nm{h}_{L^2(\Gamma^-)}+\frac{1}{\e^{\frac{1}{m}}}\nm{h}_{L^{m}(\Gamma^-)}+\im{h}{\Gamma^-}\bigg)+o(1)\im{u}{\Gamma^+}.\no
\end{eqnarray}
Absorbing $\im{u}{\Omega\times\s^1}$ into the left-hand side, we obtain
\begin{eqnarray}
\im{u}{\Omega\times\s^1}
&\leq& C\bigg(\frac{1}{\e^{1+\frac{1}{m}}}\tm{f}{\Omega\times\s^1}+
\frac{1}{\e^{2+\frac{1}{m}}}\nm{f}_{L^{\frac{2m}{2m-1}}(\Omega\times\s^1)}+\im{f}{\Omega\times\s^1}\\
&&+\frac{1}{\e^{1+\frac{1}{m}}}\nm{h}_{L^2(\Gamma^-)}+\frac{1}{\e^{\frac{1}{m}}}\nm{h}_{L^{m}(\Gamma^-)}+\im{h}{\Gamma^-}\bigg).\no
\end{eqnarray}

\end{proof}

\section{Diffusive Limit}

\begin{theorem}\label{diffusive limit}
Assume $g(\vx_0,\vw)\in C^2(\Gamma^-)$ satisfying (\ref{compatibility}). Then for the steady neutron
transport equation (\ref{transport}), there exists a unique solution
$u^{\e}(\vx,\vw)\in L^{\infty}(\Omega\times\s^1)$ satisfying (\ref{normalization}). Moreover, for any $0<\d<<1$, the solution obeys the estimate
\begin{eqnarray}
\im{u^{\e}-\u_0}{\Omega\times\s^1}\leq C(\d,\Omega)\e^{1-\d},
\end{eqnarray}
where $\u_0$ is defined in (\ref{expansion temp 11}).
\end{theorem}
\begin{proof}
We can divide the proof into several steps:\\
\ \\
Step 1: Remainder definitions.\\
We define the remainder as
\begin{eqnarray}\label{pf 1_}
R&=&u^{\e}-\sum_{k=0}^{2}\e^k\u_k-\sum_{k=0}^{1}\e^k\ub_k=u^{\e}-\q-\qb,
\end{eqnarray}
where
\begin{eqnarray}
\q&=&\u_0+\e\u_1+\e^2\u_2,\\
\qb&=&\ub_0+\e\ub_1.
\end{eqnarray}
Noting the equation (\ref{transport temp}) is equivalent to the
equation (\ref{transport}), we write $\ll$ to denote the neutron
transport operator as follows:
\begin{eqnarray}
\ll[u]&=&\e\vw\cdot\nx u+ u-\bar u\\
&=&\sin\phi\frac{\p
u}{\p\eta}-\frac{\e}{R_{\kappa}-\e\eta}\cos\phi\bigg(\frac{\p
u}{\p\phi}+\frac{\p u}{\p\tau}\bigg)+ u-\bar u\nonumber
\end{eqnarray}
\ \\
Step 2: Estimates of $\ll[\q]$.\\
The interior contribution can be estimated as
\begin{eqnarray}
\ll[\q]=\e\vw\cdot\nx \q+ \q-\bar
\q&=&\e^{3}\vw\cdot\nx \u_2.
\end{eqnarray}
We have
\begin{eqnarray}
\im{\ll[\q]}{\Omega\times\s^1}&\leq&\im{\e^{3}\vw\cdot\nx \u_2}{\Omega\times\s^1}\leq C\e^{3}\im{\nx\u_2}{\Omega\times\s^1}\leq
C\e^{3}.
\end{eqnarray}
This implies
\begin{eqnarray}\label{pf 2_}
\tm{\ll[\q]}{\Omega\times\s^1}&\leq& C\e^{3},\\
\nm{\ll[\q]}_{L^{\frac{2m}{2m-1}}(\Omega\times\s^1)}&\leq& C\e^{3},\\
\im{\ll[\q]}{\Omega\times\s^1}&\leq& C\e^{3}.
\end{eqnarray}
\ \\
Step 3: Estimates of $\ll \qb$.\\
Since $\ub_0=0$, we only need to estimate $\ub_1=(f_1^{\e}-f^{\e}_{1,L})\cdot\psi_0=\v\psi_0$ where
$f_1^{\e}(\eta,\tau,\phi)$ solves the $\e$-Milne problem and $\v=f_1^{\e}-f^{\e}_{1,L}$. The boundary layer contribution can be
estimated as
\begin{eqnarray}\label{remainder temp 1}
\ll[\e\ub_1]&=&\sin\phi\frac{\p
(\e\ub_1)}{\p\eta}-\frac{\e}{R_{\kappa}-\e\eta}\cos\phi\bigg(\frac{\p
(\e\ub_1)}{\p\phi}+\frac{\p(\e\ub_1)}{\p\tau}\bigg)+ (\e\ub_1)-
(\e\bar\ub_1)\\
&=&\e\Bigg(\sin\phi\bigg(\psi_0\frac{\p
\v}{\p\eta}+\v\frac{\p\psi_0}{\p\eta}\bigg)-\frac{\psi_0\e}{R_{\kappa}-\e\eta}\cos\phi\bigg(\frac{\p
\v}{\p\phi}+\frac{\p \v}{\p\tau}\bigg)+ \psi_0\v-\psi_0\bar\v\Bigg)\nonumber\\
&=&\e\psi_0\bigg(\sin\phi\frac{\p
\v}{\p\eta}-\frac{\e}{R_{\kappa}-\e\eta}\cos\phi\frac{\p
\v}{\p\phi}+\v-\bar\v\bigg)+\e\Bigg(\sin\phi
\frac{\p\psi_0}{\p\eta}\v-\frac{\psi_0\e}{R_{\kappa}-\e\eta}\cos\phi\frac{\p
\v}{\p\tau}\Bigg)\nonumber\\
&=&\e\Bigg(\sin\phi
\frac{\p\psi_0}{\p\eta}\v-\frac{\psi_0\e}{R_{\kappa}-\e\eta}\cos\phi\frac{\p
\v}{\p\tau}+\e^2\psi_0\v\Bigg)\nonumber.
\end{eqnarray}
Since $\psi_0=1$ when $\eta\leq R_{\min}/(4\e^{1/2})$, the effective region
of $\px\psi_0$ is $\eta\geq R_{\min}/(4\e^{1/2})$ which is further and further
from the origin as $\e\rt0$. By Theorem \ref{Milne theorem 2}, the
first term in (\ref{remainder temp 1}) can be bounded as
\begin{eqnarray}
\im{\e\sin\phi\frac{\p\psi_0}{\p\eta}\v}{\Omega\times\s^1}&\leq&
Ce^{-\frac{K_0}{\e^{1/2}}}\leq C\e^3.
\end{eqnarray}
Then we turn to the crucial estimate in the second term of (\ref{remainder temp 1}), by Theorem \ref{Milne tangential.}, we have
\begin{eqnarray}
\im{-\e\frac{\psi_0\e}{R_{\kappa}-\e\eta}\cos\phi\frac{\p
\v}{\p\tau}}{\Omega\times\s^1}&\leq&C\e^2\im{\frac{\p \v}{\p\tau}}{\Omega\times\s^1}\leq C\e^{2}\abs{\ln(\e)}^8.
\end{eqnarray}
Also, the exponential decay of $\dfrac{\p\v}{\p\tau}$ by Theorem \ref{Milne tangential.} and the rescaling $\eta=\mu/\e$ implies
\begin{eqnarray}
\tm{-\e\frac{\psi_0\e}{R_{\kappa}-\e\eta}\cos\phi\frac{\p
\v}{\p\tau}}{\Omega\times\s^1}&\leq& \e^2\tm{\frac{\p
\v}{\p\tau}}{\Omega\times\s^1}\\
&\leq&\e^2\Bigg(\int_{-\pi}^{\pi}\int_0^1(1-\mu)\lnm{\frac{\p\v}{\p\tau}(\mu,\tau)}^2\ud{\mu}\ud{\tau}\Bigg)^{1/2}\no\\
&\leq&\e^{\frac{5}{2}}\Bigg(\int_{-\pi}^{\pi}\int_0^{1/\e}(1-\e\eta)\lnm{\frac{\p\v}{\p\tau}(\eta,\tau)}^2\ud{\eta}\ud{\tau}\Bigg)^{1/2}\no\\
&\leq&C\e^{\frac{5}{2}}\Bigg(\int_{-\pi}^{\pi}\int_0^{1/\e}\ue^{-2K_0\eta}\abs{\ln(\e)}^{16}\ud{\eta}\ud{\tau}\Bigg)^{1/2}\no\\
&\leq& C\e^{\frac{5}{2}}\abs{\ln(\e)}^8.\no
\end{eqnarray}
Similarly, we have
\begin{eqnarray}
\nm{-\e\frac{\psi_0\e}{R_{\kappa}-\e\eta}\cos\phi\frac{\p
\v}{\p\tau}}_{L^{\frac{2m}{2m-1}}(\Omega\times\s^1)}&\leq&C\e^{3-\frac{1}{2m}}\abs{\ln(\e)}^8.
\end{eqnarray}
For the third term in (\ref{remainder temp 1}), we have
\begin{eqnarray}
\im{\e^3\psi_0\v}{\Omega\times\s^1}\leq C\e^3.
\end{eqnarray}
In total, we have
\begin{eqnarray}
\tm{\ll[\qb]}{\Omega\times\s^1}&\leq& C\e^{\frac{5}{2}}\abs{\ln(\e)}^8,\\
\nm{\ll[\qb]}_{L^{\frac{2m}{2m-1}}(\Omega\times\s^1)}&\leq& C\e^{3-\frac{1}{2m}}\abs{\ln(\e)}^8,\\
\im{\ll[\qb]}{\Omega\times\s^1}&\leq& C\e^{2}\abs{\ln(\e)}^8.
\end{eqnarray}
\ \\
Step 4: Diffusive Limit.\\
In summary, since $\ll[u^{\e}]=0$, collecting estimates in Step 2 and Step 3, we can prove
\begin{eqnarray}
\tm{\ll[R]}{\Omega\times\s^1}&\leq& C\e^{\frac{5}{2}}\abs{\ln(\e)}^8,\\
\nm{\ll[R]}_{L^{\frac{2m}{2m-1}}(\Omega\times\s^1)}&\leq& C\e^{3-\frac{1}{2m}}\abs{\ln(\e)}^8,\\
\im{\ll[R]}{\Omega\times\s^1}&\leq& C\e^{2}\abs{\ln(\e)}^8.
\end{eqnarray}
Also, based on our construction, it is easy to see
\begin{eqnarray}
R-\pp[R]=-\e^2(\vw\cdot\nx\u_1-\pp[\vw\cdot\nx\u_1]),
\end{eqnarray}
which further implies
\begin{eqnarray}
\tm{R-\pp[R]}{\Gamma^-}&\leq& C\e^2,\\
\nm{R-\pp[R]}_{L^{m}(\Gamma^-)}&\leq&C\e^2,\\
\im{R-\pp[R]}{\Gamma^-}&\leq& C\e^2
\end{eqnarray}
Also, the remainder $R$ satisfies the equation
\begin{eqnarray}
\left\{
\begin{array}{rcl}
\e \vw\cdot\nabla_x R+R-\bar R&=&\ll[R]\ \ \text{for}\ \ \vx\in\Omega,\\
R-\pp[R]&=&R-\pp[R]\ \ \text{for}\ \ \vw\cdot\vn<0\ \ \text{and}\ \
\vx_0\in\p\Omega.
\end{array}
\right.
\end{eqnarray}
It is easy to verify $R$ satisfies the normalization condition and compatibility condition (\ref{normalization.}) and (\ref{compatibility.}). By Theorem \ref{LI estimate}, we have for $m$ sufficiently large
\begin{eqnarray}
\im{R}{\Omega\times\s^1}
&\leq& C\bigg(\frac{1}{\e^{1+\frac{1}{m}}}\tm{\ll[R]}{\Omega\times\s^1}+
\frac{1}{\e^{2+\frac{1}{m}}}\nm{\ll[R]}_{L^{\frac{2m}{2m-1}}(\Omega\times\s^1)}+\im{\ll[R]}{\Omega\times\s^1}\\
&&+\frac{1}{\e^{1+\frac{1}{m}}}\nm{R-\pp[R]}_{L^2(\Gamma^-)}+\frac{1}{\e^{\frac{1}{m}}}\nm{R-\pp[R]}_{L^{m}(\Gamma^-)}+\im{R-\pp[R]}{\Gamma^-}\bigg)\no,\\
&\leq& C\Bigg(\frac{1}{\e^{1+\frac{1}{m}}}\bigg(\e^{\frac{5}{2}}\abs{\ln(\e)}^8\bigg)+
\frac{1}{\e^{2+\frac{1}{m}}}\bigg(\e^{3-\frac{1}{2m}}\abs{\ln(\e)}^8\bigg)+\bigg(\e^{2}\abs{\ln(\e)}^8\bigg)\no\\
&&+\frac{1}{\e^{1+\frac{1}{m}}}(\e^2)+\frac{1}{\e^{\frac{1}{m}}}(\e^2)+(\e^2)\Bigg)\no\\
&\leq&C\e^{1-\frac{3}{2m}}\abs{\ln(\e)}^8\leq C\e^{1-\d}
\end{eqnarray}
Note that the constant $C$ might depend on $m$ and thus depend on $\d$.
Since it is easy to see
\begin{eqnarray}
\im{\sum_{k=1}^{2}\e^k\u_k+\sum_{k=0}^{1}\e^k\ub_k}{\Omega\times\s^1}\leq C\e,
\end{eqnarray}
our result naturally follows. This completes the proof of main theorem.
\end{proof}


\bibliographystyle{siam}
\bibliography{Reference}

\end{document}